\pdfoutput=1
\documentclass[11pt,reqno,tbtags]{amsart}
\usepackage[]{amsmath,amssymb,amsfonts,latexsym,amsthm,enumerate}
\usepackage{amssymb,bbm,mathrsfs}
\usepackage{enumerate,graphicx,paralist}
\usepackage[numeric,initials,nobysame]{amsrefs}
\usepackage[margin=1in]{geometry}
\usepackage[margin=11pt
]{caption}

\date{}

\numberwithin{equation}{section}
\usepackage[usenames]{color}
\newtheorem{maintheorem}{Theorem}

\newtheorem*{conjecture*}{Conjecture}
\newtheorem{theorem}{Theorem}[section]
\newtheorem*{theorem*}{Theorem}
\newtheorem{lemma}[theorem]{Lemma}

\newtheorem{corollary}[theorem]{Corollary}

\newtheorem{definition}[theorem]{Definition}
\theoremstyle{definition}{
\newtheorem{remark}[theorem]{Remark}
\newtheorem*{remark*}{Remark}
\newtheorem{example}[theorem]{Example}
\newtheorem*{example*}{Example}
\newtheorem{observation}[theorem]{Observation}
}

\newcommand{\ignore}[1]{}
\newcommand{\one}{\mathbbm{1}}
\renewcommand{\P}{\mathbb P}
\newcommand{\E}{\mathbb E}
\renewcommand{\epsilon}{\varepsilon}
\renewcommand{\a}{\alpha}

\newcommand{\D}{\Delta}

\renewcommand{\L}{\Lambda}

\newcommand{\cG}{\mathcal{G}}

\newcommand{\R}{\mathbb R}
\newcommand{\ha}{{\hat{\alpha}}}

\newcommand{\sL}{\mathscr{L}}

\newcommand{\edg}{\texttt{e}}
\newcommand{\Sc}{\mathcal{S}}
\newcommand{\cF}{\mathcal{F}}
\newcommand{\cE}{\mathcal{E}}
\newcommand{\cN}{\mathcal{N}}
\newcommand{\cA}{\mathcal{A}}
\newcommand{\cB}{\mathcal{B}}
\newcommand{\sX}{\mathscr{X}}
\newcommand{\sY}{\mathscr{Y}}
\newcommand{\Bin}{\operatorname{Bin}}

\begin{document}
\title{Random triangle removal}


\author{Tom Bohman}
\address{Tom Bohman\hfill\break
Department of Mathematical Sciences\\
Carnegie Mellon University\\
Pittsburgh, PA 15213, USA.}
\email{tbohman@math.cmu.edu}
\urladdr{}

\author{Alan Frieze}
\address{Alan Frieze\hfill\break
Department of Mathematical Sciences\\
Carnegie Mellon University\\
Pittsburgh, PA 15213, USA.}
\email{alan@random.math.cmu.edu}
\thanks{T.\ B.\ is supported in part by NSF grant DMS-1001638. A.\ F.\ is supported in part by NSF grant DMS-0721878.}
\urladdr{}

\author{Eyal Lubetzky}
\address{Eyal Lubetzky\hfill\break
Theory Group of Microsoft Research\\
One Microsoft Way\\
Redmond, WA 98052, USA.}
\email{eyal@microsoft.com}
\urladdr{}

\begin{abstract}
Starting from a complete graph on $n$ vertices, repeatedly delete the edges of a uniformly chosen triangle. This stochastic process terminates once it arrives at a triangle-free graph, and the fundamental question is to estimate the final number of edges (equivalently, the time it takes the process to finish, or how many edge-disjoint triangles are packed via the random greedy algorithm).
Bollob\'as and Erd\H{o}s~(1990) conjectured that the expected final number of edges has order $n^{3/2}$,
 motivated by the study of the Ramsey number $R(3,t)$.
 An upper bound of $o(n^2)$ was shown by Spencer~(1995) and independently by R\"odl and Thoma~(1996). Several bounds were given for variants and generalizations (e.g., Alon, Kim and Spencer~(1997) and Wormald~(1999)), while the best known upper bound for the original question of Bollob\'as and Erd\H{o}s was $n^{7/4+o(1)}$ due to Grable~(1997). No nontrivial lower bound was available.

Here we prove that with high probability the final number of edges in random triangle removal is equal to $n^{3/2+o(1)}$, thus confirming the $3/2$ exponent conjectured by Bollob\'as and Erd\H{o}s and matching the predictions of Spencer \emph{et al}. For the upper bound, for any fixed $\epsilon>0$ we construct a family of $\exp(O(1/\epsilon))$ graphs by gluing $O(1/\epsilon)$ triangles sequentially in a prescribed manner, and dynamically track all homomorphisms from them, rooted at any two vertices, up to the point where $n^{3/2+\epsilon}$ edges remain. A system of martingales establishes concentration for these random variables around their analogous means in a random graph with corresponding edge density, and a key role is played by the self-correcting nature of the process. The lower bound builds on the estimates at that very point to show that the process will typically terminate with at least $n^{3/2-o(1)}$ edges left.
\end{abstract}
\maketitle
\vspace{-0.75cm}

\section{Introduction}

Consider the following well-known stochastic process for generating a triangle-free graph, and at the same time creating a partial Steiner triple system. Start from a
complete graph on $n$ vertices and proceed to repeatedly remove the edges of uniformly
chosen triangles. That is, letting $G(0)$ denote the initial graph, $G(i+1)$ is obtained from $G(i)$ by
selecting a triangle uniformly at random out of all triangles in $G(i)$ and deleting its 3 edges.
The process terminates once no triangles remain, and the fundamental question is to estimate the stopping time
\[ \tau_0 = \min\{ i : G(i)\mbox{ is triangle-free}\}\,.\]
This is equivalent to estimating the number of edges in the final triangle-free graph, since $G(i)$ has precisely $\binom{n}2 - 3i$ edges by definition. As the triangles removed are mutually edge-disjoint, this process is precisely the random greedy algorithm for triangle packing.

Bollob\'as and Erd\H{o}s (1990) conjectured that the expected number of edges in $G(\tau_0)$ has order $n^{3/2}$ (see, e.g.,~\cites{Bol1,Bol2}),
with the motivation of determining the Ramsey number $R(3,t)$. Behind this conjecture was the intuition that the graph $G(i)$ should be similar to an Erd\H{o}s-R\'enyi random graph with the same edge density.
Indeed, in the latter random graph with $n$ vertices and $\epsilon n^{3/2}$ edges there are typically about $\frac43 \epsilon^3 n^{3/2}$ triangles, thus, for small $\epsilon$, deleting all of its triangles one by one would still retain all but a negligible fraction of the edges.

It was shown by Spencer~\cite{Spencer} in 1995, and independently by R\"odl and Thoma~\cite{RT} in 1996,
that the final number of edges is $o(n^2)$ with high probability (w.h.p.).
In 1997, Grable~\cite{Grable} improved this
to an upper bound of $n^{11/6+o(1)}$ w.h.p., and further described
how similar arguments, using some more delicate calculations, should extend that result to $n^{7/4+o(1)}$.
This remained the best upper bound prior to this work. No
nontrivial lower bound was available.
(See~\cite{GKPS} for numerical simulations firmly supporting an answer of $n^{3/2+o(1)}$ to this problem.)

Of the various works studying generalizations and variants of the problem, we mention two here.
In his paper from 1999, Wormald~\cite{Wormald} demonstrated how the differential equation method
can yield a nontrivial upper bound on greedy packing of hyperedges in $k$-uniform hypergraphs.
For the special case $k=3$, corresponding to triangle packing, this translated to a bound of $n^{2-\frac{1}{57}+o(1)}$.
Also in the context of hypergraphs, Alon, Kim and Spencer~\cite{AKS} introduced in 1997 a semi-random variant of the aforementioned process
(akin to the R\"odl nibble~\cite{Rodl} yet with some key differences) which, they showed, finds nearly perfect matchings.
Specialized to our setting, that process would result in a collection of edge-disjoint triangles on $n$ vertices that covers all but $ O( n^{3/2} \log^{3/2} n)$ of the edges
of the complete graph. Alon \emph{et al}~\cite{AKS} then conjectured that the simple random greedy algorithm should match those results, and in particular
--- generalizing the Bollob\'as-Erd\H{o}s conjecture --- that
applying it to find a maximal collection of $k$-tuples with pairwise intersections at most $k-2$ would leave out an expected number of at most
$n^{k-1-\frac{1}{k-1}+o(1)}$ uncovered $(k-1)$-tuples. They added that ``at
the moment we cannot prove that this is the case even for $k = 3$'', the focus of our work here.

Joel Spencer offered \$200 for a proof that the answer to the problem is $n^{3/2+o(1)}$ w.h.p.\ (\cites{Grable,Wormald}).
The main result in this work establishes this precise statement, thus confirming the exponent conjectured by Bollob\'as and Erd\H{o}s (1990).

\begin{maintheorem}
\label{mainthm-1}
Let $\tau_0$ be the number of steps it takes the random triangle removal process to terminate starting from a complete graph on $n$ vertices,
and let $E(\tau_0)$ be the edge set of the final triangle-free graph. Then with high probability
$\tau_0 = n^2/6 - n^{3/2+o(1)}$, or equivalently, $|E(\tau_0)| = n^{3/2+o(1)}$.
\end{maintheorem}

We prove Theorem~\ref{mainthm-1} by showing that w.h.p.\ all variables in a
collection of $e^{O(1/\epsilon)} n^2$ random variables, carefully designed to support the analysis,
 stay close to their respective trajectories throughout the evolution of this stochastic hypergraph process (see~\S\ref{sec:methods} for details).
Establishing concentration for this large collection of variables hinges on their self-correction nature: the 
further a variable deviates from its trajectory, the stronger its drift is back toward its mean.
However, turning this into a rigorous proof is quite challenging given that the various 
variables interact and errors (deviations from the mean) propagating from 
other variables may interfere in the attempt of one variable to correct itself. 
We construct a system of martingales to guarantee that the drift of a variable towards its mean 
will dominate the errors in our estimates for its peers.

The tools developed here for proving such self-correcting estimates are generic and we believe they
will find applications in other settings. In particular, these methods should support
an analysis of the random greedy hypergraph matching process as well as that of the asymptotic 
final number of edges in the graph produced by the triangle-free process (see~\S\ref{sec:triangle-free}). Progress on the latter process would
likely yield improvement on the best known lower bound on the Ramsey number $R(3,t)$.

\subsection{Methods}\label{sec:methods}
Our starting point for the upper bound is a system of martingales for
dynamically tracking an ensemble of random variables consisting of the triangle count and all
co-degrees in the graph. The self-correcting behavior of these variables is roughly seen as follows: should the co-degree
of a given pair of vertices deviate above/below its average, then more/fewer than average triangles could shrink
the co-neighborhood if selected for removal in the next round, and this compensation effect would eventually drive the variable back towards its mean.
One can exploit this effect to maintain the concentration of all these variables around their analogous means in a corresponding Erd\H{os}-R\'enyi random graph, as long as the number of
edges remains above $n^{7/4+o(1)}$. (In the short note~\cite{BFL} the authors applied this argument to match this upper bound due to Grable.)

It is no coincidence that various methods break precisely at the exponent 7/4 as it corresponds to the inherent barrier where
co-degrees become comparable to the variations in their values that arose earlier in the process.
In order to carry out the analysis beyond this barrier, one can for instance enrich the ensemble of tracked variables to address the second level neighborhoods
of pairs of vertices. This way one can avoid cumulative worst-case individual errors due
to large summations of co-degree variables, en route to co-degree estimates which \emph{improve} as the process evolves.
Indeed, these ideas can push this framework to successfully track the process to the point where $n^{5/3+o(1)}$ edges remain.
However, beyond that point (corresponding to an edge density of about $n^{-1/3}$) the size of common neighborhoods of triples would become negligible, foiling the analysis.
This example shows the benefit of maintaining control over a large family of subgraphs, yet at the same time it demonstrates how for any family of \emph{bounded size} subgraphs the framework will eventually collapse.



\begin{figure}
\centering
\fbox{
\begin{tabular}{c}
\includegraphics[width=0.95\textwidth]{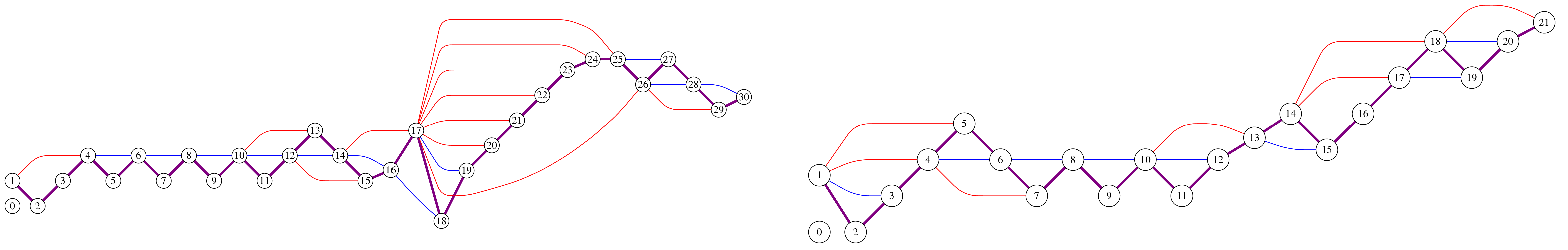} \\
\hline \\
\includegraphics[width=0.95\textwidth]{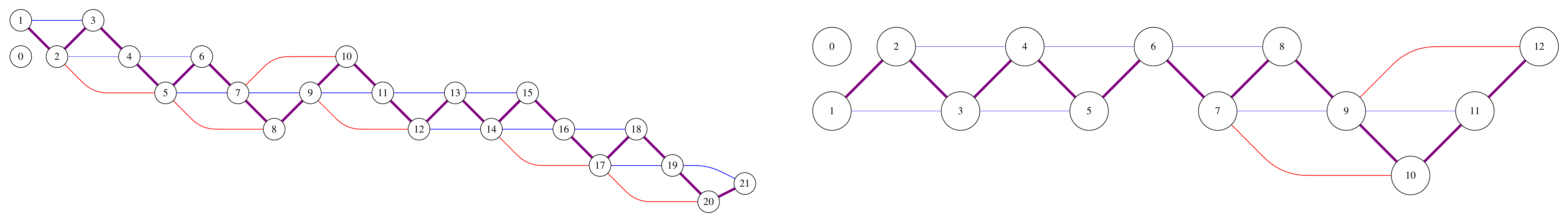}
\end{tabular}
}
\caption{A subset of the ensemble of about $2^{30}$ triangular ladders that are tracked for any pair of roots
to sustain the analysis of the upper bound to the point of $n^{3/2+\epsilon}$ edges for $\epsilon=\frac1{10}$.
}
\label{fig:M10}
\end{figure}

In order to prove the $n^{3/2+o(1)}$ upper bound we track an \emph{arbitrarily large} collection of random variables.
A natural prerequisite for understanding the evolution of the number of homomorphisms from any small subgraph $H$
to $G(i)$ rooted at a given pair of vertices $u,v$ is to have ample control over homomorphisms from subgraphs obtained by ``gluing a triangle'' to an edge to $H$ (as the number
of those dictates the probability of losing copies of $H$ in the next round on account of losing the mentioned edge). Tracking these would involve
additional larger subgraphs etc., and to end this cascade and obtain a system of variables which can self-sustain
itself we do the following. To estimate the probability of losing the ``last edge'' in the largest tracked $H$, in lieu of
counting homomorphisms from a ``forward extension'' (the result of gluing a new triangle onto the given edge)
rooted at $u,v$, we fix our candidate for the edge
in $G(i)$  that would play the role of this last edge and then count homomorphisms rooted at $u,v$ and featuring this specific edge. In other words,
we replace the requirement for control over ``forward extensions'' with control over ``backward extensions'' from any given edge in $G(i)$.
Obviously, rooting the subgraphs at as many as 4 vertices can drastically decrease the expected number of homomorphisms, say to $O(1)$, while our analysis must
involve error probabilities that tend to $0$ at a super-polynomial rate to account for our polynomial number of variables.

These points lead to a careful definition of the ensemble of graphs we wish to track. We construct what we refer to as \emph{triangular ladders}
which, roughly put, are formed by repeatedly taking an existing element in the ensemble and gluing
a triangle on one of its last added two edges. We do so until reaching a certain threshold on the number of vertices,
namely $O(1/\epsilon)$ for some fixed $\epsilon > 0$, amounting to a family of $2^{O(1/\epsilon)}$ graphs (see Figure~\ref{fig:M10}).
Special care needs to be taken so as to avoid certain substructures that would foil our analysis,
yet ignoring these for the moment (postponing precise definitions to \S\ref{sec:ladders}), note that each new step of adding a triangle (a new vertex and two new edges) also
increases the density of the ladder under consideration, as long as the edge density is at least $n^{-1/2+\epsilon}$.
When examining a given ladder in this ensemble, its forward extensions
will also be in the ensemble, by construction, unless we exceed the size threshold. In the latter case, if the threshold is taken to be suitably large,
we can safely appeal to a backward extension argument and wind up with polynomially many homomorphisms to $G(i)$.

The main challenge in this approach is to come up with a canonical definition of the ensemble such that
a uniform analysis could be applied to its arbitrarily many elements.
In particular, since as usual the errors in one variable propagate to others, we end up with a system of linear constraints on $2^{O(1/\epsilon)}$
variables. Fortunately, suitable canonical definitions of the main term in each variable as a function of its predecessor in the ensemble make
this constraint system one without cyclic dependencies, allowing for a simple solution to it. Altogether, the triangular ladder ensemble can be
maintained as long as there are $n^{3/2+\epsilon}$ edges, and while so it sustains the analysis of the number of triangles.

The lower bound builds on the fact that our analysis of the upper bound yielded asymptotic estimates for all co-degrees up to the point
of $n^{3/2+\epsilon}$ edges, for arbitrarily small $\epsilon>0$. Already a weaker bound on the co-degrees, when combined with an analysis
of the evolution of the number of triangles, suffices to show that either the final number of edges has order at least $n^{3/2}$
or at some point $i$ there are $|E(i)|\asymp n^{3/2}$ edges and $(\frac13-o(1))|E(i)|$ edge-disjoint triangles. These give rise to a subset of $n^{3/2-O(\epsilon)}$ edge-disjoint triangles
that existed earlier in the process, such that each one could guarantee an edge to survive in the final graph with probability $n^{-O(\epsilon)}$, and these events are mutually independent. Altogether
this leads to at least $n^{3/2-O(\epsilon)}$ edges in the final graph w.h.p.

\subsection{Comparison with the triangle-free process}\label{sec:triangle-free}

A different recipe for obtaining a random triangle-free graph is the so-called ``triangle-free process''. In that process, the $\binom{n}2$ edges
arrive one by one according to a uniformly chosen permutation, and each one is added if and only if no triangles are formed with it in the current graph. It was shown in~\cite{ESW} that the final number of edges in this process is w.h.p.\ $n^{3/2+o(1)}$, and the correct order of $n^{3/2}\sqrt{\log n}$, along with its ramifications on the Ramsey number $R(3,t)$, was recovered in~\cite{Bohman} (see also~\cite{BK} for generalizations).

Similarly to that process, the triangle removal process studied here can be presented as taking a uniform ordering of the $\binom{n}3$ triangles, then sequentially removing the edges of a triangle if and only if all 3 edges presently belong to the graph. The main result in this work shows that this process too culminates in a triangle-free graph with $n^{3/2+o(1)}$ edges.

Despite the high level similarity between the two protocols, the triangle removal process has proven quite challenging already
at the level of acquiring the correct exponent of the final number of edges. One easily seen difference is that, in the triangle-free process, at any given point there is a set of forbidden edges (whose addition would form at least one triangle), yet regardless of its structure the next edge to be added is uniformly distributed over the legal ones. As long as the forbidden set of edges is negligible compared to the remaining legal edges, this process therefore mostly adds uniform edges just as in the Erd\H{o}s-R\'enyi random graph. In particular, this is the case as long as $o(n^{3/2})$ edges were added (even based on all the edges that arrived rather than those selected). A coupling to the Erd\H{o}s-R\'enyi random graph up to that point immediately gives a lower bound of order $n^{3/2}$ on the expected number of edges at the end of the triangle-free process.
One could hope for an analogous soft argument to hold for the triangle removal process, where it would translate to an upper bound (as it deletes rather than adds edges).

However, in the triangle removal process studied here, deleting almost all edges save for $n^{3/2+o(1)}$ forms a substantial challenge for the analysis. Already when the edge density is a small constant, still bounded away from $0$, the set of forbidden triangles becomes much larger than the set of legal ones. Thus, in an attempt to analyze the process even up to a density of $n^{-\epsilon}$ for some $\epsilon>0$, one is forced to control the geometry of the remaining triangles quite delicately from its very beginning.



\section{Upper bound on the final number of edges modulo co-degree estimates}\label{sec:upper-bound}

Let $ (\cF_i) $ be the filtration given by the process.
Let $ N_u = N_u(i) = \{ x \in V_G : x u \in E(i) \}$ denote the neighborhood of a vertex $u\in V_G$, let $N_{u,v} = N_u \cap N_v$
be the co-neighborhood of vertices $u,v\in V_G$ and let $Y_{u,v} =|N_{u,v}|$ denote their co-degree.
Our goal is to estimate the number of triangles in $ G(i) $ which we denote by $Q(i)$.
The motivation behind tracking co-degrees is immediate from the fact that $ Q(i) = \frac13 \sum_{uv \in E(i)} Y_{u,v}$,
yet another important element is their effect on the evolution of $Q$. By definition,
\begin{align}
\E[\D Q \mid \cF_i] &= - \sum_{ uvw \in Q} \left( Y_{u,v} + Y_{u,w} + Y_{v,w} - 2\right)/Q\,, \label{eq-Q-one-step}
\end{align}
where here and in what follows $\D X = X(i+1)-X(i)$ is the one-step change in the variable $X$.

It is convenient to re-scale the number of steps $i$ and introduce a notion of edge density as follows:
\begin{align}
  \label{eq-p-def}
t=t(i)=\frac{i}{n^2}\,,\qquad & p = p(i,n) = 1 - \frac{6i}{n^2} = 1 - 6t\,.
\end{align}
With this notation we have
\begin{align}
  \label{eq-E-val}
  |E(i)| = \binom{n}2 - 3i = \binom{n}2 - \frac12 (1-p)n^2 = \frac12( n^2 p - n)\,,
\end{align}
and we see that, up to the negligible linear term, the number of edges corresponding to the edge density $p$
is in line with that of the Erd\H{o}s-R\'enyi random graph $\cG(n,p)$.
Under the assumption that $G(i)$ indeed resembles $\cG(n,p)$ we expect that the number of triangles and co-degrees would satisfy
\[ Y_{u,v} \approx  n p^2 \qquad \mbox{ and }\qquad Q \approx \frac16 n^3 p^3\,. \]
Turning to the evolution of the co-degrees, we see that
\begin{align}
\E[ \D Y_{u,v} \mid \cF_i ] & = - \sum_{x \in N_{u,v}} \frac{ Y_{u,x} + Y_{v,x} - \one_{\{uv \in E\}}}{ Q} \,,\label{eq-Y-one-step}
\end{align}
and so tight estimates on the co-degrees and $Q$ can already support the analysis of the process.
In the short note~\cite{BFL} the authors relied on these variables alone to establish a bound
of $O(n^{7/4}\log^{5/4}n)$ on the number of edges that survive to the conclusion of the
algorithm. However, once the edge density drops to $p=n^{-1/4}$ we arrive at $Y_{u,v}\approx n^{1/2}$.
As the variables $Y_{u,v}$ have order $n$ and standard deviations of order $\sqrt{n}$
while $p$ is constant, we see that at
the point $p=n^{-1/4}$ the deviations that developed early
in the process cease to be regarded as negligible and a refined analysis is required.

One can achieve better precision for the co-degrees by introducing additional variables
aimed at \emph{decreasing} the variation for $Y_{u,v}$ as the process evolves. For instance,
through estimates on the number of edges between the common and exclusive neighborhoods for every pair of vertices,
along with some additional ingredients exploiting the self-correction phenomenon,
one can gain an improvement on the $7/4$ exponent. However, eventually these arguments break down, soon after
some bounded-size subgraphs becomes too rare, thus foiling the concentration estimates.

We remedy this via the triangular ladder ensemble, defined in \S\ref{sec:ladders} and analyzed in \S\ref{sec:tracking-ladders}.
To sustain the analysis up to the point where $n^{3/2+\epsilon}$ edges remain we must simultaneously control $\exp(O(1/\epsilon))$ types of random variables, yet our ultimate goal in
those sections is simply to estimate all co-degrees $\{ Y_{uv} : u,v\in V_G\}$. The following corollary is a special case of Theorem~\ref{thm:ladders-concentrate}.
\begin{theorem}\label{thm:aux}
Define $\zeta=\zeta(p,n)$ to be
\begin{equation}
 \label{eq-zeta-def}
\zeta  = n^{-1/2} p^{-1} \log n
\end{equation}
and for some (arbitrarily large) absolute constant $\kappa>0$ let
\[\tau_{\textsc{q}}^* = \min\left\{ t : \left| Q/(\tfrac16 n^3 p^3)-1\right| \geq \kappa\, \zeta^2 \right\}\,.\]
Then for every $M\geq 3$ w.h.p.\
\begin{align*}
\left|Y_{u,v}/(np^2) - 1 \right| &\leq 3^{3M-1}\,\zeta
\end{align*}
for every $u,v\in V_G$ and all $t$ such that $t \leq \tau_{\textsc{q}}^*$ and $p(t)\geq n^{-1/2+1/M}$.
\end{theorem}

The heart of this work is in establishing the above theorem (Sections~\ref{sec:extension}--\ref{sec:tracking-ladders}), which is complemented by the following result: A multiplicative error of $1+O(\zeta)$ for all co-degree variables $Y_{u,v}$ can be enhanced to
bounds tight up to a factor of $1+O(\zeta^2)$ on the number of triangles $Q$.
\begin{theorem}
\label{thm:Q}
Set $ \Phi(p,n) = p^{-2/\log n} \log n$ and for some fixed $\alpha>0$ let
\begin{align}
\tau^*_{\textsc{y}} = \min\left\{ t \;:\; \exists u,v\mbox{ such that }\left| Y_{u,v} - n p^2 \right| > \alpha n^{1/2} p \,\Phi  \right\} \label{Yuvbound}\,.
\end{align}
Then w.h.p.\ as long as $t\leq \tau^*_{\textsc{y}}$ and $p(t) \geq n^{-1/2} \log^2 n$ we have
\begin{align}
\big| Q - n^3p^3/6 \big| &\leq \alpha^2 n^2 p \,\Phi^2 \label{Qbound}\,.
\end{align}
\end{theorem}
\begin{proof}
Let
\[ X(i) = Q - \frac{n^3p^3}{6}\quad,\quad Z(i) = |X(i)| - \a^2 n^2 p\Phi^2 \,.\]
We define below a narrow {\em critical interval} $I_Q$ that has its upper
endpoint at the bound we aim to establish for $|X|$.  As long as $|X|$ lies in
this interval it is subject to a self-correcting drift, which 
turns out to be just enough to show that $ \E[ \Delta Z \mid \cF_i ] < 0$.  
That is, as long as $|X|\in I_Q$ the sequence $ Z(i)$ is a
supermartingale. Standard concentration estimates then imply that $|X|$ is unlikely
to ever cross this critical interval, leading to the desired estimate on $Q$. 

Define
\[ \Lambda = \log^{-3/2}n\]
and observe that for all $p\geq n^{-1/2}\log^2 n$
we have $n^{-1} p^{-2} \Phi^2 = O(1/\sqrt{\log n})$, whence as long as~\eqref{Qbound} is valid we have
$Q=(\frac16+O(\Lambda))n^3p^3$. This clearly holds at the beginning of the process, and in what follows
we may condition on this event as part of showing that the bound~\eqref{Qbound} will be maintained 
except with an error probability that tends to 0 at a super-polynomial rate.

We begin with estimates on the one-step expected changes of our variables. 
Recall that due to~\eqref{eq-Q-one-step} we have
\[ \E[\D Q \mid \cF_i] = - \sum_{ uvw \in Q} \frac{ Y_{u,v} + Y_{u,w} + Y_{v,w} - 2}{Q}
= 2 - \frac{1}{Q} \sum_{ u v \in E} Y_{u,v}^2\,.\]
To bound $\E[\D Q \mid \cF_i]$ we thus need an estimate on $\sum_{ u v \in E} Y_{u,v}^2$,
achieved by the next simple lemma.
\begin{lemma}\label{lem-ai^2}
Let $a_1,\ldots,a_m \in \R$ and suppose that $|a-a_i|\leq \delta$ for all $i$ and some $a\in\R$. Then
\[
\frac{\left(\sum_{i} a_i\right)^2}m \leq \sum_{i} a_i^2 \leq \frac{\left(\sum_{i} a_i\right)^2}{m} + 4m\delta^2\,.
\]
\end{lemma}
\begin{proof}
The lower bound is due to Cauchy-Schwarz. For the upper bound
fix $b=\frac1m \sum_i a_i$ and note that the assumption on the $a_i$'s implies that $|b-a_i|\leq 2\delta$ for all $i$.
Observe that the convex function $\sum_{i} a_i^2$
achieves its maximum over the convex set $\left\{\sum_{i} a_i= bm\right\}\cap\big(\bigcap_{i}\left\{|b-a_i|\leq 2\delta\right\}\big)$ at an extremal point where
$a_i=b\pm 2\delta$ for all $i$. Hence, $a_i=b-2\delta$ for $\lfloor m/2\rfloor$ indices, $a_i=b+2\delta$ for $\lfloor m/2\rfloor$ indices and
if $m$ is odd there is a single $a_i=b$. In particular
\begin{equation*}
\sum_{i=1}^ma_i^2\leq \lfloor m/2\rfloor (b-2\delta)^2+\lfloor m/2\rfloor (b+2\delta)^2 +b^2 \one_{\{m\,\equiv\, 1\,(\mathrm{mod}\, 2)\}}\leq mb^2+4m\delta^2\,.
  \qedhere
\end{equation*}
\end{proof}
Since $\sum_{u v\in E} Y_{u,v} = 3Q$ and  $| Y_{u,v} - np^2 | < \alpha n^{1/2}p \Phi$ by Eq.~\eqref{Yuvbound},
an application of this lemma together with the fact that $|E(i)| = n^2 p/2 - n/2$ gives that
\begin{align*}
\frac1Q \sum_{uv\in E} Y_{u,v}^2 &\geq \frac{9Q}{|E|} \geq \frac{18 Q}{n^2p}\,,
\end{align*}
together with the following upper bound:
\begin{align*}
\frac1Q \sum_{uv\in E} Y_{u,v}^2 &\leq \frac{9Q}{|E|} + 4\frac{|E|}{Q} \a^2 n p^{2}\Phi^2 \leq \frac{18 Q}{n^2p\big(1-\frac1{np}\big)}+
\frac{2n^2 p}{\big(\frac16+O(\L)\big)n^3p^3} \a^2 n p^{2}\Phi^2 \\
&\leq  \frac{18 Q}{n^2p}+ O(p)+ (12+O(\L))\a^2 \Phi^2
= \frac{18 Q}{n^2p}+ (12+O(\L))\a^2 \Phi^2
\,.
 \end{align*}
(In the last equality we absorbed the $O(p)$-term into the $O(\L)$ error-term factor of the last expression
 since $\L \Phi^2\to\infty$.) Adding this to our estimate for $\E[\D Q \mid \cF_i]$
 (where again the additive 2 may be absorbed in our error term) while observing that $\D (-\frac16 n^3p^3) = 3np^2 + O(p/n)$ yields
 \[ \bigg| \E[\D X \mid \cF_i] - \frac{18}{n^2 p} \big( - Q + \tfrac16 n^3p^3\big) \bigg| \leq  (12+O(\L))\a^2 \Phi^2 + O(p/n) = (12+O(\L))\a^2 \Phi^2\,,\]
again incorporating the $O(p/n)$ term into the $O(\L)$ error. Thus, by the definition of $X$
 \[ \E\left[\D |X| \;\big|\; \cF_i\right] \leq -\frac{18 }{n^2p}|X| + (12+O(\L))\a^2 \Phi^2 \,.\]
Now assume that $i_0$ is the first round where $|X|$ raises above $ \ha^2 n^{2} p \Phi^2$, i.e., it enters the interval
\[
I_Q = \left( \ha^2 n^{2} p \Phi^2~,~\a^2 n^{2} p \Phi^2 \right) \qquad \mbox{ where } \qquad \ha = \left( 1- \log^{-1} n \right)^{1/2} \a\,.
\]
Further let 
\[\tau = \min\{i>i_0: |X(i)| < \ha^2 n^{2} p \Phi^2 \}\,.\]
Since $\D\left(-\a^2 n^{2} p \Phi^2\right) = \big(6+O(\frac{1}{n^2 p})\big) (1-\frac{4}{\log n}) \a^2 \Phi^2 $
 and the error-term $O(\frac{1}{n^2p})$ can easily be increased to $O(1/\Lambda)$ as $p\geq 1/n$ (with much room to spare),
 the upper bound on $\D|X|$ gives
\[\E\left[ \D Z\mid \cF_i\,,\,\tau>i\right] \leq
\big[-3(\ha/\a)^2 +2 + (1-\tfrac{4}{\log n}) + O(\L)\big]\,6\a^2 \Phi^2  \leq -\frac{6-o(1)}{\log n}\a^2 \Phi^2\,.
\]
It follows that $(Z(j \wedge \tau))_{j>i_0}$ is indeed a supermartingale for large $n$.

Next consider the one-step variation of $Z$. Denoting the selected triangle in a given round by $uvw$, the change in $Q$ following this round
is at most $Y_{u,v}+Y_{u,w}+Y_{v,w}$ and in light of our co-degree estimate~\eqref{Yuvbound} this expression deviates from its expected value
of $3np^2$ by at most $3\a n^{1/2}p \Phi$.
In particular, $|\D Z| = O\big( \sqrt{n} p_0 \log n \big)$ and letting $p_0 = p(i_0) = 1 - 6i_0/n^2$ this ensures that
\[ Z(i_0) \leq (\ha^2-\a^2) n^{2} p_0 \Phi^2  + O\big( \sqrt{n} p_0 \log n \big) \leq -\tfrac12 \a^2 n^{2} p_0 \log n \,,\]
where the last inequality holds for large $n$. With at most $n^2 p_0$ steps remaining until the process terminates, Hoeffding's inequality establishes that for some fixed $c>0$,
\[ \P\left(\cup_{j\geq i_0} \{Z(j\wedge\tau) \geq 0\}\right) \leq \exp
\left( - c \frac{  (n^2  p_0 \log n)^2}{ n^2 p_0 (\sqrt{n} p_0 \log n)^2} \right)
= e^{- c n / p_0} \leq e^{- c n}\,. \]
Altogether, w.h.p.\ $|X(i)| < \a^2 n^2 p \Phi^2$ for all $i$, as required.
\end{proof}

Observe that for all $p\geq n^{-1/2}$
the function $\Phi$ defined in Theorem~\ref{thm:Q} satisfies $1 \leq \Phi/\log n \leq e$, thus
the estimates \eqref{Yuvbound}--\eqref{Qbound} precisely feature a relative error of order $n^{-1/2}p^{-1}\log n$ in line
with the definition of $\zeta$ in Theorem~\ref{thm:aux}. Specifically, the conclusion of Theorem~\ref{thm:aux} provides the estimate $|Y_{u,v}-np^2| \leq \alpha n^{1/2}p\Phi$
for $\alpha = 3^{3M-1}$, thus fulfilling the hypothesis of Theorem~\ref{thm:Q} en route to a bound of $|Q/(\frac16 n^3p^3)-1| \leq C \zeta^2$ for $C = (\alpha e)^2$. Altogether, we can apply both theorems in tandem as long as $p \geq n^{-1/2+1/M}$. In particular, at that
point~\eqref{Qbound} guarantees that the process is still active with $(\frac16+o(1))n^{3/2+3/M}$ triangles, yet
there are merely $(\frac12-o(1))n^{3/2+1/M}$ edges left. This establishes the upper bound in our main result
modulo Theorem~\ref{thm:aux}.


\section{From degrees to extension variables}\label{sec:extension}

Our goal in this section is to show that strong control on degrees, co-degrees and the number of triangles
implies some level of control on arbitrary subgraph counts.  Let $Y_u$ denote the degree of the vertex $u$
and as before let $Y_{u,v}$ and $Q$ denote the co-degree of $u,v$ and total number of triangles respectively.
Our precise assumption will be that
for some absolute constant $C>0$
\begin{align}
\left\{\begin{array}{cll}
\left| Y_u/(np)-1\right| &< C\, \zeta & \quad\mbox{ for all $u\in V_G$}\\
\noalign{\medskip}
\left| Y_{u,v}/(np^2) - 1\right| &< C\, \zeta  &\quad\mbox{ for all $u,v\in V_G$}\\
\noalign{\medskip}
\left| Q/(\frac16 n^3 p^3)-1\right| &< C\, \zeta^2
\end{array}\right.\label{eq-codeg-ensemble-guarantee}
\end{align}
where $\zeta = \zeta(n,p) = n^{-1/2} p^{-1} \log n$.

The function $\zeta$ begins the process at $\zeta = n^{-1/2+o(1)}$ and then gradually
increases with $p$ until reaching a value of $n^{-\epsilon+o(1)}$ for some (arbitrarily small) $\epsilon>0$
at the final stage of our analysis. However, in order to emphasize that the value of this $\epsilon$ plays no role in
this section we will only make use of the fact that $p(t)$ is such that
\begin{equation}
  \label{eq-zeta-assumption}
  \log(1/\zeta) \big/ \log \log n \to \infty\,.
\end{equation}
Setting
\begin{equation}
  \label{eq-tau-star}
  \tau_\star = \min\{ t > 0 \;:\;\mbox{Eqs.~\eqref{eq-codeg-ensemble-guarantee} or~\eqref{eq-zeta-assumption} are violated}\}\,,
\end{equation}
we will show that as long as $t \leq \tau_\star$ we can w.h.p.\ asymptotically
estimate the number of copies of any balanced (a precise definition follows) labeled graph with $O(1)$ vertices rooted
at a prescribed subset of vertices. These estimates will
feature a multiplicative error-term of $1+O(\zeta^{\delta})$, where $\delta>0$ will only
depend on the uniform bound on the sizes of the graphs under consideration.

\begin{definition}\label{def:extension}[extensions and subextensions]
An {\bf extension graph} is defined to be a graph $ H = ( V_H,E_H ) $ paired
with a subset $I_H \subseteq V_H$ of distinguished vertices which form an independent set in $H$.
The {\bf scaling} of an extension graph is defined to be
\[\Sc_H =\Sc_H(n,p)= n^ {v_H - \iota_H} p^{e_H}\,,\]
where $ e_H = |E_H| $, $ v_H = |V_H|$ and $\iota_H = |I_H|$.
The {\bf density} of $H$ is defined to be
\[ m_H = e_H/( v_H - \iota_H )\,.\]

\noindent
A {\bf subextension} $K$ of $H$, denoted $K \subset H$, is an extension graph $ K = ( V_K, E_K) $ with $ V_K \subset
V_H$, $E_K \subset E_H$, and the same distinguished set $I_K = I_H$ (thus $V_K \supset I_H$).
We say that $K$ is a {\bf proper subextension}, denoted by $K \subsetneq H$, whenever $E_K \subsetneq E_H$.
We will denote the trivial subextension $K=(I_H,\emptyset)$ (the edgeless graph on $I_H$) by $\mathbf{1}$.
\noindent
For any subextension $K \subset H$, define the {\bf quotient}
$ H/K $ to be the extension graph $ H/K = ( V_H, E_H \setminus E_K) $ with the distinguished vertex set $ I_{H/K} = V_K $.
Observe that
\[ \Sc_H = \Sc_K \, \Sc_{H/K} \,.\]
\end{definition}

Let $H=(V_H,E_H)$ be an extension graph associated with the distinguished subset $I_H\subset V_H$
as in Definition~\ref{def:extension} and set $ e_H = |E_H| $, $ v_H = |V_H|$ and $\iota_H = |I_H|$.
Let $ \varphi : I_H \to V_G $ be an injection.  We are interesting in tracking the
number of copies of $H$ in $G=G(i)$ with $ \varphi (I_H) $ playing the role of the distinguished
vertices.  That is, we are interested in counting the extensions from $ \varphi( I_H) $
to copies of $H$.  Formally, an injective map $ \psi : V_H \to V_{G} $ {\bf extends $ \varphi $ to a copy of $H$}
if $\psi$ is a graph homomorphism that agrees with $\varphi$, that is, $ \psi \mid_{I_H} \equiv \varphi $ and $uv \in E_H$
implies that $\psi(u)\psi(v) \in E$.
Define
\[  \Psi_{H,\varphi} = \# \Big\{ \psi\in V_H \to V_G \;:\; \psi \text{ extends }
\varphi \text{ to a copy of } H \Big\} \]
Note that $\Psi_{H,\varphi}$ counts \emph{labeled} copies of $H$ rooted at the given vertex subset $\varphi(I_H)$.
%

We will track the variable $ \Psi_{H, \varphi}$ for extension graphs $H$ that are dense
relative to their subextensions. More precisely, we say that $H$ is {\bf balanced}
if its density satisfies $m_H \geq m_K$ for any subextension $K$ of $H$.
Furthermore, $H$ is {\bf strictly balanced} if
its density is strictly larger than the densities of all of its subextensions.
Our aim is to control the variable $\Psi_{H,\varphi}$ for any balanced extension graph $H$ up to the point where the scaling of
$H$ becomes constant, formulated as follows.  Let
\begin{equation}\label{eq-tH-def}
t_H = \min\{ t \;:\; \Sc_H(n,p(t)) \leq 1\} \,,
\end{equation}
where the times $t$ being considered here are, as usual, of the form $i/n^2$ for integer $i$.
Recalling that
$ m_H = e_H/( v_H - \iota_H ) $ observe that $p(t_H) = n^{-1/m_H} $, thus
\begin{equation}
  \label{eq-density-comparison}
  m_{H_1} \geq m _{H_2} ~\Leftrightarrow~ t_{H_1} \leq t_{H_2}
\end{equation} and in particular any balanced $H$ has $t_H \leq t_K$ for any subextension $K$ of $H$ (a strict inequality
holds when $H$ is strictly balanced).
Thus, the threshold up to which we can track $\Psi_{H,\varphi}$ will indeed be dictated by $t_H$ rather than by $t_K$ for
some $K\subsetneq H$.

 Finer precision in tracking the variables $\Psi_{H,\varphi}$ will be achievable up to a point slightly earlier than $t_H$
such that $\Sc_H$ is still reasonably
large. To this end it will be useful to generalize the quantity $t_H$ as follows: With~\eqref{eq-zeta-def} in mind, for
any $r\geq 0$ let
\begin{equation}
  \label{eq-t-minus-def}
  t^-_H(r) = \min\big\{ t : \Sc_H(n, p(t)) \leq 1/\zeta^{r} \big\}
\end{equation}
(again selecting among all time-points of the form $t=i/n^2$ for integer values of $i$).
We will establish these more precise bounds on $\Psi_{H,\varphi}$ up to $t_H^-(\delta)$ for some suitably small $\delta>0$,
i.e., up to a point where the
scaling $\Sc_H$ is still super-logarithmic in $n$. (Recall that $\zeta(t)$ is decreasing in $t$ and yet $\log (1/\zeta)$
tends to $\infty$ faster than $\log\log n$.)

When examining the times $t_H$ and $t_H^-$ one ought to keep in mind that, for any extension graph $H$ on $O(1)$ vertices,
the value of $\Sc_H$ decreases
within a single step of the process by a multiplicative factor of $1+O(n^{-2})$. In particular, in that case
$ (1-O(n^{-2}))\zeta^{-r} \leq \Sc_H \leq \zeta^{-r}$
at time $t_H^-(r)$.

The main result of this section is the following estimate on the variables $\Psi_{H_,\varphi}$, where the proportional
error will be in terms of $\zeta$, the
varying proportional error we have for the variables $Y_{u,v}$ and $Q$.

\begin{theorem}
  \label{thm-Psi-H-estimate}
For every $L>0$ there is some $0<\delta<1$ so that w.h.p.\ for all $t \leq \tau_\star$ and every
extension graph $H$ on $v_H \leq L$ vertices and injection $\varphi:I_H\to V_G$ we have:
\begin{enumerate}
  [\!(i)]
  \item\label{it-fine}
  If $H$ is balanced then $ \left| \Psi_{H, \varphi} - \Sc_H \right| \leq \Sc_H \zeta^{\delta}$ for all $t
\leq t_H^-(\sqrt{\delta})$.

  \item\label{it-coarse}
If $H$ is strictly balanced then
  $\left| \Psi_{H, \varphi} - \Sc_H \right| \leq (\Sc_H)^{1-\sqrt{\delta}} (\log n)^{\rho}$ with $\rho=3 e^{v_H-\iota_H}$
holds for all $t_H^-(\sqrt{\delta}) \leq t \leq t_H$.
\end{enumerate}
\end{theorem}

Before proving Theorem~\ref{thm-Psi-H-estimate} we will derive the following corollary which will be applicable to
extension graphs that are not necessarily balanced.
\begin{corollary}
\label{cor:general-ext}
Fix $L>0$ and let $0<t\leq \tau_\star$. Then w.h.p.\ for every extension graph $H$ on $v_H \leq L$ vertices
with
$\rho=\rho_H = 3\exp(v_H-\iota_H)$ we have:
\begin{enumerate}
  [(1)]
  \item\label{it-scaling-geq-1} If $\Sc_{K}(t) \geq 1$ for all subextensions $K\subset H$ then $\Psi_{H,\varphi}(t)\leq
(1+o(1))\Sc_H (\log n)^{\rho}$.
  \item\label{it-scaling-leq-1} If $\Sc_{H/K}(t) \leq 1$ for all subextensions $K\subset H$ then
$ \Psi_{ H, \varphi}(t)  \leq (1+o(1)) (\log n)^{\rho}$.
\end{enumerate}
\end{corollary}
\begin{proof}
Assume w.h.p.\ that the estimates in Theorem~\ref{thm-Psi-H-estimate} hold simultaneously for all extension graphs on at
most $L$ vertices throughout time $t$. We may further assume that every vertex $v\in V_H\setminus I_H$ has $\deg(v)\geq 1$. Indeed, for the statement in Part~\eqref{it-scaling-geq-1} this is w.l.o.g.\ since adding any isolated vertex to $V_H\setminus I_H$ would contribute a factor of $(1-o(1))n$ to $\Psi_{H,\varphi}$, balanced by a factor of $n$ to $\Sc_H$. For Part~\eqref{it-scaling-leq-1} this assumption is implied by the hypothesis $\Sc_{H/K}(t) \leq 1$, as we can take some subextension $K$ with $V_K=V_H\setminus\{v\}$ to obtain that $(p(t))^{\deg(v)} \leq 1/n$.

Define a sequence of strictly increasing sets $ I_H = X_0 \subsetneq X_1 \subsetneq X_2 \subsetneq \ldots
\subsetneq X_k = V_H $ as follows.
Associate each candidate for $X_i$ for $i\geq1$, one that contains $X_{i-1}$, with the extension graph $K_i=(X_i,E_i)$ where
\[ E_i = \{ (u,v) \in E_H \;:\; u,v\in X_i\mbox{ and either $u\notin X_{i-1}$ or $v\notin X_{i-1}$}~\}\,,\]
i.e., the set of edges of the induced subgraph on $X_i$ that do not have both endpoints in $X_{i-1}$, and set the
distinguished vertex set of
$K_i$ to be $I_{K_i}=X_{i-1}$ (guaranteed to be an independent set by our definition of $E_i$). With this definition
in mind, let $X_i$ be the
subset of vertices that maximizes $m_{K_i}$. If more than one such subset exists, let $X_i$ be one of these which has
minimal cardinality.

To see that the sequence of $X_i$'s is strictly increasing, let $i\geq 1$ and consider the potential values for $X_i$
building upon some
$X_{i-1} \subsetneq V_H$. Taking $X_{i}=X_{i-1}$ would yield $m_{K_i}=0$, strictly smaller than the density that would
correspond to $X_i=H$
(otherwise $V_H\setminus X_{i-1}$ would consist of isolated vertices in $H$, contradicting our minimal
degree assumption).
Hence, indeed $X_{i-1}\subsetneq X_i$ and overall $K_i$ is strictly balanced by construction. (We note in passing that
$K_k = H / K_{k-1}$ and
that $K_i$ is only a subextension of $H$ for $i\leq1$ since $I_H \subsetneq I_{K_i}$ for $i>1$.)

Motivating these definitions is the fact that for a given sequence $X_0,\ldots,X_k$ as above one can recover
$\psi\in V_H\to V_G$ (counted by $\Psi_{H,\varphi}$) iteratively from a sequence of $\psi_i \in V_{K_i}\to V_G$ (in which $\psi_0 = \varphi$ and $(\psi_{i})|_{X_{i-1}} = \psi_{i-1}$ for all $i$), hence
\begin{equation}
  \label{eq-Psi-Psi(Ki)}
  \Psi_{H,\varphi}(t) = \sum_{\psi_1 \in \Psi_{K_1,\varphi}} \,\sum_{\psi_2\in\Psi_{K_2,\psi_1}} \!\!\!\ldots \!\!\!\sum_{\psi_{k-1} \in \Psi_{K_{k-1},\psi_{k-2}}} \!\!\!\!\!\!\Psi_{K_k,\psi_{k-1}}(t) \leq \prod_{i=1}^k  \max_{\eta} \Psi_{K_i,\eta}(t)\,.
\end{equation}
(Here and throughout this section, with a slight abuse of notation we use the notation $\psi_i \in \Psi_{K_i}$ to denote an injective map $\psi_i$ counted by the corresponding variable $\Psi_{K_i}$.)

To prove Part~\eqref{it-scaling-geq-1} of the corollary, observe that $K_1 \subset H$ and so by hypothesis we
have $\Sc_{K_1}(t)\geq 1$, or equivalently, $t \leq t_{K_1}$. For any $2 \leq i \leq k$, the fact that $X_{i-1}\subsetneq X_i$
implies that $m_{K_{i-1}}$ is as large as $m_{K'}$, where $K'$ is the extension graph with vertex set $X_{i}$, distinguished
vertices $I_{K'}=X_{i-2}$ and all edges of $H$ between vertices of $X_{i}$ excluding those whose endpoints both lie in
$X_{i-2}$.
Since $E_{K'} = E_{i-1} \cup E_i$ we get
\[ \frac{|E_{i-1}|}{|X_{i-1}|-|X_{i-2}|} \geq \frac{|E_{i-1}|+|E_i|}{|X_{i}|-|X_{i-2}|}
\,,\]
which readily implies that
\[ m_{K_{i-1}} = \frac{|E_{i-1}|}{|X_{i-1}|-|X_{i-2}|} \geq \frac{|E_i|}{|X_{i}|-|X_{i-1}|} = m_{K_i}\,.
\]
Equivalently (via~\eqref{eq-density-comparison}), $t_{K_i}\leq t_{K_{i+1}}$ and by induction we thus have $t \leq t_{K_i}$
for all $i$.
Recalling that $K_i$ is strictly balanced, we can appeal to Theorem~\ref{thm-Psi-H-estimate} and derive that for any injection $\eta:I_{K_i}\to V_G$
\[ \Psi_{K_i,\eta}(t) \leq \left(1 + \max\big\{\zeta^\delta\,,\,(\Sc_{K_i})^{-\sqrt{\delta}}(\log n)^{\rho_i}\big\} \right)
\Sc_{K_i}
\leq (1+o(1))\Sc_{K_i} (\log n)^{\rho_i}\,,
\]
where $\rho_i = 3\exp(|X_i|-|X_{i-1}|)$. It is easily seen that by definition $\Sc_H = \prod_{i=1}^k \Sc_{K_i}$ and so
revisiting~\eqref{eq-Psi-Psi(Ki)} yields that under the hypothesis of Part~\eqref{it-scaling-geq-1} we have
\[  \Psi_{H,\varphi}(t) \leq (1+o(1))\Sc_{H} (\log n)^{\sum_{i=1}^k \rho_i}\,. \]
The proof of Part~\eqref{it-scaling-geq-1} is completed by the fact that $\exp(|X_i|-|X_{i-1}|) > 2$ for all $i$
(since the $X_i$'s are strictly increasing)
and in particular
\begin{equation}
  \label{eq-sum-rho-i}
  \sum_{i} \rho_i = 3\sum_i e^{|X_i|-|X_{i-1}|} < 3 \prod_i e^{|X_i|-|X_{i-1}|} = 3 e^{v_H - \iota_H} = \rho\,.
\end{equation}

It remains to prove Part~\eqref{it-scaling-leq-1}. Observe that the hypothesis in this part implies that $\Sc_{H}(t)\leq 1$
(by taking $K=\mathbf{1}$), that is we are considering some time point $t\geq t_H$.

To utilize the assumption on the scaling of quotients of $H$, we go back to the definition of the sequence of $X_i$'s and
argue that
\begin{equation}\label{eq-m(K_i)}
m_{K_i} \geq m_{H/K_{i-1}}\quad\mbox{for all $i\geq 1$}\,.
\end{equation}
Indeed, this follows from the fact that selecting $X_i=H$ would exactly correspond to having $K_i = H/K_{i-1}$, and
yet $X_i$ is chosen so as to maximize the value of $m_{K_i}$.

While $K_{i-1}$ itself may not be a subextension of $H$ (one has $K_{i-1}\subseteq H$ iff $i\leq 1$), one can modify
its distinguished vertex set to remedy this fact. Namely, the extension graph $K'=(X_{i-1},E_{i-1})$ with the distinguished
vertex set $I_{K'}=I_H$ is a subextension of $H$, and since $X_0 \subseteq X_{i-1}$ we get that
\[ \Sc_{H/K_{i-1}} = n^{v_H-|X_{i-1}|}p^{e_H-|E_{i-1}|} \leq n^{v_H-|X_0|}p^{e_H-|E_{i-1}|} = \Sc_{H/K'} \leq 1\,,\]
where the last inequality follows from our hypothesis on all the quotients of $H$.

From the inequality $\Sc_{H/K_{i-1}} \leq 1$ established above we infer that $t \geq t_{H/K_{i-1}}$.
At the same time, the combination of~\eqref{eq-density-comparison} and~\eqref{eq-m(K_i)} implies that
$t_{H/K_{i-1}} \geq t_{K_i}$, and altogether $t \geq t_{K_i}$.
Now, since $\Psi_{K_i,\varphi}(t)$ is monotone non-increasing in $t$ and we only aim to bound this quantity from above,
it suffices to provide an upper bound on $\Psi_{K_i,\varphi}(t_{K_i})$. To this end, apply Part~\eqref{it-coarse} of
Theorem~\ref{thm-Psi-H-estimate} (bearing in mind that $\Sc_{K_i}=1+O(n^{-2})$ at time $t_{K_i}$ and that $K_i$ is
strictly balanced) to get that for any injection $\eta:I_{K_i}\to V_G$
\[ \Psi_{K_i, \eta}(t_{K_i})\leq \left(1 + O(n^{-2})\right)(1+o(1))\left(1+(\log n)^{\rho_i}\right)\,, \qquad \mbox{ where }
\rho_i=3 e^{|X_i|-|X_{i-1}|}\,.\]
At this point we again appeal to~\eqref{eq-Psi-Psi(Ki)} to infer that
\[ \Psi_{H,\varphi}(t) \leq (1+o(1))(\log n)^{\sum_{i=1}^k \rho_i} \leq (1+o(1))(\log n)^\rho\,,\]
where we used the fact that $\sum_i \rho_i\leq \rho$ as seen in~\eqref{eq-sum-rho-i}.
This concludes the proof.
\end{proof}

\begin{proof}[\textbf{\emph{Proof of Theorem~\ref{thm-Psi-H-estimate}}}]

Let $\cA_L$ denote the set of all balanced extension graphs on at
most $L$ vertices and at least one edge (in the edgeless case one has $\Psi_{H,\varphi}=\Sc_H$
trivially by definition). Since $\cA_L$ is finite
and $m_H>0$ for all $H\in\cA_L$ we can choose a sufficiently small constant $ 0<\delta < 1$ as follows:
\begin{equation}
  \label{eq-delta-def}
  \delta = \min\bigg(\left\{ \Big|\frac1{m_A}-\frac1{m_B}\Big|^2 \,:\,
 \begin{array}{c}A,B\in\cA_L\\ m_A\neq m_B\end{array}
  \right\} \cup\left\{\frac1{(3e_A)^2} \,:\, A \in \cA_L\right\}\bigg)\,.
\end{equation}
Combining the fact that $|m_A^{-1} - m_B^{-1}| \geq \sqrt{\delta}$ whenever $m_A\neq m_B$ together with the identity
$p(t_A)=n^{-1/m_A}$ implies that for any $A,B\in\cA_L$ with $t_A < t_B$
\[ \Sc_B(t_A) = \left(n^{m_B^{-1}} p(t_A)\right)^{e_B} = n^{\left(m_B^{-1} - m_A^{-1}\right)e_B}
\geq n^{\sqrt{\delta} e_B} > \zeta^{-\sqrt{\delta}}\,,\]
where the last inequality was due to $e_B\geq 1$ (with room to spare). In particular,
\[   t_A < t_B ~\Rightarrow~ t_A < t_B^-(\sqrt{\delta})\quad\mbox{ for all $A,B \in \cA_L$ }\,.
\]
Contrary to the monotonicity of $t_A$ in $m_A$ as per~\eqref{eq-density-comparison}, it is generally not
true that $t_A^-(r)$ is monotone in $m_A$. However, the above definition of $\delta$ will enable us to still
relate these thresholds as we will later argue that if $m_A \geq m_B$ then $t_A^-(\sqrt{\delta})
\leq t_B^-((3-o(1))\delta)$.

\medskip\noindent\emph{Proof of Part~\eqref{it-fine}.}
Let $H$ and $ \varphi$ be fixed.
We will apply the previously used martingale arguments based on critical intervals.
We begin
with the expected change in $ \Psi_{H,\varphi}$.
Here the probability of deleting a given copy of $H$ corresponds as usual to the number of triangles
resting on each of its edges,
and summing these separately can only incur a double count of $O(1)$ since $v_H=O(1)$. Therefore,
\begin{align*}
\E \left[ \Delta \Psi_{H,\varphi} \right] & = - \sum_{\psi \in \Psi_{H,\varphi}} \sum_{uv \in E_H} \frac{ Y_{\psi(u)
\psi(v)}-O(1)}{Q}
 = -\frac{ \Psi_{H,\varphi} \, e_H}{Q} \left(1 + O(\zeta)\right) np^2 \\
 & = -\left(6 + O(\zeta)\right)\frac{ \Psi_{H,\varphi} \, e_H}{n^2 p}\,.
\end{align*}
Now set $ X = \Psi_{H,\varphi} - \Sc_H $. Our critical interval for $|X|$ will be
\begin{align*}
I_\Psi =\big[ (1-\tfrac{\delta}{2e_H})\Sc_H \zeta^{\delta} ~,~ \Sc_H \zeta^{\delta}\big]\,.
\end{align*}
To evaluate the one-step change in $X$, we first see that the one-step change in $\Sc_H$ is deterministically given by
\begin{align}
\Delta \Sc_H &= n^{v_H-\iota_H}\left(\left(p-6/n^2\right)^{e_H}-p^{e_H}\right) = -
n^{v_H-\iota_H}\frac{6e_H}{n^2}\left(p-O(n^{-2})\right)^{e_H-1}\nonumber\\
&= -\left(1+O(n^{-2}p^{-1})\right)\frac{6e_H}{n^2 p} \Sc_H = -\left(6+o(\zeta)\right)
\frac{e_H}{n^2 p}\Sc_H\,.\label{eq-Delta-SH}
\end{align}
When $|X|\in I_\Psi$ we therefore have
\begin{equation*}
\begin{split}
\E[ \Delta X ] & = -\left(6 + O(\zeta)\right) \frac{e_H}{n^2 p} (\Sc_H + X) + (6+o(\zeta)\frac{e_H}{n^2 p}\Sc_H \\
& = - (6+o(1))\frac{  e_H  }{ n^2 p}X + O(\zeta) \frac{ e_H }{ n^{2} p } \Sc_H = - (6+o(1)) \frac{ e_H  }{ n^2 p} X,
\end{split}
\end{equation*}
where the last inequality used the fact that $\Sc_H \zeta  = o(X)$ whenever $|X|\in I_\Psi$ as then we have
$|X| \asymp \Sc_H \zeta^\delta$ for $\delta<1$.
Next we define what would be a supermartingale when $|X| \in I_\Psi$, namely
\[ Z = |X| - \Sc_H \zeta^{\delta} \,.\]
As in~\eqref{eq-Delta-SH}, the one-step change in $\Sc_H \zeta^{\delta}$ deterministically satisfies
\begin{equation}
  \label{eq-Delta-SHzeta}
  \Delta \Sc_H\zeta^{\delta} =  -\left(1+O(n^{-2}p^{-1})\right)\frac{6(e_H-\delta)}{n^2 p}\Sc_H \zeta^\delta
= -(6+o(1))\frac{e_H-\delta}{n^2 p}\Sc_H \zeta^\delta\,,
\end{equation}
hence whenever $|X|\in I_\Psi$ we have
\begin{align*}
\E[ \Delta Z] & = - (6+o(1)) \frac{ e_H }{ n^2 p}|X|  + (6+o(1)) (e_H-\delta) \frac{ \Sc_H \zeta^{\delta}}{ n^2p}\\
&\leq -\frac{6+o(1)}{n^2 p}\Sc_H \zeta^\delta \left((1-\tfrac{\delta}{2e_H})e_H - (e_H-\delta)\right)
\leq -\frac{3\delta + o(1)}{n^2 p}\Sc_H \zeta^\delta < 0\,,
\end{align*}
with the last inequality valid for sufficiently large $n$.
Thus, $Z(t\wedge \tau_\Psi)$ is indeed a supermartingale when $\tau_\Psi$ is the stopping time for having
$|X|$ exit the interval $I_\Psi$.

We turn to establish concentration for $Z(t\wedge \tau_\Psi)$. We will invoke the following inequality due to Freedman~\cite{Freedman}, which was originally stated for martingales yet its proof extends essentially unmodified to supermartingales.
\begin{theorem}[\cite{Freedman}, Thm~1.6] \label{thm:concentrate}
Let $(S_0,S_1,\ldots)$ be a supermartingale w.r.t.\ a filter $(\cF_i)$.
Suppose that $S_{i+1}-S_i\leq B$ for all $i$, and write $V_t=\sum_{i=0}^{t-1} \E\left[(S_{i+1} - S_{i})^2 \mid \cF_{i}\right]$. Then
for any $s,v>0$
\[\P\big(\{S_t\geq S_0+s,\,V_t\leq v\}\mbox{ for some $t$}\big)\leq \exp\bigg\{-\frac{s^2}{2(v+Bs)}\bigg\}\,.\]
\end{theorem}
To qualify an application of this theorem, we must first establish $L^\infty$ and $L^2$ bounds on $\Delta Z$, which will follow from individually examining the various edges of $H$ whose potential removal at time $t < t_H^-(\sqrt{\delta})$ may
decrement the value of $\Psi_{H,\varphi}$.

For each edge $ab \in E_H$ we let $T = T_{ab}$ denote the subextension of $H$ that has the smallest scaling
among subextensions that contain $ab$ (and a minimal number of vertices if there
is more than one subgraph with the minimum scaling).  Notice that $ t_H \leq t_T $ since $T$ is a subextension of $H$.
Furthermore, if $K$ is any subextension of $T$ then, being also a subextension of $H$ we again have $t_H \leq t_K$.
In particular, $t < t_K$ for all $K \subset T$ and we can invoke Corollary~\ref{cor:general-ext}
(Part~\eqref{it-scaling-geq-1}) to give
\[ \Psi_{T,\varphi} \leq (1+o(1))\Sc_T (\log n)^{\rho_T}\,,\]
where $\rho_T = 3 \exp(v_T - \iota_T)$.

It is worthwhile explaining why invoking Corollary~\ref{cor:general-ext} at this point is justified. Indeed, one can prove the statements of both parts of the present theorem (and in effect also its corollary) by gradually extending the time range where they all hold. Formally, we consider the stopping time $\tilde{\tau}$ at which one of the assertions of the theorem fails (for any extension graph). As long as $\tilde{\tau} > t$ we may assume the bounds of this theorem (and hence also those of its corollary) as part of our attempt to increase their validity to the next time step, and so on. The bound we will present for $Z(t\wedge \tau_\Psi)$ will remain valid in conjunction with the additional stopping time $\tilde{\tau}$, and similarly for the second part of the theorem (featuring an analogous proof). Finally, analyzing $Z(t\wedge\tau_\Psi\wedge \tilde{\tau})$ and its Part~\eqref{it-coarse} counterpart is sufficient for proving the theorem since, by definition, at least one of these supermartingales (for the various extension graphs) will need to violate its corresponding large deviation bound in order for $\tilde{\tau}$ to occur.

Consider the effect of removing a given edge $xy\in E_G$ (as part of some triangle $xyz$) on $\Psi_{H,\varphi}$.
Copies of $H$ are lost iff $xy$ corresponds to some edge $ab\in E_H$ in at least one such copy, hence we can bound
this amount by adjusting $I_H$ to include $a,b$, which would be mapped to $x,y$. Namely, let $H^*=H^*_{ab}$ be the
extension graph with $I_{H^*}=I_H\cup\{a,b\}$ whose edges are $E_H$ minus any edges between vertices of $I_{H^*}$.
Further let $\varphi^*$ be the corresponding injection extending $\varphi$ to include $x,y$ as the images of $a,b$
(if this is inconsistent with $\varphi$ then this copy remains intact) to yield
\[ |\Delta\Psi_{H,\varphi}| \leq \sum_{ab\in E_H} \Psi_{H^*_{ab},\varphi^*}\,.\]
Recalling the above definition of $T=T_{ab}$, let $T^*_{ab}$ similarly include $a,b$ in its distinguished
set of vertices.
We then get that
\begin{equation}
  \label{eq-Delta-Psi}
  |\Delta\Psi_{H,\varphi}| \leq \sum_{ab\in E_H} \sum_{\eta \in \Psi_{T^*_{ab},\varphi^*}}
\Psi_{H^*_{ab}/T^*_{ab},\eta} = \sum_{ab\in E_H} \sum_{\eta \in \Psi_{T^*_{ab},\varphi^*}}
\Psi_{H/T,\eta}\,,
\end{equation}
where the equality is simply due to the fact that $V_T = V_{T^*}$, thus $H^*_{ab}/T^*_{ab} = H^*_{ab}/T$, and
moreover $a,b\in V_T$ and so $H^*_{ab}/T=H/T$. Consider this last term $\Psi_{H/T,\eta}$. We claim that
\begin{equation}
  \label{eq-Psi-H/T}
  \Psi_{H/T,\eta} \leq (1+o(1))\Sc_{H/T} (\log n)^{\rho_{H/T}}\,.
\end{equation}
Indeed, by the choice of $T$ we have $\Sc_H \geq \Sc_T$ (as $H$ is itself could play the role of $T$), and
so $\Sc_{H/T}=\Sc_{H}/\Sc_{T} \geq 1$. Moreover, if $K$ is any subextension of $H$ with $ab\in E_K$ and
$V_T\subset V_K$ then again $\Sc_{K/T} \geq 1$. Yet by definition, if $ab\notin E_K$ then $K/T = \bar{K}/T$
where $\bar{K}$ includes the extra edge $ab$ (as both of its endpoints lie in $T$) and altogether every
$K\subset H$ with $V_T\subset V_K$ has $\Sc_{K/T}\geq1$. Noticing that every subextension of $H/T$ can
be written as $K/T$ for such $K$ now qualifies an application of Part~\eqref{it-scaling-geq-1} of
Corollary~\ref{cor:general-ext}, from which the desired inequality~\eqref{eq-Psi-H/T} follows.

To conclude the $L^\infty$-bound on $\Delta Z$ we need to estimate $\Psi_{T^*_{ab},\varphi^*}$. We will argue
that this quantity is at most $\log^{O(1)}n$. Let $U^*$ be a subextension of $T^*$, and let $U$ be the
modification of $U^*$ into a subextension of $T$ via setting its distinguished vertex set to $I_T$ and
adding any edges among these vertices in $T$ (in particular the edge $ab$). By construction, $\Sc_{U^*}/\Sc_{U} =
\Sc_{T^*}/\Sc_T$. In addition, since $U$ contains the edge $ab$ then by the choice of $T$ we must have $\Sc_U \geq \Sc_T$.
Combining the two we get that $\Sc_{U^*}\geq \Sc_{T^*}$ and in particular $\Sc_{T^*/U^*} \leq 1$, i.e., all the quotients
of $T^*$ have scaling at most $1$. Part~\eqref{it-scaling-leq-1} of Corollary~\ref{cor:general-ext} now implies that
\[ \Psi_{T^*_{ab},\varphi^*} \leq (1+o(1))(\log n)^{\rho_{T^*}} < (1+o(1))(\log n)^{\rho_{T}}\,.\]
Combining this with~\eqref{eq-Delta-Psi} and \eqref{eq-Psi-H/T} we get that
\[ |\Delta \Psi_{T,\varphi}| = O\left(\Sc_{H/T} (\log n)^{\rho_{H/T}+\rho_T}\right)\,,\]
while recalling the estimates for $\Delta \Sc_H$ and $\Delta \Sc_H\zeta^\delta$ in~\eqref{eq-Delta-SH},\eqref{eq-Delta-SHzeta}
we see that they are negligible in comparison (note that $\Sc_T \leq n^2p$ since a valid candidate for $T$ is the graph containing the single edge $ab$)  now extends this to the following $L^\infty$ bound on the one-step change in $Z$.
\[ |\Delta Z| = O\left(\Sc_{H/T} (\log n)^{\rho_{H/T}+\rho_T}\right)\,.\]

For an $L^2$ bound on the one-step change $\Delta Z$, again consider the effect of removing an edge corresponding
to some $ab\in E_H$. For a given $\psi \in \Psi_{T,\varphi}$, the probability that the next drawn triangle will
contain this edge is $Y_{\psi(a)\psi(b)}/Q
= (6+o(1))n^{-2}p^{-1}$ due to our estimates on co-degrees and the total number of triangles present. In particular,
the probability to remove an edge corresponding to $ab\in E_H$ in any copy specified by some $\psi\in \Psi_{T,\varphi}$
is at most
\[ \Psi_{T,\varphi} O(n^{-2}p^{-1}) = O\left( \Sc_T (\log n)^{\rho_T} n^{-2} p^{-1} \right)\,.\]
As shown above, this can only modify $Z$ by an additive term of $O(\Sc_{H/T} (\log n)^{\rho_{H/T}+\rho_T})$. Thus, if
we enter the critical interval $I_\Psi$ at some time $t_0=t(i_0)$ then for any $i\geq i_0$,
\begin{align*}
\E [(\Delta Z)^2 \mid \cF_{i}] &\leq \sum_{ab\in E_H} O\left( \Sc_T (\log n)^{\rho_T} n^{-2} p^{-1}
\left(\Sc_{H/T} (\log n)^{\rho_{H/T}+\rho_T}\right)^2\right)\\
&= O\left( \Sc_H \Sc_{H/T} n^{-2} p^{-1} (\log n)^{2\rho_{H/T}+3\rho_T}\right)\,.
 \end{align*}
 The quantities $\Sc_H$ and $\Sc_{H/T}$ above are w.r.t.\ time $t(i)$, yet our scaling is monotone decreasing
in $i$ and so we can replace these by the corresponding scaling terms at $t_0$. Summing the result over at
most $O(n^2 p)$ remaining possible steps gives
\[
  \sum_{i\geq i_0} \E [(\Delta Z)^2 \mid \cF_{i}] \leq O\left( \Sc_H(t_0) \Sc_{H/T}(t_0)
(\log n)^{2\rho_{H/T}+3\rho_T}\right)\,,
\]
where $\Sc_H(t_0)$ is short for $\Sc_H(n,p(t_0))$. Since $e_H>0$ one can bound $\sum_{i\geq i_0} \Sc_H(t)\Sc_{H/T}(t) p^{-1}(t)$ via integration to obtain that it is $O(n^2 \Sc_H(t_0)\Sc_{H/T}(t_0))$, leading to the above inequality.

Compare this last bound for the summation $\sum_i \E[(\Delta Z)^2 \mid \cF_{i}]$ with the product of the $L^\infty$
bound $B=O\left(\Sc_{H/T} (\log n)^{\rho_{H/T}+\rho_T}\right)$ and a deviation of
$s = \frac{\delta}{3e_H}\Sc_H \zeta^\delta$. The upper bound on
 $\sum_i \E[(\Delta Z)^2\mid\cF_{i}]$ has order at least $B s \zeta^{-\delta} (\log n)^{\rho_{H/T}+2\rho_T}$, clearly dominating $B s$. Turning to Freedman's inequality
and once again applying the fact that $\Sc_H = \Sc_T \Sc_{H/T}$ we get that for some fixed $c>0$,
\begin{align*}
\P\Big(Z(t\wedge \tau_\Psi) \geq Z(t_0) + s\mbox{ for some $t\geq t_0$}\Big) &\leq
\exp\left(-c \Sc_T(t_0) \zeta^{2\delta} (\log n)^{-2\rho_{H/T}-3\rho_T}\right)\,.
 \end{align*}
 We claim that $\Sc_T(t_0) \geq (1-o(1))\zeta^{-3\delta}$, and in fact this holds for
any $t_0 \leq t \leq t_H^-(\sqrt{\delta})$.
If $m_T < m_H$ then $t < t_T^-(\sqrt{\delta})$ since $t \leq t_H < t_T$ and given our definition of
$\delta$. Moreover, $\sqrt{\delta} \leq \frac13$ (as implied by Eq.~\eqref{eq-delta-def} and the fact that
$e_A \geq 1$) thus in this case
\begin{equation}\label{eq-ST-lower-bound}
\Sc_T \geq (1-o(1))\zeta^{-\sqrt{\delta}} \geq  (1-o(1))\zeta^{-3\delta}\,.\end{equation}
On the other hand, when $m_T = m_H$ we can write
\[ \Sc_T = \big(n^{1/m_T}p\big)^{e_T} = \big(n^{1/m_H}p\big)^{e_T} = (\Sc_H)^{e_T/e_H} \geq \Sc_H^{1/e_H} \geq (1-o(1))
\zeta^{-\sqrt{\delta}/e_H}\,,\]
with the final inequality due to $t \leq t_H^-(\sqrt{\delta})$. By~\eqref{eq-delta-def} we have $\sqrt{\delta}
\leq 1/(3 e_H)$, hence~\eqref{eq-ST-lower-bound} is again valid.
Altogether we can conclude that
\begin{align*}
\P\Big(Z(t\wedge \tau_\Psi) \geq Z(t_0) + s\mbox{ for some $t\geq t_0$}\Big) &\leq \exp\left(-(c-o(1))\zeta^{-\delta}
(\log n)^{-2\rho_{H/T}-3\rho_T}\right) \\ &\leq \exp\left(-\log^2 n\right)\,,
 \end{align*}
where the last inequality is thanks to~\eqref{eq-zeta-assumption}. Eq.~\eqref{eq-ST-lower-bound} similarly implies that
\[ \Sc_{H/T}(\log n)^{O(1)} = \frac{\Sc_H}{\Sc_T} (\log n)^{O(1)} \leq \Sc_H \zeta^{3\delta - o(1)}\,,\]
that is, the $L^\infty$-bound on $Z$, which is valid regardless of whether $|X|\in I_\Psi$, satisfies
$B = o(\Sc_H \zeta^\delta)$.
Accounting for the first step in which $|X|$ enters the interval $I_\Psi$ we get
\[ Z(t_0) + s \leq -\tfrac{\delta}{2e_H}\Sc_H \zeta^\delta + B + \tfrac{\delta}{3e_H}\Sc_H \zeta^\delta < 0\]
for sufficiently large $n$. In particular, $|X| < \Sc_H\zeta^\delta$ for all $t\geq t_0$ except with probability
$\exp(-\log^2 n)$. This probability allows us to absorb a union bound over all $n^{O(1)}$ choices for $(H,\varphi)$ and
$t_0$, thus concluding the proof of Part~\eqref{it-fine} of the theorem.

\medskip\noindent\emph{Proof of Part~\eqref{it-coarse}.}
Let $H$ and $ \varphi$ be fixed, and let $ X = \Psi_{H,\varphi} - \Sc_H$.
We follow the framework of applying a martingale argument within an appropriate critical interval for $|X|$,
which would now be
\begin{align*}
I_\Psi =\big[ \big(1-\tfrac{\sqrt{\delta}}2\big) (\Sc_H)^{1-\sqrt{\delta}}(\log n)^{\rho_H}  ~,~ (\Sc_H)^{1-\sqrt{\delta}}(\log n)^{\rho_H} \big]\,,
\end{align*}
where $\rho_H=3\exp(v_H-\iota_H)$.
From Part~\eqref{it-fine} of the theorem we get that $|X| \leq \Sc_H \zeta^\delta$ at time $t = t_H^-(\sqrt{\delta})$.
Yet at this time $\Sc_H \leq \zeta^{-\sqrt{\delta}}$
by definition and so $\Sc_H^{\sqrt{\delta}} \zeta^{\delta} \leq 1$. We thus infer that
 $|X| \leq \Sc_H \zeta^\delta \leq \Sc_H^{1-\sqrt{\delta}}$, and in particular $|X|\notin I_\Psi$ at time $t^-_H(\sqrt{\delta})$ provided that $n$ is large enough.

As in the proof of the first part of the theorem,
\begin{align*}
\E \left[ \Delta \Psi_{H,\varphi} \right] & = -(6+O(\zeta))\frac{e_H}{n^2 p} \Psi_{H,\varphi} \,, \qquad
\Delta \Sc_H = -(6+o(\zeta))\frac{e_H}{n^2 p}\Sc_H\,,
\end{align*}
and so
\begin{align*}
 \E[ \Delta X ] = -(6+o(1))\frac{e_H}{n^2 p} X + O(\zeta)\frac{e_H}{n^2p}\Sc_H
= -(6+o(1))\frac{e_H}{n^2 p} X\,,
\end{align*}
where the last equality will be justified by showing that $\Sc_H \zeta = o(|X|)$ when $|X|\in I_\Psi$. Indeed, with $0<\delta<1$ a sufficient condition for $\Sc_H\zeta =
o((\Sc_H)^{1-\sqrt{\delta}}(\log n)^{\rho_H})$ is to have $\Sc_H = o\left(\zeta^{-1}(\log n)^{\rho_H}\right)$, which in turn holds since
$t \geq t_H^{-}(\sqrt{\delta})$ and thanks to the assumption in~\eqref{eq-zeta-assumption} on the
decay rate of $\zeta$.

Similarly to~\eqref{eq-Delta-SH} we have $\Delta (\Sc_H)^{1-\sqrt{\delta}} = -(6+o(\zeta))(1-\sqrt{\delta})\frac{e_H}{n^2 p} (\Sc_H)^{1-\sqrt{\delta}}$
deterministically, hence the random variable defined by
\[ Z = |X| - (\Sc_H)^{1-\sqrt{\delta}} (\log n)^{\rho_H} \]
satisfies the following whenever $|X|\in I_\Psi$ and for any sufficiently large $n$:
\[
\E[ \Delta Z] = -(6+o(1))\frac{ e_H }{ n^2 p}|X|  + (6+o(1))(1-\sqrt{\delta})\frac{e_H}{n^2 p}(\Sc_H)^{1-\sqrt{\delta}}(\log n)^{ \rho_H} < 0\,.
\]
Therefore, if $t_0=t(i_0)$ is the first time where $|X|>(1-\frac{\sqrt{\delta}}2)(\Sc_H)^{1-\sqrt{\delta}}(\log n)^{\rho_H}$ and
$\tau_\Psi = \min\{t\geq t_0 : |X(t)|\notin I_\Psi\}$ then $Z(t\wedge \tau_\Psi)$ is a supermartingale.

For $L^\infty$ and $L^2$ bounds on $\Delta Z$ we again wish to analyze the effect of removing some edge $xy\in E_G$
playing the role of some $ab\in E_H$ in a copy of $H$. Letting $T = T_{ab}$ as before be the subextension of $H$ that
has the smallest scaling
among subextensions that have $ab\in E_T$, we now claim that in the setting of this part of the theorem we in fact
have $ T = H $.
  Indeed, since $H$ is strictly balanced, any $K \subsetneq H$ must have $t_H < t_K$ and thus also
$t_H < t^-_K(\sqrt{\delta})$.
  The hypothesis $t_H^-(\sqrt{\delta}) \leq t\leq t_H$ then implies that $\Sc_H \leq \zeta^{-\sqrt{\delta}}$ and
yet $\Sc_K > \zeta^{-\sqrt{\delta}}$, precluding the choice $T_{ab}=K$.

Revisiting the argument used to prove the previous part of the theorem, we let $H^*_{ab}$ denote the extension graph
that adds $a,b$ to $I_H$ and let $\varphi^*$ denote a corresponding injection. We then have
\[ |\Delta\Psi_{H,\varphi}| \leq \sum_{ab\in E_H} \Psi_{H^*_{ab},\varphi^*} \]
and yet, the by the minimality of $\Sc_{H^*_{ab}}$ we must have $\Sc_{H^*_{ab}/K^*}\leq 1$ for any subextension $K^*\subset H^*_{ab}$ and
so Part~\eqref{it-scaling-leq-1} of Corollary~\ref{cor:general-ext} implies that
\[ \Psi_{H^*_{ab},\varphi^*} \leq (1+o(1))(\log n)^{\rho_{H^*}}\,.\]
Combining these two equations yields that the potential change in $\Psi_{H,\varphi}$ is at most
\[ |\Delta\Psi_{H,\varphi}| = O\Big(\max_{ab\in E_H} (\log n)^{\rho_{H^*_{ab}}}\Big)\,.\]
Similarly to~\eqref{eq-Delta-SH}, the quantities $\Sc_H$ and $(\Sc_H)^{1-\sqrt{\delta}}(\log n)^{\rho_H}$ deterministically
change by at most
\[ O\big(\Sc_H n^{-1} p^{-2} (\log n)^{\rho_H}\big) = O\left(\Sc_H \zeta^2 (\log n)^{2+\rho_H}\right) = o(1)\]
 due to the fact that $\Sc_H \leq \zeta^{-\sqrt{\delta}}$ for $\delta<1$ and the decay rate of $\zeta$. Altogether we get
\begin{equation}
  \label{eq-Z-ii-L-inf-bound}
  |\Delta Z| = O\Big(\max\big\{ (\log n)^{\rho_{H^*_{ab}}} : {ab\in E_H}\big\}  \Big) =
O\left( (\log n)^{(1/e) \rho_{H}}\right)\,, \end{equation}
  where we used the fact that $\iota_{H^*} > \iota_H $ since at least one of $a,b$ is in
$I_{H^*_{ab}}\setminus I_H$ (as otherwise $ab\notin E_H$ as $I_H$ is an independent set), and so
$\rho_H = 3e^{v_H-\iota_H} \geq e \rho_{H^*}$.
Note that this $L^\infty$-bound holds in general regardless of whether $|X|\in I_\Psi$.
In particular, since $|X|\leq(1-\frac{\sqrt{\delta}}2)(\Sc_H)^{1-\sqrt{\delta}}(\log n)^{\rho_H}$ one step prior to $t_0$ and can only increase by
$O((\log n)^{(1/e)\rho_{H}})$ at a given step, we have
\begin{align}
|X(t_0)| &= (1-\tfrac{\sqrt{\delta}}2)(\Sc_H(t_0))^{1-\sqrt{\delta}}(\log n)^{\rho_H} + O\left((\log n)^{(1/e)\rho_{H}} \right)\nonumber\\
&= (1-\tfrac{\sqrt{\delta}}2+o(1))(\Sc_H(t_0))^{1-\sqrt{\delta}}(\log n)^{\rho_H}\label{eq-X(t0)}\,,
 \end{align}
(where the term $(\log n)^{(1/e)\rho_H}$ was absorbed in the higher order $(\log n)^{\rho_H}$-term) and similarly, for all $t_0\leq t <\tau_\Psi$
\[
|X(t)| \leq (1+o(1))(\Sc_H(t))^{1-\sqrt{\delta}}(\log n)^{\rho_H}\,.
\]
As such, since the probability to remove an edge that corresponds to some $ab\in E_H$ in a given copy $\psi\in\Psi_{H,\varphi}(t)$ is at most
$Y_{\psi(a)\psi(b)}/Q=(6+o(1))n^{-2}p^{-1}$, the overall probability for this event to occur for some
$\psi\in\Psi_{H,\varphi}(t)$ is at most
\[ \Psi_{H,\varphi}(t) O(n^{-2}p^{-1}) \leq O\left(\left(\Sc_H(t) +
(\Sc_H(t))^{1-\sqrt{\delta}}(\log n)^{\rho_H}\right)n^{-2}p^{-1}\right)\,.
\]
This implies that for any $t=t(i)$ for $i \geq i_0$ we have that $\E [(\Delta Z)^2 \mid \cF_{i}]$ is at most
\[\sum_{ab\in E_H} O\left(\left(\Sc_H(t) +
(\Sc_H(t))^{1-\sqrt{\delta}}(\log n)^{\rho_H}\right)n^{-2}p(t)^{-1}(\log n)^{2e^{-1}\rho_{H}} \right)\,,
\]
and integrating this over the $O(n^2 p(t_0))$ remaining possible steps (again noting that $e_H (1-\sqrt{\delta})>0$ and so
$\sum_{i\geq i_0} (\Sc_H)^{1-\sqrt{\delta}} p^{-1} = O(n^2 (\Sc_H(t_0))^{1-\sqrt{\delta}})$ as in the proof of Part~\eqref{it-fine} of this theorem)
now gives
\begin{align}
  \sum_{i\geq i_0} \E [(\Delta Z)^2 \mid \cF_{i}] &\leq \sum_{ab\in E_H} O\left(\left(\Sc_H(t_0) +
(\Sc_H(t_0))^{1-\sqrt{\delta}}(\log n)^{\rho_H}\right)(\log n)^{2e^{-1} \rho_H} \right) \nonumber\\
&= O\left(\left(\Sc_H(t_0) + (\Sc_H(t_0))^{1-\sqrt{\delta}}(\log n)^{\rho_H}\right)(\log n)^{2e^{-1}\rho_{H}} \right)  \label{eq-L2-sum-bound-coarse} \,.
\end{align}
Compare this bound with the product of the $L^\infty$ bound given in~\eqref{eq-Z-ii-L-inf-bound} and a deviation value of
$s = \frac{\sqrt{\delta}}4 (\Sc_H(t_0))^{1-\sqrt{\delta}}(\log n)^{\rho_H}$. This product has order $(\Sc_H(t_0))^{1-\sqrt{\delta}}(\log n)^{(1+e^{-1})\rho_H}$ whereas our
bound in~\eqref{eq-L2-sum-bound-coarse} has order at least $(\Sc_H(t_0))^{1-\sqrt{\delta}}(\log n)^{(1+2e^{-1})\rho_H}$, strictly larger.
We can thus extract from Freedman's inequality that for some fixed $c>0$,
\begin{align*}
\P\Big(\exists t>t_0: Z(t\wedge \tau_\Psi) \geq Z(t_0) + s\Big) &\leq
\exp\left(- c \frac{(\Sc_H(t_0))^{2-2\sqrt{\delta}} (\log n)^{(2-2e^{-1})\rho_H}}
{\Sc_H(t_0) + (\Sc_H(t_0))^{1-\sqrt{\delta}}(\log n)^{\rho_H}}\right)
 \end{align*}
which, according to the larger term between $\{\Sc_H\,,\,(\Sc_H)^{1-\sqrt{\delta}}(\log n)^{\rho_H}\}$, is at most
\[ \exp\left(-(c/2) \min\left\{ (\Sc_H)^{1-2\sqrt{\delta}}(\log n)^{(2-2e^{-1})\rho_{H}}\,,\,
(\Sc_H)^{1-\sqrt{\delta}}(\log n)^{(1-2e^{-1})\rho_{H}}\right\} \right)\,.
\]
Going back to the definition $\rho_H= 3\exp(v_H-\iota_H)$, since $v_H-\iota_H \geq 1$ (otherwise $H$ is edgeless) we have
$\rho_H \geq 3e$ and thus the exponent of the $\log n$ term above is at least $3(e-2)>2$ and we can deduce that for any
sufficiently large $n$ we have
\[ \P\Big(\exists t>t_0: Z(t\wedge \tau_\Psi) \geq Z(t_0) + s\Big) \leq \exp(- \log^2 n)\,.\]
Finally, recall that $Z(t_0)=-(\frac{\sqrt{\delta}}2-o(1))(\Sc_H(t_0))^{1-\sqrt{\delta}}(\log n)^{\rho_H}$ as guaranteed by~\eqref{eq-X(t0)}. We can thus conclude that $Z(t) < 0$, that is $|X(t)|<(\Sc_H(t))^{1-\sqrt{\delta}}(\log n)^{\rho_H}$,
for all $t\geq t_0$ and large enough $n$ except with probability $\exp(-\log^2 n)$. A union bound over the $n^{O(1)}$ choices for $t_0$ and
the pairs $(H,\varphi)$
completes the proof of the theorem.
\end{proof}


\begin{figure}
\centering
\fbox{
\includegraphics[width=\textwidth]{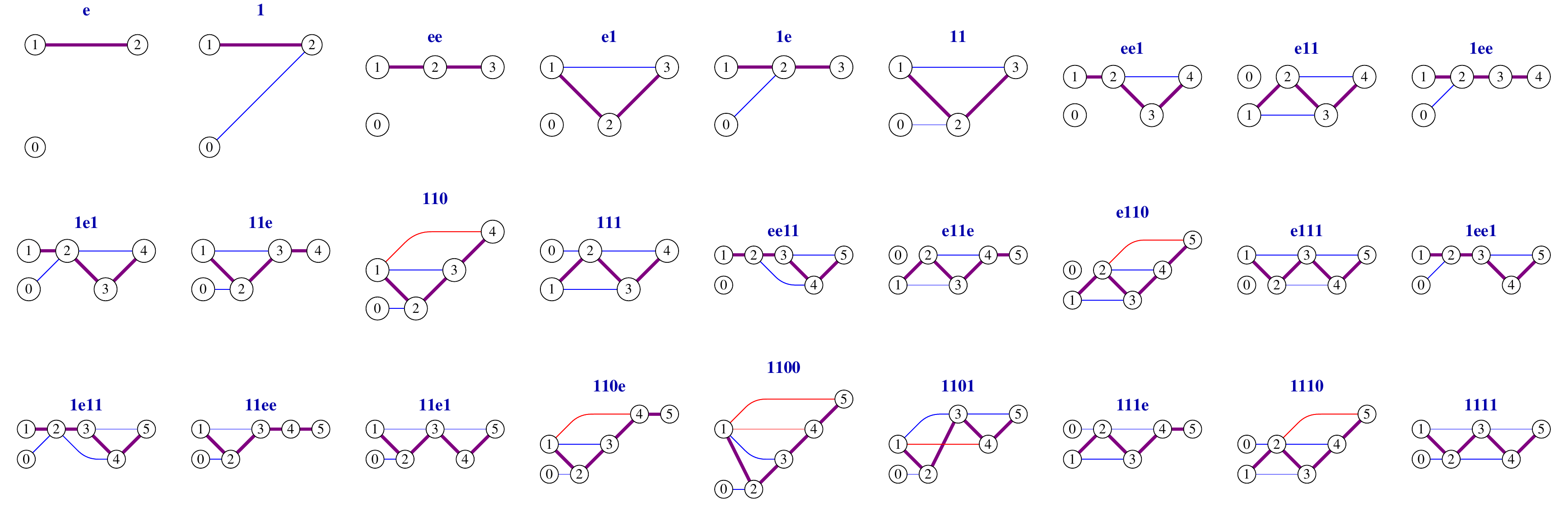}
}
\caption{The set of triangular ladders $\sL_\pi$ (labeled by $\pi$) for all $|\pi|\leq 4$.}
\label{fig:ladders}
\end{figure}

\section{Triangular ladders and their extensions}\label{sec:ladders}
\begin{definition}\label{def:ladders}
  [triangular ladders]
Let $\Pi \subset \{\emph{\edg},0,1\}^*$ denote the set of all words
which do not have $0$ in indices 1 and 2, have at most 2 occurrences of $ \emph{\edg}$ and do not
 contain the substrings $ \emph{\edg} 0 $, $ \emph{\edg} 1 0 $ or $ \emph{\edg} 1 \emph{\edg}$.
For any $\pi \in \Pi$ of length $k \ge 1$,
the {\bf triangular ladder} $ \sL_\pi$ is the following labeled undirected graph on the vertex set $\{0,1,2,\ldots,k+1\}$:\\
\begin{enumerate}[(i)]
\item \label{it-k=1} $k=1$: $\sL_1$ has the two edges $\{02, 12\}$; $ \sL_\edg$ has the single edge 12.
\item \label{it-k>1} $k\geq2$:
let $H = \sL_{\pi^-}$ where $\pi^-$ is the
$(|\pi|-1)$-prefix of $\pi$. Let $\sL_\pi$ be the graph obtained from $H$
as follows: First add the vertex $k+1$ and the edge $(k,k+1)$. In addition, if $\pi(k)=1$
add the edge $(k-1,k+1)$, whereas if $\pi(k)=0$ 
add the edge $(v,k+1)$, where $v$ is the unique neighbor of $k$ in $H$ other than $k-1$.
\end{enumerate}
\end{definition}
\noindent
Observe that vertex 0 is isolated iff $\pi(1)=\edg$. The set of triangular ladders of length at most $5$ is illustrated in Fig~\ref{fig:ladders}. Implicit in the above definition is the fact that by construction, the last vertex of $\sL_\pi$, i.e., the vertex $k+1$
when $|\pi|=k$, has precisely two neighbors if $ \pi(k) \in \{0,1\}$ (one of which is the vertex $k$) and exactly
one neighbor (the vertex $k$) when $ \pi(k)= \edg$.
In particular, since the substring $\edg0$ cannot appear in $\pi\in\Pi$, the case $\pi(k)=0$ in Item~\eqref{it-k>1} above
has $\pi(k-1)\neq\edg$ and so the vertex $k$ indeed has two neighbors in $H$.
Further note that if $ \pi( \ell) = \edg $ for $\ell>1$ then the vertex $ \ell $ is a cut-vertex of
$ \sL_\pi $. Namely, in this case there is no edge $(a,b) \in E( \sL_{\pi}) $ such that
$ a < \ell < b $.

\begin{definition}
The {\bf scaling} of a triangular ladder $ \sL_\pi$ w.r.t.\ some given $n$ and $p$ is defined to be
$ \Sc_{\pi} = \Sc_{\pi}(n,p) = n^{v_\pi-2} p^{e_\pi} $, where $ e_\pi = |E(\sL_\pi)|$ and $v_\pi = |V(\sL_\pi)|$.
\end{definition}
\noindent
The scaling $\Sc_\pi$ will correspond asymptotically to the expected number of labeled copies of $\sL_\pi$
rooted at two given vertices (matching the labels $0$ and $1$) along our process.
Recalling the notion of extension graphs (see Definition~\ref{def:extension}), $\Sc_\pi$
is simply the scaling of the extension graph obtained from $\sL_\pi$ via the distinguishing vertex set $I=\{0,1\}$.
 Observe that if $|\pi|=k$ and $\pi$ contains $0\leq s\leq 2$ copies of $ \edg$ then
 $ \Sc_\pi = n^{k} p^{2k-s} $.

Our focus will be on tracking a special subset of the family of all triangular ladders, whose formulation will require the following definition.

\begin{definition}[$f$-fan]\label{def:fan}
For $ f \geq 3$, an {\bf $f$-fan} at some vertex $a>0$ in a triangular ladder $ \sL_\pi $ is
a sequence of vertices that has one of the following two forms:
\begin{enumerate}[(1)]
\item $\pi(a-2)\neq \emph{\edg}$ and vertices $ a+1, a+2, \ldots, a+f+1 $ are all adjacent to the vertex $a$.
\item $\pi(a-2)= \emph{\edg}$ and vertices $ a-1,a+1, \ldots, a+f $ are all adjacent to the vertex $a$.
\end{enumerate}
Equivalently, $ \sL_\pi$ has an $f$-fan at $a>0$ with $\pi(a-2)\neq \emph{\edg}$ if $ \pi(a+1) =1 $
and $\pi(a+i)=0$ for $2\leq i \leq f$.
Similarly, $\sL_\pi$ has an $f$-fan at $a>0$ with $\pi(a-2)=\emph{\edg}$ if $ \pi(a+1) =1$
and $\pi(a+i)=0$ for $2\leq i \leq f-1$.
\end{definition}
Note that there cannot be an $f$-fan at 0, nor can there be one at a vertex $a$ whenever $\pi(a)=\edg$.
Figure~\ref{fig:fans} illustrates the structure of $f$-fans as defined above.
\begin{figure}
\centering
\fbox{
\includegraphics[width=5.5in]{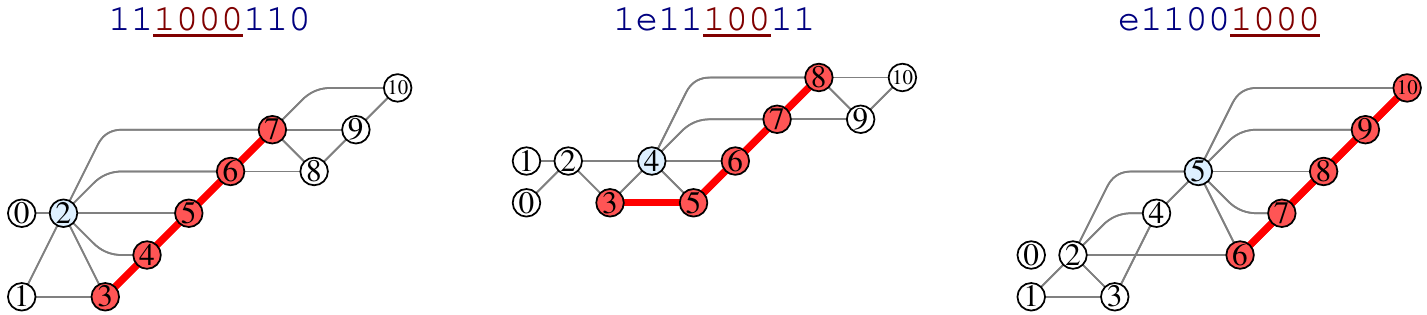}
}
\caption{Fans in triangular ladders: The vertices highlighted in red form 4-fans at vertices 2, 4 and 5 (note the
exceptional structure of the fan at $a=4$ due to having $\pi(a-2)=\edg$).}
\label{fig:fans}
\end{figure}

\begin{definition}\label{def:M-bnd}[$M$-bounded triangular ladders]
For $M\geq3$, define $ \cB_M $ to be the set of all $ \pi \in \Pi $ that do not contain an $M$-fan and in addition satisfy one of the following conditions:
\begin{enumerate}[(a)]
\item \label{no-e} $\pi $ contains no copy of $\emph{\edg}$ and has length at most $3M-1$, or

\item \label{one-e} $\pi $ contains 1 copy of $\emph{\edg}$, has length at most $2M$ and does not have $\pi(2M)=\emph{\edg}$, or

\item \label{two-e} $\pi $ contains 2 copies of $\emph{\edg}$, has length at most $ M+1$ and does not have $ \pi(M+1) = \emph{\edg}$.
\end{enumerate}
If $\pi \in \cB_M$ we say that $\sL_\pi$ is an {\bf $M$-bounded triangular ladder}.
\end{definition}
It is easily seen that $\cB_M\subset \cB_{M'}$ for any $M \leq M'$.
The following additional observation demonstrates some of the reasons for the precise choice of constants in the definition of $\cB_M$.
\begin{observation}
The family $\cB_M$ is closed under the prefix operation.
Indeed, let $\sigma$ be a proper prefix of some $\pi\in \cB_M$.
Precluding an $M$-fan from $\pi$ clearly also precludes it from $\sigma$. While $\sigma$ may contain fewer copies of $\edg$,
these would meet relaxed length criteria in Definition~\ref{def:M-bnd}, hence $\sigma\in \cB_M$.
\end{observation}

Having defined the family of graphs we wish to track throughout the process, we proceed to a classification of the edges in $\sL_\pi$ for $\pi\in\cB_M$. This will in turn be used to study the effect of removing such edges as part of triangles eliminated by our process.

\begin{definition}[boundary edges]\label{def:boundary}
Let $ \sL_\pi$ be an $M$-bounded triangular ladder for some $M\geq 3$ and let $yz \in E(\sL_\pi)$
be such that $ 0<y<z$ and $ \pi(y) \neq \emph{\edg}$. Let
\[ \pi' = \begin{cases}
\pi_{z-1} \circ 0 & \text{ if } y < z-1 \\
\pi_{z-1} \circ 1 & \text{ if } y =z-1\,,
\end{cases} \]
where $ \pi_{\ell} $ denotes
the $\ell$-prefix of $\pi$ and $\circ$ denotes string concatenation.
We say $ yz$ is a {\bf boundary edge} of $ \sL_\pi$ w.r.t.\ $\cB_M$ iff $ \pi' \notin \cB_M $.  In other
words, an edge is a boundary edge if the ladder obtained by removing all vertices larger than $z$ and
then attaching a new vertex adjacent to $y,z$ no longer belongs to the family $\cB_M$.
We classify a boundary edge $yz$ into one of two categories:
\begin{itemize}[$\bullet$]
  \item {\bf outer boundary edges}: $z=|\pi|+1$ (last vertex in $\sL_{\pi}$) and in addition
  $|\pi|$ is maximal w.r.t.\ its number of copies of $\emph{\edg}$ (i.e., $|\pi$ is $3M-1$, $2M$ or $M+1$ if it
  has 0, 1 or 2 such copies resp.).
\item {\bf side boundary edges}: all other boundary edges. Necessarily $y<z-1$, there is at most 1 copy of \emph{\edg}\ in $\pi$,
the ladder $\sL_\pi$ has an $(M-1)$-fan at $y$ whilst $\sL_{\pi'}$ has an $M$-fan at $y$.
\end{itemize}
\end{definition}

\begin{figure}
\centering
\fbox{
\includegraphics[width=\textwidth]{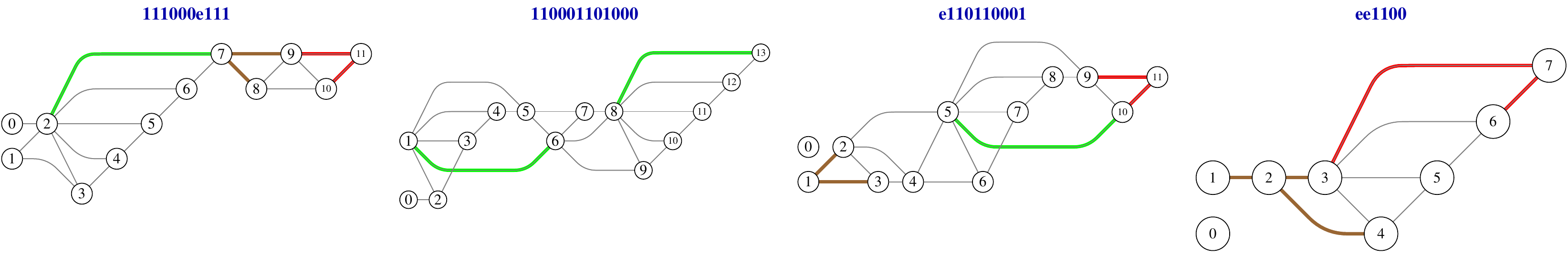}
}
\caption{Examples of $M$-bounded triangular ladders for $M=5$. Outer boundary edges w.r.t.\ $\cB_5$ are highlighted
in red,
side boundary edges are green and initial edges are brown.}
\label{fig:bnd-ladders}
\end{figure}

To see that indeed side boundary edges are accompanied by $\pi$ of the form specified in the above
definition ($y<z-1$, at most 1 copy of $\edg$ and an $M$-fan in $\pi'$) argue as follows. First,
we claim that if $\pi$ has 2 copies of $\edg$ then $ \sL_\pi $ cannot have side boundary
edges. To show this, recall that in this case $|\pi|\leq M+1$ and there are 2 copies of $\edg$ in $\pi_{M}$.
After excluding the case $|\pi|=M+1$, $z=M+2$ (corresponding to an outer boundary edge) we are left with either
$z = M+1$ or $z < M+1$.
In the former case, $|\pi'|=M+1$ and $\pi'$ retains both $\edg$ copies of $\pi$ in its $M$-prefix,
thus $\pi'\in\cB_M$ by Item~\eqref{two-e} of Definition~\ref{def:M-bnd}. In the latter, $|\pi'| \leq M$,
and either it contains 2 copies of $\edg$ --- here $\pi'\in\cB_M$ as per the value of $|\pi'|$ --- or
it contains 1 copy of $\edg$ and again $\pi'\in\cB_M$ since $\sL_{\pi'}$ does not contain an $M$-fan (any $f$-fan
in $\pi'$ would have $\pi'$ be at least $f+2$ symbols long).
We have thus established that $\pi$ can have at most 1 copy of $\edg$. Next, recall that by definition $\pi$ has no $M$-fan. The exclusion of the
cases $|\pi|=2M$, $z=|\pi|+1$ with 1 copy of $\edg$ and $|\pi|=3M-1$, $z=|\pi|+1$ with no copies of $\edg$
(corresponding to outer boundary) implies that $|\pi'|\leq |\pi|$, and yet
$\pi'\notin \cB_M$ (since $yz$ is a boundary edge), hence it must contain an $M$-fan. Comparing
$\pi$ to $\pi'$ now implies that $\pi'=\pi_{z-1}\circ0$ so as to accommodate a new $M$-fan (missing
from $\pi_{z-1}$), therefore indeed $y<z-1$. Finally, it further follows from this analysis that
\begin{equation}
  \label{eq-side-boundary-yz-relation}
  z = \left\{\begin{array}{ll}
  y+M   & \mbox{if $\pi(y-2) \neq \edg$}\,,\\
  y+M-1   & \mbox{if $\pi(y-2) = \edg$}\,.
\end{array}\right.
\end{equation}


We complement the definition of boundary edges by categorizing the each of the remaining edges into one of two types:
\begin{definition}[non-boundary edge classes]\label{def:non-boundary-edges} Let $\sL_\pi$ be an $M$-bounded triangular ladder
and let $yz$ be a non-boundary edge of $\sL_\pi$. We say that $yz$ is an {\bf initial edge} if $y<z$ and $\pi(y)=\emph{\edg}$,
and otherwise we say that $yz$ is an {\bf interior edge}.
\end{definition}

Figure~\ref{fig:bnd-ladders} depicts the family of 3-bounded triangular ladders,
highlighting, outer boundary edges, side boundary edges and initial edges in these ladders w.r.t.\ $\cB_3$.

\begin{remark*}
While $f$-fans and $M$-bounded triangular ladders were both defined for values of $f,M\geq 3$, one could
extend these notions
to $f=M= 2$. However, going back to Definition~\ref{def:fan}, a $2$-fan at vertex $3$ occurs whenever
$\pi$ begins with $\edg111$,
as opposed to $f$-fans for $f>2$ which are accompanied by non-overlapping patterns of the form $10\ldots0$.
In light of this anomaly at $M=2$, and given that our application would have $M$ be arbitrarily large,
our attention will be restricted to $f,M\geq3$.
\end{remark*}

\begin{definition}[backward extension]
Let $ \sL_\pi $ be a triangular ladder that contains some outer boundary edge $yz$.
The {\bf backward extension from $yz$ to $ \sL_\pi$},
denoted by $ B_{\pi,yz} $, is defined as the graph obtained by deleting from $\sL_ \pi $
all edges within the distinguished vertex set $ I = \{0,1,y,z\} $.
\end{definition}
\noindent
Note that we have the following
scalings for these extension graphs
\begin{equation}
  \label{eq-backward-ext-density}
  \Sc_{B_{\pi, yz}} = \begin{cases}
 n^{3M-3} p^{6M-3}   & \text{ if } |\pi| = 3M-1\,, \\
  n^{ 2M-2} p^{ 4M-2}  & \text{ if } |\pi| = 2M\,, \\
n^{M-1} p^{2M-1} & \text{ if } |\pi| = M+1\,.
\end{cases}
\end{equation}

As in Definition~\ref{def:extension}, the density of a backward extension graph $H$ is equal to $m_H=e_H/(v_H-4)$,
as the distinguished vertex set of $H$ is $I=\{0,1,y,z\}$ for some outer boundary edge $yz$. To qualify
an application of Theorem~\ref{thm-Psi-H-estimate} on backward extension graphs we need to verify that
they are balanced (that is, the density of a backward extension graph $H$ is as large as the density of any
subextension $K$ of $H$).

\begin{lemma}\label{lem:balance}
Let $\sL_\pi$ be an $M$-bounded triangular ladder for some $M\geq 3$, and let $yz$
be an outer boundary edge of $\sL_\pi$. Then the backward extension graph $B_{\pi, yz}$ is
balanced.
\end{lemma}
\begin{proof}
With Definition~\ref{def:boundary} in mind, without loss of generality we have $z=|\pi|+1$
where $|\pi|$ is either $3M-1$, $2M$ or $M+1$ depending on whether it has $0$, $1$ or $2$ occurrences of $\edg$, respectively.
We will first prove the statement of the lemma for the case where $\pi(1)=\edg$.

Suppose for the sake of contradiction that $B=B_{\pi, yz} = (V_B,E_B)$ is not balanced, and let $C=(V_C,E_C)$
be a maximal proper subextension of $B$ such that $ m_C > m_B $.
Let $a,b$ be vertices such that $[a,b)$ is the first interval of vertices missing from $V_C$,
that is, $j\in V_C$ for $1\leq j < a$ and for $j=b$ whereas
$ j \notin V_C$ for $a\leq j < b$. Observe that $2\leq a < b \leq z = |V_B|$ since $\{1,y,z\}\subset V_C \subsetneq V_B$.
By the maximality assumption, the
extension graph we get by adding $a, \dots, b-1$ (and all incident edges) to $C$
has density at most $m_B$.  Let $e_0$ be the number of edges within
$ a, \dots, b-1 $ or connecting these vertices to $ V_C $ in $B$ (or equivalently
in $\sL_\pi$, since $\{0,1,y,z\}\subset V_C$ and so the two edge-sets are equal).
 Since $m_C > m_B$ we now get
\begin{equation}
\label{eq:req}
\frac{e_0}{ b-a} < m_B\,.
\end{equation}
As each vertex that joins a triangular ladder has either one or two neighbors among the previous vertices, one of which is
always the preceding
vertex, we also have
\begin{equation}
\label{eq:starter}
 e_0 \geq 2 (b-a) +1 - \#\left\{ a \leq j \leq b-1 \;:\; \pi(j-1)=\edg\right\}\,.
 \end{equation}
In this observation we included the edge $ (b-1,b) $ (the $+1$ term in the right hand side)
yet disregarded another potential edge joining $b$ to some $\gamma < b-1$ (which is present
iff $\pi(b-1) \neq \edg$). We call $\gamma b$ the
`extra' edge, and a careful consideration of when $\gamma b$ exists will play a role below.
Consider now 2 cases depending on the structure of $ \pi $.

\begin{enumerate}[(1)]
\item\label{it-one-e-copy} $\pi$ contains one copy of $ \edg$:

Here we have $|V_B|=|V(\sL_\pi)|=2M+2$ and $e_\pi =2(|V_B|-2)-1=4M-1$. We claim that $1y,1z \notin E(\sL_\pi)$
and thus the extension graph $B$ satisfies $E_B=E(\sL_{\pi})\setminus \{yz\}$.
 Indeed, as the pattern $\edg10$ is forbidden from any $\pi\in \Pi$, the neighborhood of the vertex $1$
is always a subset of $\{2,3\}$.
Since the form of $\pi$ dictates that $z=2M+1$ and that $\sL_\pi$ has no $M$-fan, in particular
$y \geq M+1 \geq 4$ and so neither $y$ nor $z$ are neighbors of $1$, as claimed. Altogether, $|E_B| = e_\pi -1 = 4M-2$ and
\begin{equation}\label{eq:extra1}
m_B = \frac{|E_B|}{|V_B| - 4} = 
2 + \frac{1}{M-1} \,.
\end{equation}
Further note that the extra edge $\gamma b$ with $\gamma<b-1$ is present since $b-1 \geq a \geq 2$ and
so $\pi(b-1)\neq \edg$.
Several cases are now possible.

\begin{inparaenum}[\bf{Case} (\ref{it-one-e-copy}.1):]
\item $ a > 2 $.
Then the symbol $ \edg$ does not
appear in $ \pi(a-1), \dots, \pi(b-2) $.
It follows from~\eqref{eq:req},\eqref{eq:starter},\eqref{eq:extra1}
that $b-a > M-1$, and as $3 \leq a < b \leq 2M+1$ we get
\begin{equation}
  \label{eq-b-a-bounds}
   M \leq b-a \leq 2M-2\,.
\end{equation}
Examine the extra edge $\gamma b$. By Definition~\ref{def:M-bnd}, the ladder $ \sL_\pi $ does
not contain an $M$-fan, implying that
$\gamma \geq b-M$
(this inequality is tight if there is an $(M-1)$-fan at $\gamma$).
Combining this with the fact $b-a\geq M$ yields $a \leq \gamma < b-1$, hence
 we can adjust~\eqref{eq:starter} by the extra edge $\gamma b$ to have
 $ e_0 \ge 2(b-a)+2 $. By~\eqref{eq-b-a-bounds} we now get
 \begin{equation}
   \label{eq-b-a-comparison}
   \frac{e_0}{b-a} \geq 2 + \frac{2}{b-a} \geq 2 + \frac{1}{M-1} = m_B\,,
 \end{equation}
contradicting~\eqref{eq:req}.

\item $a=2$ and $ b >3 $. Note that here $b < 2M+1$ as otherwise $V_C=\{0,1,2M+1\}$ contradicting the
fact that $\{0,1,y,z\}\subset V_C$.
Since $ \pi(1) = \edg $, we have a contribution of $-1$
from the last term in~\eqref{eq:starter}.  On the other hand, since $ b >3$,
the extra edge $\gamma b$ must feature $2 \leq \gamma \leq b-2$ (the vertex $1$ has only $\{2,3\}$ as neighbors in particular
excluding $b\geq 4$). Taking this into account we can modify~\eqref{eq:starter} into $ e_0 \geq 2(b-a) +1 $.
At this point we can immediately preclude the case $b+1\in V_C$.
Indeed, if  $ b+1 \in V_C$ then this vertex has 2 neighbors among previous vertices (as $\pi(b)\neq\edg$),
one of which must be among $2, \ldots, b-1$, thus $e_0 \geq 2(b-a)+2$. At the same time, having $b+1\leq |V_B|=2M+1$
implies that $b-a \leq 2M-2$ and so inequality~\eqref{eq-b-a-comparison} is again valid, contradicting~\eqref{eq:req}.
Assume therefore that $ b+1 \not\in V_C $ and let $ b'$ be the smallest index greater
than $b$ such that $ b' \in V_C$.  Let $e_1$ be the number of edges
among $ 2, \ldots,b-1, b+1, \ldots,  b'-1$ plus those joining this set with $ V_C$.  In place of~\eqref{eq:req}
we now write
\[  \frac{e_1}{b'- 3} < m_B = 2 + \frac{1}{M-1}\,, \]
and accounting for the edge 12, two edges going backwards from each of the vertices $3 \leq j \leq b'-1$ and one
edge going back from $b'$ we arrive at $ e_1 \ge 2b'-4 $. This implies that $e_1/(b'-3) \geq 2 + 2/(b'-3) \geq m_B$,
a contradiction.

\item $a=2$ and $b=3$. In this case, the first interval missing from $V_C$ is simply the singleton $\{2\}$. Recall
that $|V_B|=2M+2\geq 8$ and that $\pi$ begins with the prefix $\edg 11$ (thus the vertex 2 is adjacent to $\{1,3,4\}$
to the very least in $\sL_\pi$). Clearly, if $ 4 \in V_C$ then $e_0/(b-a) \geq 3 > m_B$ in contradiction with~\eqref{eq:req}.
Assume therefore that $4 \notin V_C$ and let $b'$ be the smallest index greater
than $3$ such that $ b' \in V_C$. Again let $e_1$ be the number of edges among the vertex
$2,4, \ldots , b'-1$ plus the number of edges joining these vertices with  $V_C$.
Instead of~\eqref{eq:req} we have
\begin{equation}
\label{eq:reqspec}
 \frac{e_1}{b'-3} < m_B = 2 + \frac{1}{M-1}.
  \end{equation}
Counting the edges $12,23$, two edges going backward from each $ 4 \leq j \leq b'-1 $ as well as the edge
$ (b'-1,b') $ we get $ e_1 \geq 2b'-5$.
Therefore, it follows from~\eqref{eq:reqspec} that $ b' > M+2$. This case will be concluded
by examining the extra edge $\gamma'b'$ such that $\gamma' < b'-1$ (present since $\pi(b'-1)\neq \edg$).
Crucially, $\sL_\pi $ does not contain an $M$-fan at the vertex 3 and so
 $\gamma' \neq 3$ (if $f$ is the maximal $f$-fan at $3$ then by definition the neighbors of vertex $3$ are precisely
$\{2,4,\ldots,3+f\}$, a set precluding $b'$). Since $b' \geq M+3\geq 6$ we further have $\gamma'\neq 1$ and
overall the extra edge $\gamma'b'$ further increments $e_1$ yielding $e_1\geq 2b'-4$. Since $b'\leq 2M+1$ this gives
$e_1/(b'-3) \geq 2+2/(b'-3) \geq 2+1/(M-1)$, contradicting~\eqref{eq:reqspec}.
\end{inparaenum}

\item\label{it-two-e-copies} $\pi$ contains two copies of $\edg$:

Here $|V_B|=|V(\sL_\pi)|=M+3$ and $e_\pi=2(|V_B|-2)-2 = 2M$. Again we claim that $1y,1z \notin E(\sL_\pi)$
and hence $E_B=E(\sL_{\pi})\setminus \{yz\}$.
To see this, observe that when $\pi$ begins with $\edg\edg$ the vertex $1$ has only $2$ as its neighbor, which
has a subset of $\{1,3,4\}$ as its neighbors due to the forbidden pattern $\edg10$. Therefore, in this case having
$z=M+2\geq 5$ implies that $y>2$ and $1y \notin E(\sL_\pi)$. Alternatively, if $\pi$ has 2 copies of $\edg$ and
yet $\pi(2)\neq \edg$, then the index $j>1$ such that $\pi(j)=\edg$ in fact satisfies $j>3$ (and hence is not
a neighbor of 1) due to the forbidden patterns $\edg0$ and $\edg1\edg$, and moreover $j<M+1$ due to
Definition~\ref{def:M-bnd}. The fact that the vertex $j$ is a cut-vertex of $\sL_\pi$ now leads to
the conclusion that $|E_B|=e_\pi - 1 = 2M-1 $, thus
\[ m_B = \frac{|E_B|}{|V_B|-4}
= 2 + \frac{1}{M-1} \,.\]

\begin{inparaenum}[\bf{Case} (\ref{it-two-e-copies}.1):]
\item There are no copies of $ \edg$
in $ \{\pi(a-1), \ldots, \pi( b-2)\} $. Observe that $a\geq 3$ since $\pi(1)=\edg$.
By~\eqref{eq:starter} we have
\begin{equation}\label{eq:extra2}
 e_0/( b-a) \ge 2 + 1/(b-a) \geq 2 + 1/(M-1)
\end{equation}
(where we used the fact
that $b-a \leq M-1$ as $3 \leq a < b \leq M+2$), contradicting~\eqref{eq:req}.

\item $a \geq 3$ and the symbol $\edg$ appears (once) in $\{\pi(a-1), \ldots, \pi(b-2)\} $. Here there are no
later occurrences of $\edg$ and so the extra edge
$\gamma b$ (with $\gamma < b-1$) is present. Moreover, $\gamma \geq \ell$
where $a-1\leq \ell \leq b-2$ is the aforementioned index of $\edg$ in $\pi$, since, crucially, vertex $\ell$
is a cut-vertex of $\sL_\pi$.
It now follows that when $\ell \geq a$ the extra edge $\gamma b$ contributes to the count of $e_0$ and gives
$e_0 \geq 2(b-a)+1$, which as shown above (see \eqref{eq:extra2}) contradicts~\eqref{eq:req}.
 Conversely, if $\pi(a-1)=\edg$ then the
forward neighbors of $a-1$ are confined to the subset $\{a,a+1\}$. Consequently, as long as $b > a+1$ we again
get that $\gamma \geq a$ and the extra edge $\gamma b$ again leads to the sought inequality $e_0\geq 2(b-a)+1$
that yields a contradiction to \eqref{eq:req}.
It remains to deal with the case $\pi(a-1)=\edg$ and $b=a+1$, where we claim that $b < z$. To see this, observe
that if $b$ were the last vertex of $C$ (vertex $z=M+2$) then its neighbors would be $\{a-1,a\}$. However, since
$y\neq a$ (by the fact that $a\notin V_C$ whereas $y$ must belong to $C$ as it is a distinguished vertex) we would
then get $y=a-1$ and $\pi(y)=\edg$, contradicting the fact that $yz$ is a boundary edge
(see Definition~\ref{def:boundary}).
It now follows that $|\pi|\geq b = a+1 $ and since $\pi(a-1)=\edg$ we must have $\pi(a)=\pi(a+1)=1$ by definition of $\Pi$.
In particular, the edge $(a,a+2)$ belongs to $E(\sL_\pi)$ (and also to $B$, recalling that $a$ is not a
distinguished vertex).
To conclude this case argue as follows. Either $a+2 \in V_C$, in which case $e_0 = 3$ due to the neighbors
$\{a-1,a+1,a+2\}$ of the vertex $a$ and we have $e_0/(b-a) = 3 > 2+1/(M-1)$ as needed. Otherwise, letting $b'$
be the smallest index larger than $b=a+1$ such that $b'\in V_C$ (which must exist since $z=M+3\in V_C$), we
can consider the set $\{a, a+2,\ldots,b'-1\}$ and denote by $e_1$ the number of edges incident to it in $C$.
We must have $e_1/(b'-a-1) < m_B$ and yet $e_1 \geq 2(b'-a-1)+1$ since it includes two edges going backward
from each of $a+2,\ldots,b'-1$, the two edges $(a-1,a)$,$(a,a+1)$ and finally the edge $(b'-1,b')$. This
yields $e_1/(b'-a-1)\geq 2+1/(b'-a-1)$, and as $3\leq a < b' \leq M+2$ we get that
 $e_1/(b'-a-1) \geq 2+1/(M-1) = m_B$,
contradiction.

\item $a= 2$ and $b>3$. If the second occurrence of $ \edg $ does
not appear in $\{ \pi(2), \ldots, \pi(b-1) \}$ then the extra edge $\gamma b$ contributes to
$e_0$ (we are assured that $1 b\notin E(\sL_\pi)$ since $b > 3$) and again we obtain $e_0 \geq 2(b-a)+1$
to contradict~\eqref{eq:req}.
(Note that $b-a\leq M-1$ in this case --- hence the contradiction --- due to the assumption that $\pi$ features a
second copy of $\edg$ in one of the indices $\{b,\ldots,|\pi|-1\}$, in fact implying that $b-a \leq M-2$).
Assume therefore that the second occurrence
of $ \edg$ does appear in the sequence. Observe that the assertion $b<z=M+2$ still holds since otherwise
$V_C=\{1,z\}$ leaving out the distinguished vertex $y$. Now, if $ b+1 \in V_C$ then this vertex has
a neighbor in $ 2, \ldots , b-1 $, and the corresponding edge cancels the contribution of the extra copy
of $\edg$ to the right hand side of~\eqref{eq:starter}. This again gives $e_0 \geq 2(b-a)+1$ and leads to
the same contradiction (recall that $b \leq M+1$ and so $b-a \leq M-1$).
It remains to deal with the case $b+1 \notin V_C$. Again let $b'$ be the minimal index such that
 $b<b'\in V_C$ and let $e_1$ count the number of edges incident to $2,\ldots,b-1,b+1,\ldots,b'-1$, we must
have $e_1/(b'-3) < m_B$. However, $e_1$ includes two edges going backward from all but two of these $b'-3$
vertices (those corresponding to copies of $\edg$ that contribute just one backward edge), the extra edge
$\gamma b$ and the edges $(b,b-1)$ and $(b'-1,b')$. Overall, $e_1 \geq 2(b'-3)+1 \geq 2+1/(M-1) = m_B$ as before.

\item $a=2$ and $b=3$. Recall that the last vertex in $V_C$ is $z=M+2 \geq 5$, and examine the vertex $4$.
Regardless of whether $\pi$ begins with $\edg11$ or $\edg\edg11$, the vertex $2$ has $\{1,3,4\}$ as its neighbors
in $ E(\sL_\pi)$ and
so if $4\in V_C$ we immediately get that $e_0/(b-a) = 3 > m_B$, contradiction. Assume therefore that $4 \notin V_C$,
let $ b'>4$ be the minimal index beyond $4$ with $b' \in V_C$ (bearing in mind that $z\in V_C$ by definition) and let
$e_1$ be the number of edges incident to $2,4,\ldots,b'-1$ in $V_C$.

If the symbol $ \edg$ does not appear in $ \{\pi(4), \ldots ,\pi(b'-2) \}$ (it is forbidden from occurring at $\pi(3)$
as per the definition of $\Pi$) then $ e_1 \geq 2b'-5 $, since the vertices $4,\ldots,b'-1$ together send $2(b'-4)$
edges backwards, accompanied by the edges $12$, $23$ and $(b'-1,b')$. Thus, in this case
$e_1/(b'-3)\geq 2+1/(b'-3) \geq m_B$, as desired. Finally, if the symbol $\edg$ does appear in
$\{\pi(4),\ldots,\pi(b'-2)\}$ then the extra edge $\gamma' b'$ (satisfying $\gamma' < b'-1$) is present.
 Moreover, if we let $4\leq j \leq b'-2$ be the index such that $\pi(j)=\edg$ then $\gamma' \geq j$ since the
vertex $j$ is a cut-vertex in $\sL_\pi$. As such, the edge $\gamma' b'$ contributes to $e_1$, thus canceling the
 negative contribution of the extra $\edg$ symbol. That is, again $e_1 \geq 2b'-5$, eliminating this case.
\end{inparaenum}
\end{enumerate}
All cases above led to a contradiction, completing the proof of the lemma for ladders beginning with $\edg$.
It remains to prove the balance property for ladders beginning with $1$, which will follow from similar arguments yet here
the vertex 0 will play a delicate role.

\begin{enumerate}[(a)]
\item\label{it-0-no-e} $\pi$ contains no copies of $ \edg$:
Here $v_\pi=3M+1$ and $e_\pi = 6M-2$, and therefore $|V_B\setminus I_B|=3M-3$ (here we used the fact that
$y\neq 0,1$ --- and thus $|I_B|=4$ --- since $z=3M$ and forbidding an $M$-fan implies that $y\geq z-M = 2M$)
and $|E_B|=6M-3$. Therefore
\[ m_B = \frac{6M-3}{3M-3} = 2 + \frac1{M-1}\,. \]
As before, if we consider the first interval $\{a,\ldots,b-1\}$ missing from $V_C$ we must have $e_0/(b-a) < m_B$.
The lack of copies of $\edg$ implies that $e_0 \geq 2(b-a) + 1$ (with the additive 1 accounting for the edge $(b-1,b)$ as usual)
and so $b-a > M - 1$. This in turn implies that $b-a \geq M$ and so the extra edge $\gamma b$ lands in the interval $[a,b)$
(since $\gamma \geq b-M$ due to the forbidden $M$-fans). As such, we cannot possibly have $b=z$ since both of its neighbors are missing
from $V_C$ (whereas $y$ must belong to $V_C$ as a distinguished vertex). Therefore, $b \leq z-1 = 3M-1$ and, since $a\geq 2$
(again recall that $1$ is distinguished) we have $b-a \leq 3M-3$. Furthermore, this entitles us to consider the vertex $b+1$.
One more edge incident to this interval is needed to secure the contradiction: if $b+1\in V_C$ it immediately provides such an edge,
since its two neighbors are $\gamma' < b$ and $\gamma'\geq \gamma \geq a$. In this case we get
\[ \frac{e_0}{b-a} \geq 2 + \frac{3}{b-a} \geq 2 + \frac{3}{3M-3} = m_B\,,\]
as required. On the other hand, if $b+1\notin V_C$ we can again look at the next missing interval $\{b+1,\ldots,b'-1\}$ where $b'\in V_C$
satisfies $b' \leq z = 3M$ (hence $b'-a \leq 3M-2$). The edges $(b-1,b),\gamma b, (b'-1,b')$ now give the required contradiction since
\[ \frac{e_1}{b'-a-1} \geq 2 + \frac{3}{b'-a-1} \geq 2 + \frac{3}{3M-3} = m_B\,.\]

\item\label{it-0-one-e} $\pi$ contains 1 copy of $ \edg$:
We have $v_\pi=2M+2$ and $e_\pi = 4M-1$, thus $|V_B\setminus I_B|=2M-2$ (again $y\neq 0,1$
since $z=2M+1$ and $y\geq z-M = M+1$)
and $|E_B|=4M-2$, and as before we obtain
\[ m_B = \frac{4M-2}{2M-2} = 2 + \frac1{M-1}\,. \]

If $\edg\notin\{\pi(a-1),\ldots,\pi(b-2)\}$ then exactly as before we have $e_0 \geq 2(b-a)+1$, implying that
$b-a\geq M$ and so, if present, the extra edge $\gamma b$ has an endpoint in $[a,b)$. At the same time, $b \leq z-1 = 2M$ and so $b-a \leq 2M-2$ and
we can conclude that $e_0/(b-a) \geq 2 + 2/(2M-2) = m_B$, contradiction. However, $\gamma b$ may not exist --- this is the case iff $\pi(b-1)=\edg$.
Again then necessarily $b < z$ and we can turn to the vertex $b+1$. If in $V_C$, the enforced pattern $\edg11$ implies that $(b-1,b+1)$ belongs to $E_B$
and, counting this in place of $\gamma b$ we get $e_0 /(b-a)\geq 2 + 2/(2M-2)$ as before. If $b+1\notin V_C$ then joining $[a,b)$ to the next missing interval $[b+1,b')$ gives $e_1/(b'-a-1) \geq 2+ 2/(b'-a-1)$ on account of the edges $(b-1,b)$ and $(b'-1,b')$, and the desired inequality $e_1/(b'-a-1)\geq m_B$ follows from the fact that $b'-a \leq z-2 = 2M-1$.

If $\edg\in\{\pi(a-1),\ldots,\pi(b-2)\}$ we can only say that $e_0 \geq 2(b-a)$ since the contribution of the edge $(b-1,b)$ is negated by
the shortage of one edge due to the $\edg$ symbol. Nevertheless, if the extra edge $\gamma b$ is such that $\gamma \in [a,b)$ then
$e_0 \geq 2(b-a)+1$ and the argument used above implies that $b-a \geq M$ and $b \leq z-1 = 2M$. We now
consider $b+1$: it has a neighbor $\gamma'$ with $\gamma \leq \gamma' < b$, in particular satisfying $\gamma' \in [a,b)$. Thus,
 if $b+1\in V_C$ we are done by having $e_0/(b-a) \geq 2 + 2/(b-a) \geq m_B$. If
$b+1\notin V_C$ then together with the next missing interval we have $e_1/(b'-a-1) \geq 2+2/(b'-a-1)$ (accounting for the edges $\gamma b$,
$(b-1,b)$ and $(b'-1,b')$ and subtracting one edge for the $\edg$ symbol), and the fact that $b'-a \leq z-2 = 2M-1$ now gives the sought
inequality $e_1/(b'-a-1) \geq 2+1/(M-1) = m_B$.

Finally, if there is an $\edg$ symbol in $\{\pi(a-1),\ldots,\pi(b-2)\}$ and yet the extra edge $\gamma b$ leaps over the interval $[a,b)$
then necessarily $\pi(a-1)=\edg$ (since having $\pi(j)=\edg$ makes $j$ a cut-vertex) and $\gamma = a-1$. This can only occur when
$b = a+1$, and since neither of its neighbors can play the role of $y$ (having $a\notin V_C$ and $\pi(\gamma)=\edg$) we can
infer that $b < z$. The symbol $\edg$ must be followed by $11$ (in lieu of any additional $\edg$ symbols in $\pi$) and so the edge
$(a,b+1)$ belongs to $\sL_\pi$ (and to $B$, since $a \neq y$). Altogether, $e_0 = 3$ on account of the neighbors $\{a-1,a+1,a+2\}$ of $a$
and so $e_0/(b-a) = 3 > m_B$.

\item\label{it-0-two-e} $\pi$ contains 2 copies of $ \edg$:
Here $v_\pi=M+3$ and $e_\pi = 2M$, leading to $|V_B\setminus I_B|=M-1$ (now $z = M+2$ and $y\geq z-M = 2$)
and $|E_B|=2M-1$. We get that
\[ m_B = \frac{2M-1}{M-1} = 2 + \frac1{M-1}\,. \]
Let $0\leq s\leq 2$ denote the number of $\edg$ symbols in $\{\pi(a-1),\ldots,\pi(b-2)\}$.

If $s=0$ then $e_0\geq 2(b-a)+1$ and so $b-a > M-1$. In this case either $\gamma \in [a,b)$ or the edge $\gamma b$ does not exist,
either way precluding $b$ from being equal to $z$,
thus $b-a \leq (z-1)-2 = M-1$, contradiction.

If $s=1$ then a-priori we only have $e_0 \geq 2(b-a)$. If the extra edge $\gamma b$ falls in $[a,b-2]$ then
$e_0/(b-a) \geq 2+1/(b-a)$ and again $b-a > M-1$, while at the same time $b<z$ and so $b-a \leq M-1$.
If the extra edge leaps over $[a,b)$ then
necessarily $\pi(a-1)=\edg$ and $b=a+1$. As before, here we must have 3 edges incident to $a$ and so $e_0/(b-a)=3 > m_B$.
If there is no extra edge $\gamma b$ then $\pi(b-1)=\edg$ and $b \leq z-1=M+1$, thus we can again appeal to $b+1$ and either get an extra edge
$(b-1,b+1)$ from the pattern $\edg11$ whenever $b+1\in V_C$ (leading to $e_0/(b-a) \geq 2+1/(b-a) \geq 2+1/(M-1) = m_B$), or join $[a,b)$ with the next missing interval $[b+1,b')$ and for an extra edge $(b'-1,b')$ (leading to $e_1/(b'-a-1) \geq 2+1/(b'-a-1) \geq m_B$).

It remains to treat $s=2$, in which case clearly $\gamma\in[a,b)$ and so $b<z$. If $b+1\in V_C$ then it has a neighbor $\gamma'$ with
$\gamma\leq \gamma' < b$ and we arrive at $e_0 \geq 2(b-a) + 1$ (accounting for $(b-1,b),\gamma b, (\gamma',b+1)$ while subtracting 2
for the $\edg$ symbols). We have $b-a \leq (z-1)-2 = M-1$ and so $e_0 / (b-a) \geq 2 + 1/(M-1) = m_B$, as needed.
The final case has $b+1\notin V_C$, and combining $\{a,\ldots,b-1\}$ with the next missing interval $\{b+1,\ldots,b'-1\}$ now gives
$e_1 \geq 2(b'-a-1) + 1$, where we counted the edges $(b-1,b),\gamma b, (b'-1,b')$ and subtracted two edges for $\edg$ symbols.
We have $b'-a \leq z-2 = M$ and $e_1/(b'-a-1) \geq 2 + 1/(M-1) = m_B$, concluding the proof.   \qedhere
\end{enumerate}
\end{proof}

Let $M\geq 3$ be some arbitrarily large constant. We aim to establish control over the
number of labeled copies of every $M$-bounded triangular ladder along the
triangle removal process as long as
\begin{equation}  \label{eq-p-delta}
  p \geq p_M := n^{-\frac{1}{2} + \frac{1}{M} }\,.
\end{equation}
Recalling~\eqref{eq-backward-ext-density} observe that $\Sc_H(p) \geq n$ for any backward extension $H$ as long as we have $p \geq p_M$, with room to spare (more precisely, $\Sc_H(p_M) \geq n^{3/2 - 1/M}$ when the corresponding ladder $\sL_\pi$ has $|\pi|=M+1$ and is even larger when $|\pi|=2M$ or $|\pi|=3M-1$).
 Since $\zeta$ as defined in~\eqref{eq-zeta-def} satisfies $\zeta^{-1} = o(\sqrt{n})$, we have $\Sc_H > \zeta^{-\sqrt{\delta}}$
for any $t(p)$ with $p\leq p_M$ and any $0 <\delta<1$. In particular, in this regime $t\leq t_H^-(\sqrt{\delta})$ for any backward extension $H$, thus we may invoke Part~\eqref{it-fine} of Theorem~\ref{thm-Psi-H-estimate} to control the variables $\Psi_{H,\varphi}$ for all the extension graphs in $\cB_M$ (each of which has at most $L=3M+1$ vertices). Further note that at $t(p_M)$ we have $\zeta(t) = n^{-1/M+o(1)}$, thus the decay rate of $\zeta$
satisfies Eq.~\eqref{eq-zeta-assumption} and we can conclude:
\begin{corollary}
  \label{cor-backward-ext-pM}
For every $M\geq 3$ there is some $0<\delta(M) < 1$ such that w.h.p.\ the following holds: For every backward extension graph $B=B_{\pi,yz}$ with $\pi\in\sL_\pi$ and any corresponding injection $\varphi$ we have $|\Psi_{B,\varphi} - \Sc_B| \leq \Sc_B \zeta^\delta$ as long as $t \leq \tau_\star$ and $p(t) \geq p_M$.
\end{corollary}

Backward extensions are only defined for outer boundary edges.  For all other edges (side boundary, initial or interior)
we will need to keep track of the following.
\begin{definition}[forward extension]\label{def:forward-ext}
Let $ \sL_\pi $ be an $M$-bounded triangular ladder and let $yz \in E( \sL_\pi)$ with $y<z$ be an edge that
is not an outer boundary edge.  The {\bf forward extension from $yz$ to $ \sL_\pi$}, denoted by $F_{\pi,yz}$, is the graph
obtained by deleting from $\sL_ \pi$ all edges between vertices in the following set $I$ of distinguished vertices:
\begin{compactitem}[\indent$\bullet$]
\item If $ yz$ is an initial edge then $ I = \{0,1, \ldots, y\} \cup \{ z\} $.
\item If $ yz$ is a side boundary edge then $ I = \{0,1, \ldots, y\} \cup \{ z\} $ in case $\pi(y-2)\neq \emph{\edg}$
and otherwise $ I = \{0,1, \ldots, y-2\} \cup \{ y,z\} $.
\item If $ yz $ is an interior edge then $ I = \{0,1, \ldots, z\} $.
\end{compactitem}
\end{definition}

Similarly to the case of backward extension graphs, estimates to within a multiplicative factor of $1+O(\zeta^\delta)$ over forward extension graphs can be readily derived from Theorem~\ref{thm-Psi-H-estimate}.
\begin{corollary}\label{cor:Psi-forward}
Let $M\geq 3$ and let $0<\delta<1$ be the constant from Corollary~\ref{cor-backward-ext-pM}. Then w.h.p.\
for every forward extension graph $F=F_{\pi,yz}$ with $\pi\in\cB_M$ and any corresponding injection $\varphi$ we have
$\Psi_{F,\varphi} = (1+O(\zeta^\delta))\Sc_F$ as long as $t \leq \tau_\star$ and $p(t) \geq p_M$. Furthermore, when $yz$ is not a side boundary edge the stronger bound $\Psi_{F,\varphi} = (1+O(\zeta))\Sc_F$ holds throughout this regime.
\end{corollary}
\begin{proof}
Suppose first that $yz$ is an \emph{interior} edge of $\sL_\pi$. We wish to count all possible mappings $\psi\in\Psi_{F,\varphi}$ by iteratively exposing the images of $z+1,\ldots,|V(\sL_\pi)|$ in $V_G$. Following Definition~\ref{def:ladders} of triangular ladders we see that every vertex $j$ is attached to either 1 or 2 vertices among its predecessors $\{0,1,\ldots,j-1\}$ (depending on whether or not we then encounter the symbol $\edg$ in $\pi$). Therefore, given $\psi(1),\ldots,\psi(j-1)$ for some $j>z$, the number of possibilities for $\psi(j)$ is dictated by either $Y_u$ for some $u \in \{\psi(1),\ldots,\psi(j-1)\}$ or by $Y_{u,v}$ for some $u,v$ in this set.
Using~\eqref{eq-codeg-ensemble-guarantee} to control these variables up to a relative error of $1+O(\zeta)$ and iterating this expansion we get
\begin{equation}
  \label{eq-Psi-F}
  \Psi_{F,\varphi} = \left(1 + O\big(|V(\sL_\pi)|\, \zeta\big)\right)\Sc_F = \left(1+O(\zeta)\right)\Sc_F\,.
\end{equation}
This argument remains valid of course when $yz$ is an \emph{initial} edge with $y=z-1$, since again we begin with the distinguished set $I=\{0,1,\ldots,z\}$ whose images are predetermined by $\varphi$. On the other hand, if $yz$ is an initial edge with $y < z-1$ then necessarily $y = z-2$ (to see this observe that $\pi(y)=\edg$ by Definition~\ref{def:non-boundary-edges} and thus, since $\pi$ does not contain the substring $\edg 1 0$, the potential forward neighbors of $y$ are confined to the set $\{y+1,y+2\}$) and similarly the neighbors of $z-1$ in $\sL_\pi$ are precisely $y,z$.
Thus, the number of possibilities for $\psi(z-1)$ is given by $Y_{\psi(y)\psi(z)}$ and from this point on we can iteratively
expose the images of $z+1,\ldots,|V(\sL_\pi)|$ as above to establish~\eqref{eq-Psi-F}.

It remains to deal with the case where $yz$ is a \emph{side boundary} edge, which as we next describe demonstrates the necessity for a
careful analysis of $f$-fans. We focus on  the case $\pi(y-2)\neq\edg$, noting that the case $\pi(y-2)=\edg$ will follow by
a nearly identical argument. Recall that the distinguished vertex set is now $I=\{0,1,\ldots,y\}\cup\{z\}$.
By Definition~\ref{def:boundary}, here we have $y<z-1$, there is at most 1 occurrence of $\edg$ in $\pi$ and the ladder $\sL_{\pi_{z-1}}$ contains an $(M-1)$-fan at $y$. Therefore, if we attempt to construct a copy $\psi \in \Psi_{F,\varphi}$ iteratively as before, upon reaching the vertex $z-1$ we will find that its neighbors in $\sL_\pi$ are precisely $\{y,z-2,z\}$ and so, once the images of these vertices are preset, $\psi(z-1)$ is given by some variable $N_{uvw}$ which we may very well have no control over (our approximation within a relative error of $1+O(\zeta)$ for these variables ceases to be valid once $p<n^{-1/3+o(1)}$ as the triple co-neighborhoods become merely poly-logarithmic in size).

To remedy this issue we will rely on Theorem~\ref{thm-Psi-H-estimate}. Let $K$ be the subextension of $F$ that has $V_K=\{0,1,\ldots,z\}$. Since $I_K=I_F=\{0,1,\ldots,y\}\cup\{z\}$, each vertex $j\in \{y+2,\ldots,z-1\}$ has $\{y,j-1,j+1\}$ as its neighbors in $K$.
We can temporarily think of these edges as being directed by letting each such vertex $j$ send one edge to $j+1$ and another to $y$,
while the vertex $y+1$ does the same and sends an additional edge to some $\ell<y$. With this view it becomes clear that $e_K = 2(v_K-\iota_K)+1$ and furthermore that $K$ is a (strictly) balanced extension graph.
To conclude the proof, note that $v_K-\iota_K = z-y-1$ whereas $z=y+M$ since $z$ is the final vertex of an $(M-1)$-fan at $y$ (see Eq.~\eqref{eq-side-boundary-yz-relation} bearing in mind that we are in the case $\pi(y-2)\neq \edg$).
Plugging this into the expression for $v_K-\iota_K$ we obtain that $\Sc_K = n^{M-1} p^{2M-1}$.
Hence, as long as $p\geq p_M$ we have $\Sc_K \geq n^{(3M-2)/(2M)}$. This translates into a lower bound of at least $n^{7/6}$ by the assumption $M\geq 3$ and, since $\zeta^{-1}=o(\sqrt{n})$, this implies in turn that $t < t^{-}_K(\sqrt{\delta})$ for any $0<\delta<1$. We can therefore appeal to Part~\eqref{it-fine} of Theorem~\ref{thm-Psi-H-estimate} and get that
\[ \left|\Psi_{K,\varphi} - \Sc_K\right| \leq \Sc_K \zeta^\delta\]
where $\delta=\delta(M)$ is precisely the same constant as in Corollary~\ref{cor-backward-ext-pM}.
From this point we can complete $K$ into a copy of $F$ by iteratively enumerating over the options for $z+1,\ldots,|V(\sL_\pi)|$ as before, each time at a cost of a multiplicative factor of $1+O(\zeta)$. Altogether we deduce that
\[ \left|\Psi_{F,\varphi} - \Sc_F \right| \leq (\zeta^\delta + O(\zeta))\Sc_F = O(\zeta^\delta)\Sc_F\,,\]
as required. Essentially the same argument extends this estimate also to the case of a side boundary edge $yz$ with $\pi(y-2)=\edg$.
Now $I_F = \{0,1,\ldots,y-2\}\cup\{y,z\}$ yet $z=y+M-1$ (see~\eqref{eq-side-boundary-yz-relation}) and a closer look reveals that the vertices
$y-1,y+1,\ldots,z$ forming the $(M-1)$-fan play an identical role to that assumed by vertices $y+1,\ldots,z$ in the previous case (e.g., each can
be viewed as sending one edge to its successor in this list and another to $y$, the first vertex $y-1$ sends one additional to $\ell=y-2$
and so on).
Therefore, the subextension $K$ is again strictly balanced and has $e_K = 2(v_K-\iota_K)+1$.
Having $v_K-\iota_K = z-y$ and $z=y+M-1$ we again get $\Sc_K = n^{M-1} p^{2M-1} \geq n^{7/6}$ for all $p\geq p_M$, and the rest of the argument
holds unmodified.
\end{proof}

\section{Tracking triangular ladders}\label{sec:tracking-ladders}
\subsection{Expected change}\label{sec:ladders-expected}

For a labeled undirected graph $H$ and $u,v\in V_G$ write $\Psi_{H,uv}=\Psi_{H,uv}(t)$
to denote the random variable $\Psi_{H,\varphi}$ where $\varphi$ maps
$0\mapsto u$ and $1\mapsto v$. If $H=\sL_\pi$ we will use the abbreviated form $\Psi_{\pi,uv}$, and as usual we use $v_\pi$ and $e_\pi$ to denote $|V(\sL_\pi)|$ and $|E(\sL_\pi)|$, respectively.
Our aim in this subsection is to estimate the expected change in variables $ \Psi_{\pi,uv} $
within one step of the process for any $M$-bounded triangular ladder.

\begin{theorem}\label{thm:ladders-expected}
Let $M\geq 3$ and let $0<\delta<1$ be the constant from Corollaries~\ref{cor-backward-ext-pM} and~\ref{cor:Psi-forward}.
For any $\pi\in \cB_M$ let $\pi^-$ denote its $(|\pi|-1)$-prefix and define
\[X_{\pi,uv} = \Psi_{ \pi,uv} - \frac{\Sc_{\pi}}{\Sc_{\pi^-}}\Psi_{ \pi^-,uv} \]
where $ \Psi_{\emptyset,uv} := 1 $. Let $\zeta  = n^{-1/2} p^{-1} \log n$ and for some absolute constant $\kappa>0$ let
\[\tau_{\textsc{q}}^* = \tau_{\textsc{q}}^*(\kappa) = \min\left\{ t : \left| Q/(\tfrac16 n^3 p^3)-1\right| \geq \kappa\, \zeta^2 \right\}\,.\]
Let $i$ be some time such that $t(i) \leq \tau_{\textsc{q}}^*$ and $p(i) \geq p_M$.
Suppose that $\Psi_{\pi,uv}=(1+o(1))\Sc_\pi$ holds at time $i$ for all $\pi\in\cB_M$ and all $u,v\in V_G$. Then
\begin{equation}
  \label{eq-Delta-Psi(pi,uv)}
  \E[\Delta X_{\pi,uv}\mid\cF_i] = -(6-o(1))\frac{e_\pi}{n^2 p} X_{\pi,uv} + \Xi_0(\pi) + \Xi_1(\pi) + \frac{O\left(\zeta^{1+\delta}\right)}{n^2 p}\Sc_{\pi}
\end{equation}
for every $\pi\in\cB_M$ and every $u,v\in V_G$, where the correction term $\Xi_1(\pi)$ is given by
\[ \Xi_1(\pi)= -\frac{6+o(1)}{n^3 p^3} X_{\pi\circ 1,uv}\; \one_{\big\{\mbox{$(|\pi|,|\pi|+1)$ is an initial/interior edge}\big\}}\,,\]
whereas if there exists $y<|\pi|$ such that $(y,|\pi|+1)\in E(\sL_\pi)$ then
\begin{align*}
\Xi_0(\pi) = &-\frac{6+o(1)}{n^2 p} X_{\pi,uv} \;\one_{\big\{\mbox{$(y,|\pi|+1)$ is an initial edge}\big\}} \\
& -\frac{6+o(1)}{n^3 p^3} X_{\pi\circ 0,uv} \; \one_{\big\{ \mbox{$(y,|\pi|+1)$ is an interior edge}\big\}}\\
&-\frac{6+o(1)}{n^3 p^3} \frac{\Sc_{\pi}}{\Sc_{\pi_{y-1}\circ\emph{\edg}}} X_{\pi_{y-1}\circ\emph{\edg} 1,uv} \;  \one_{\big\{\mbox{$(y,|\pi|+1)$ is a side boundary edge, $\pi(y-2)\neq\emph{\edg}$}\big\}}\\
&-\frac{6+o(1)}{n^3 p^3} \frac{\Sc_{\pi}}{\Sc_{\pi_{y-2}\circ\emph{\edg}}} X_{\pi_{y-2}\circ\emph{\edg} 1,uv} \;  \one_{\big\{\mbox{$(y,|\pi|+1)$ is a side boundary edge, $\pi(y-2)=\emph{\edg}$}\big\}}
\end{align*}
and otherwise $\Xi_0(\pi)=0$.
\end{theorem}

Our proof will crucially rely on the properties of triangular ladders in order to express the number of copies of $\sL_\pi$ that rest on a given edge in terms of the total number of other triangular ladders.
To demonstrate this we use the following simple lemma of~\cite{BB} whose proof is provided for completeness.
\begin{lemma}
\label{lem:jpatrick}
Suppose $ (x_i)_{i\in I} $ and $ (y_i)_{i \in I}$ are real numbers so that
$|x _i - x| \le \delta_1$ and $|y_i-y| < \delta_2$ for some $x,y\in\R$ and every $i \in I$.  Then
\[ \bigg| \sum_{i \in I} x_i y_i -  \frac{1}{|I|} \Big( \sum_{i \in I} x_i \Big)
\Big( \sum_{i \in I} y_i \Big) \bigg|  \le  2 |I| \delta_1 \delta_2 \,.\]
\end{lemma}
\begin{proof}
By the triangle inequality, $\left| \sum_{i \in I} (x_i - x)(y_i - y) \right| \leq |I|\delta_1 \delta_2$,
and at the same time it is easy to see that $\sum_{i\in I} (x_i-x)(y_i-y)$ is precisely equal to
\begin{align*}
\sum_{i\in I} x_i y_i - \frac1{|I|}\Big(\sum_i x_i\Big)\Big(\sum_i y_i\Big) +
|I|\bigg(\frac{\sum_i x_i}{|I|}-x\bigg)\bigg(\frac{\sum_i y_i}{|I|}-y\bigg)\,.
\end{align*}
The hypothesis on the $x_i$'s and $y_i$'s implies that the last term above is at
most $|I|\delta_1 \delta_2$ in absolute value, from which the desired result follows.
\end{proof}

Prior to commencing with the proof of Theorem~\ref{thm:ladders-expected} we wish to
illustrate the role that the backward/forward extensions play in the analysis of the variables $\Psi_{\pi,uv}$.

\begin{example}
  [expected one-step change in $\Psi_{\edg\edg}$]
In general, in order to estimate the one-step change in $\Psi_{\pi,uv}$ we examine the edges of the given ladder and for each one consider the probability of hitting this edge (as part of some labeled copy of the given ladder rooted at $u,v\in V_G$) by the current triangle as well as the number of copies of the ladder that it participates in.
For $\sL_{\edg\edg}$ we write separate summations for the two edges $12$ and $23$ to get
\begin{align*}
  \E [\Delta \Psi_{\edg \edg,uv}\mid\cF_i] &= -\sum_{a \in N_u} \frac{Y_{u,a} (Y_a-1)}{Q} -
  \sum_{
  \substack{ a\in N_u \\ b \in N_a \setminus \{u\} }
  } \frac{Y_{a,b} - \one_{\{ua \in E\}}}Q  \\
  &= -\sum_{a \in N_u} \frac{Y_{u,a} (Y_u-1)}{Q} - \frac{ \Psi_{\edg \edg 1,uv}}Q  + O\left(\frac{1}{n^3 p^3}\right)\Sc_{\edg\edg}\,,
\end{align*}
where the indicators $\one_{\{ua\in E\}}$ treat the double counting of a triangle that contains both these edges, and the last equality rewrote their sum as $O(\Sc_{\edg\edg}/(n^3p^3))$ via the hypothesis on the degrees and $Q$. We now apply Lemma~\ref{lem:jpatrick} with $ I = N_u$, the sequence $ x_{a} = Y_{u,a}$ with $ x = np^2$ and $\delta_1 =  np^2  O(\zeta)$
and the sequence $ y_{a} = Y_a$ with $y = np$ and $ \delta_2 = np O(\zeta)$ (the variables $\Psi_{\pi,uv}$
count labeled graphs, thus for instance $\sum_{a\in N_u} Y_{u,a} = \Psi_{\edg 1,uv}$). Since $|I|\delta_1\delta_2/Q = O(\zeta^2 p)$ while $\Sc_{\edg\edg} = n^2p^2$ this gives
\begin{align*}
  \E [\Delta \Psi_{\edg \edg,uv}\mid\cF_i] &= -\frac{ \Psi_{\edg 1,uv}\Psi_{\edg \edg,uv}}{Y_u Q} - \frac{ \Psi_{\edg \edg 1,uv}}Q
+ \frac{O \left(\zeta^2\right)}{n^2p}\Sc_{\edg \edg}  \,.
\end{align*}
(The $O(\frac{1}{n^3 p^3})$-term was absorbed by the $O(\frac{\zeta^2}{n^2 p})$-term as the latter has order $\frac{\log^2 n}{n^3 p^3}$ by Def.~\eqref{eq-zeta-def}.)
\end{example}

We stress that many of our applications of Lemma~\ref{lem:jpatrick} have this
common form.
We are computing the expected change in
some variable $ \Psi_{\pi,uv} $ due to the loss of a
particular edge in the triangular ladder.  We compute this
as a sum over
such edges $yz \in E(\sL_\pi)$. For each one we break up the ladder $\sL_\pi$ into a smaller subgraph $K$ that contains $u,v$ and the edge $yz$, then proceed to \begin{inparaenum}
  [(i)]
  \item count the number of labeled copies of $K$ in $G$ rooted at $u,v$ (thereby yielding an estimate on the probability of removing such an edge in the current round), denoted by some $\Psi_{K,uv}$, and
  \item  count the number of extensions from $K$ to $\sL_\pi$ in $G$ (giving the total number of copies of $\sL_\pi$ eliminated by removing the edge under consideration), denoted by some $\Psi_{\sL_\pi / K}$.
\end{inparaenum}
Observe that, by definition, $\sum_{\psi\in\Psi_{K,uv}} \Psi_{\sL_\pi / K,\psi} = \Psi_{\pi,uv}$.
Moreover, given the injection $\psi:K\to V_G$ the probability to remove $\psi(y)\psi(z)$ is equal to $Y_{\psi(y)\psi(z)}/Q$, thus we arrive at the expression
\[\sum_{\psi\in \Psi_{K,uv}} \frac{Y_{\psi(y)\psi(z)}}Q \Psi_{\sL_\pi/K,\psi}\,.\]
Applying Lemma~\ref{lem:jpatrick} on this summation we now get a main term which is a product of $\Psi_{\pi,\psi}$ (the result of $\sum_\psi \Psi_{\sL_\pi/K,\psi}$) and $\Psi_{K',uv}$, where $K'$ is obtained from $K$ by attaching to it a new vertex incident to the edge $yz$ (the result of $\sum_\psi Y_{\psi(y)\psi(z)}$). We now see the reason behind our definition of the family of triangular ladders, since if $K$ is a triangular ladder then $K'$ will commonly also belong to our family of tracked variables. As an error term in this analysis we will sustain the approximation error in $Y_{\psi(y)\psi(z)}/Q$, i.e., an $O(\zeta/(n^2 p))$ term, multiplied by $\Psi_{K,uv}$ and then by the approximation error for $\Psi_{\sL_\pi/K}$. Usually we will choose $K$ so that $\sL_\pi/K$ will be the forward extension $F_{\pi,yz}$ introduced in the previous section, whose corresponding variable $\Psi_{F,\varphi}$ was estimated there to within an error of $O(\zeta^\delta)\Sc_F$ uniformly over the injection $\varphi$. Overall, we get the desired
$ O(\zeta^{1+\delta}/(n^2 p))\Psi_{\pi,uv} $ error-term.

We are now ready to prove the main result of this subsection.
\begin{proof}[\emph{\textbf{Proof of Theorem~\ref{thm:ladders-expected}}}]
By hypothesis we control the variables $Y_u=\Psi_{\edg,uv}$, $Y_{u,v}=\Psi_{1,uv}$ and $Q$ to within a multiplicative factor of $1 + O(\zeta)$.
Via Corollaries~\ref{cor-backward-ext-pM} and~\ref{cor:Psi-forward} this guarantees estimates to within a factor of $1+O(\zeta^\delta)$ in backward and forward extension variables $B_{\pi,yz}$ and $F_{\pi,yz}$.
(Forward extensions of initial/interior edges feature a better approximation of $1+O(\zeta)$.)

As mentioned above, we will analyze the different possible edges $yz\in E(\sL_\pi)$ whose images in $E(G)$ (under any mapping $\psi\in \Psi_{\pi,uv}$) may be removed in the upcoming step, and the possible impact this event would have on $X_{\pi,uv}$.

Given $\cF_{i}$, for each $\psi\in \Psi_{\pi,uv}$ we define the events
\begin{equation}
  \label{eq-cE-psi-def}
  \cE_{\psi} = \bigcup\big\{ \cE_\psi^{yz} : yz \in E(\sL_\pi)\big\} \quad\mbox{where}\quad
\cE_{\psi}^{yz} = \big\{ \psi(y)\psi(z)\mbox{ is removed at time $i$}\big\}\,,
\end{equation}
and for each edge $yz\in E(\sL_\pi)$ further set
\[ \Delta_{yz} \Psi_{\pi,uv} = -\sum_{\psi\in \Psi_{\pi,uv}} \one_{\cE_{\psi}^{yz}}\,.\]
According to these definitions, $\Delta \Psi_{\pi,uv} = - \sum_{\psi\in \Psi_{\pi,uv}} \one_{\cE_\psi}$. At the same time, we crucially have
\begin{equation}
  \label{eq-cE-approx}
  \P(\cE_{\psi}\mid\cF_i) = \sum_{yz\in E(\sL_\pi)} \P(\cE_\psi^{yz}\mid\cF_i) + O(1/Q)
\end{equation}
since $\P(\cE_\psi^{yz}\,,\,\cE_\psi^{y'z'}\mid\cF_i) \leq 1 / Q$ for any distinct edges $yz,y'z'\in E(\sL_\pi)$ (the event that both are removed identifies a single triangle in $G$), and there are $O(e_\pi^2) = O(1)$ such pairs. It follows that
\begin{align}
  \E [\Delta \Psi_{\pi,uv}\mid\cF_i] &=  \sum_{yz\in E(\sL_\pi)} \E[\Delta_{yz}\Psi_{\pi,uv}] + O(\Psi_{\pi,uv}/Q) =  \sum_{yz\in E(\sL_\pi)} \E[\Delta_{yz}\Psi_{\pi,uv}] + \frac{o(\zeta^2)}{n^2 p}\Psi_{\pi,uv}\,,
    \label{eq-Dyz-Psi}
\end{align}
using the fact that $Q \asymp n^3 p^3$ whereas $\zeta^{-2} n^2 p \asymp n^3 p^3 \log^2 n$.
Our ultimate goal in this theorem is to estimate the one-step change in $X_{\pi,uv} = \Psi_{\pi,uv} - (\Sc_{\pi}/\Sc_{\pi^-})\Psi_{\pi^-,uv}$.
As usual we write the change from time $i$ to time $i+1$ as
\[ \Delta X_{\pi,uv} = \Delta \Psi_{\pi,uv} - \frac{\Sc_{\pi}(i+1)}{\Sc_{\pi^-}(i+1)}\Delta \Psi_{\pi^-,uv}
- \Psi_{\pi^-,uv} \Delta\left(\frac{\Sc_{\pi}}{\Sc_{\pi^-}}\right)\,.\]
Since we have $\Sc_{\pi}/\Sc_{\pi^-} = n p^\theta$ where $\theta = 1$ if the last symbol in $\pi$ is $\edg$ and otherwise $\theta =2$, the change in the last term is deterministically given by
\[ \Psi_{\pi^-,uv} \Delta\left(\frac{\Sc_{\pi}}{\Sc_{\pi^-}}\right) = -\theta n p^{\theta-1} \left(\frac{6}{n^2}\right)\left(1 - O\Big(\frac{1}{n^2 p}\Big)\right) \Psi_{\pi^-,uv} = -\frac{6\theta-o(\zeta^3)}{n^2 p} \frac{\Sc_{\pi}}{\Sc_{\pi^-}}\Psi_{\pi^-,uv}\,,
\]
relying in the last equality on the fact that $\zeta^{-3} = n^{\frac32-o(1)}p^3 = o(n^2 p)$. Adding the similar deterministic estimate $(\Sc_{\pi}/\Sc_{\pi^-})(i) = (1-o(\zeta^3))(\Sc_\pi/\Sc_{\pi^-})(i-1)$  we can rewrite the equation for $\Delta X_{\pi,uv}$ as
\begin{equation}
  \label{eq-DX}
  \Delta X_{\pi,uv} = \Delta \Psi_{\pi,uv} - (1+o(\zeta^3))\frac{\Sc_{\pi}}{\Sc_{\pi^-}}\Delta \Psi_{\pi^-,uv}
 +\frac{6\theta-o(\zeta^3)}{n^2 p} \frac{\Sc_{\pi}}{\Sc_{\pi^-}}\Psi_{\pi^-,uv}\,.
\end{equation}
Finally, bearing in mind that $\theta$ counts the number of edges that are incident to the vertex $|\pi|+1$, we can break up the expected change in $\Delta \Psi_{\pi,uv}$ and $\Delta \Psi_{\pi^-,uv}$ as per equation with~\eqref{eq-Dyz-Psi} and get
\begin{align*}
  \E[\Delta X_{\pi,uv}\mid\cF_i] &=
  \sum_{\substack{yz\in E(\sL_\pi) \\ y<z \leq |\pi|}} \left(\E\left[\Delta_{yz}\Psi_{\pi,uv}\mid\cF_i\right] - (1-o(\zeta^3))\frac{\Sc_{\pi}}{\Sc_{\pi^-}}\E\left[\Delta_{yz}\Psi_{\pi^-,uv}\mid\cF_i\right]\right) \nonumber\\
 &+ \sum_{\substack{yz\in E(\sL_\pi) \\ y<z = |\pi|+1}} \left(\E\left[\Delta_{yz}\Psi_{\pi,uv}\mid\cF_i\right] + \frac{6-o(\zeta^3)}{n^2 p}\frac{\Sc_{\pi}}{\Sc_{\pi^-}}\Psi_{\pi^-,uv} \right)
   + \frac{o(\zeta^2)}{n^2 p}\Psi_{\pi,uv}\,.
\end{align*}
(Notice that the first sum goes over all edges of $\sL_{\pi^-}$, leaving $\theta$ edges of $E(\sL_\pi)\setminus E(\sL_{\pi^-})$ for the second sum. This allows us to divide the last term of~\eqref{eq-DX} to each of these latter summands.)
By assumption we have $(\Sc_{\pi}/\Sc_{\pi^-})\Psi_{\pi^-,uv}=(1+o(1))\Psi_{\pi,uv}$ and so the $o(\zeta^3)$ error terms above are easily absorbed in the last $o(\zeta^2)$ term to give the following:
\begin{align}
  \E[\Delta X_{\pi,uv}\mid\cF_i] &=
  \sum_{\substack{yz\in E(\sL_\pi) \\ y<z \leq |\pi|}} \left(\E\left[\Delta_{yz}\Psi_{\pi,uv}\mid\cF_i\right] - \frac{\Sc_{\pi}}{\Sc_{\pi^-}}\E\left[\Delta_{yz}\Psi_{\pi^-,uv}\mid\cF_i\right]\right) \nonumber\\
 &+ \sum_{\substack{yz\in E(\sL_\pi) \\ y<z = |\pi|+1}} \left(\E\left[\Delta_{yz}\Psi_{\pi,uv}\mid\cF_i\right] + \frac{6}{n^2 p}\frac{\Sc_{\pi}}{\Sc_{\pi^-}}\Psi_{\pi^-,uv} \right)
   + \frac{o(\zeta^2)}{n^2 p}\Psi_{\pi,uv}\,.
  \label{eq-E-DX}
\end{align}
This identity becomes particularly useful in light of the fact that if $yz\in E(\sL_{\pi^-})$ then $yz$ has the same edge classification in both $\sL_\pi$ and $\sL_{\pi^-}$. (Indeed, the criterion for classifying $yz$ as a boundary edge is a function of $\pi_{z-1}$ and the values of $y,z$, whereas initial edges are identified by $\pi(y)$.) We now proceed to estimate the appropriate quantities in~\eqref{eq-E-DX} for each $yz\in E(\sL_\pi)$ depending on the type of the edge under consideration.

\begin{itemize}[\noindent$\bullet$]
\item \emph{Initial edges}:
If $yz$ is an initial edge with $y<z$ then $\pi(y)=\edg$ and $z\in\{y+1,y+2\}$.
Consider the case $z=y+1$ and examine $\pi_y$, the length-$y$ prefix of $\pi$.
We will count the number of (labeled, rooted at $u,v$) copies of $\sL_{\pi_y}$ to estimate the number of edges that can play the role of $yz$ in $G$. Given some mapping $ \varphi \in \Psi_{\pi_y,uv} $, in the event that we remove the edge $\varphi(y) \varphi(z) $ the number of affected copies $\psi\in\Psi_{\pi,uv}$ (i.e., those agreeing with $\varphi$ on $y,z$) is simply the number of forward extensions $F=F_{\pi, yz}$ from
$ \varphi$ to the full copy of $\sL_\pi$. Thus,
\[ \E[\Delta_{yz}\Psi_{\pi,uv}\mid\cF_i] =  -\sum_{\varphi\in\Psi_{\pi_y,uv}}\frac{Y_{\varphi(y)\varphi(z)} \Psi_{F,\varphi}}{Q}\,.\]
For a fixed $ \varphi \in \Psi_{\pi_y,uv} $, the probability of removing the image of $yz$ is $ Y_{\varphi(y) \varphi(z)}/Q $, whereas the sum over
$ \varphi \in \Psi_{ \pi_y,uv}$ of $ Y_{\varphi(y) \varphi(z)} $ is $ \Psi_{ \pi_y \circ 1,uv} $. We claim that $\pi_y\circ1\in\cB_M$.
For $y=1$ indeed $\edg1 \in \cB_M$. For $y>1$, clearly $\pi$ has either $\mathfrak{C}_\pi \in \{1,2\}$ copies of $\edg$.
Going back to the definition of $\cB_M$ we see that, as $\pi\in\cB_M$, necessarily $|\pi|\leq \ell$
where $\ell=2M$ if $ \mathfrak{C}_\pi=1$ and $\ell=M+1$ if $\mathfrak{C}_\pi=2$. Moreover, $y \neq \ell$ so as to avoid $\pi(\ell)=\edg$,
thus $y < \ell$ and it follows that $\pi_y \circ 1$ has length at most $\ell$, at most $\mathfrak{C}_\pi$ copies of $\edg$
and a last symbol of $1$, thus $\pi_y\circ 1\in\cB_M$ as claimed.

As stated before, $\sum_{\varphi } \Psi_{F,\varphi}$ gives $ \Psi_{\pi,uv} $ itself.
Altogether, an application of Lemma~\ref{lem:jpatrick} gives
\begin{equation}\label{eq-initial}
\E[\Delta_{yz}\Psi_{\pi,uv}\mid\cF_i] = - \frac{ \Psi_{\pi_y \circ 1,uv}}{ \Psi_{\pi_y,uv} Q}\Psi_{\pi,uv} + \frac{O \left( \zeta^2\right)}{ n^2p} \Psi_{\pi,uv}\,,
\end{equation}
where in the notation of that lemma we substituted $|I|=\Psi_{\pi_y,uv}$, $\delta_1 = O(\zeta)/(n^2p)$ and $\delta_2 = O(\zeta)\Sc_F$ uniformly over the injection $\varphi$ thanks to Corollary~\ref{cor:Psi-forward} (note that $yz$ is not a side boundary edge), thus $|I|\delta_1 \delta_2 = O(\zeta^2/(n^2p)) \sum_{\varphi \in\Psi_{\pi_y,uv}}\Psi_{F,\varphi} = O(\zeta^2/(n^2p))\Psi_{\pi,uv}$.

Now consider the case of an initial edge $yz$ where $ z = y+2$. If we wanted to use argument analogous to the one above we would examine the ladder $\sL_{\pi_{z-1}}$. However, in this case the sum over $\varphi\in\Psi_{\pi_{z-1},uv}$ of $Y_{\varphi(y)\varphi(z)}$ would give $\Psi_{\pi_{z-1}\circ 0,uv}$, and $\pi_{z-1}\circ 0\notin\cB_M$ since $\pi(y)=\edg$ and the substrings $\edg10$ and $\edg0$ are forbidden. A closer look at $\sL_{\pi_{z-1}}$ reveals that the vertices $y+1,z=y+2$ play identical roles in this graph --- they are connected to one another and each shares an edge with $y$. By switching their labels we can examine the graph $\sL_{\pi_y}$ where the role of $z$ will be played by the isomorphic $y+1$. This gives
\[ \E[\Delta_{yz}\Psi_{\pi,uv}\mid\cF_i] =  -\sum_{\varphi\in\Psi_{\pi_y,uv}}\frac{Y_{\varphi(y)\varphi(y+1)} \Psi_{F_{\pi,yz},\varphi'}}Q \,,\]
where $\varphi'$ above stands for the composition of $\varphi$ and the map $z\mapsto y+1$. One should note that the above equality used the fact that $F_{\pi,yz}$ has the distinguished vertex set $\{1,\ldots,y\}\cup\{z\}$ thus vertex $y+1$ that was not mapped in $\varphi'$
is indeed added to form a full copy of $\sL_\pi$.
Henceforth the analysis of this summation can be carried out exactly as in the above case of $z=y+1$, and we conclude that Eq.~\eqref{eq-initial} holds also for the case of $z=y+2$.

Having established~\eqref{eq-initial} for all initial edges, we now wish to estimate the contribution to $\E[\Delta X\mid\cF_i]$ for such edges. This will depend on $y,z$ as per the following three cases:
\begin{enumerate}
  [(1)]
  \item $y<z \leq |\pi|$: From~\eqref{eq-initial}, applied both to $\Psi_{\pi,uv}$ and to $\Psi_{\pi^-,uv}$, we have
  \begin{align}
\E\left[\Delta_{yz}\Psi_{\pi,uv}\mid\cF_i\right] - \frac{\Sc_{\pi}}{\Sc_{\pi^-}}\E\left[\Delta_{yz}\Psi_{\pi^-,uv}\mid\cF_i\right]
&= -\frac{ \Psi_{\pi_y \circ 1,uv}}{ \Psi_{\pi_y,uv} Q} X_{\pi,uv} + \frac{O \left( \zeta^2\right)}{ n^2p} \Psi_{\pi,uv} \nonumber\\
&= -\frac{ 6-o(1)}{ n^2 p} X_{\pi,uv} + \frac{O \left( \zeta^2\right)}{ n^2p} \Psi_{\pi,uv} \,,\label{eq-initial-DX-1}
\end{align}
where the last equality applied our hypothesis that $\Psi_{\sigma,uv} = (1+o(1))\Sc_{\sigma}$ for any $\sigma \in \cB_M$ (namely for $\pi_y$ and $\pi_y\circ 1$).

\item $y=|\pi|$ and $z=|\pi|+1$: Here $\pi_y=\pi$, thus the main term in the right-hand-side of~\eqref{eq-initial} simplifies into $-\Psi_{\pi\circ1,uv}/Q$. Writing $\Psi_{\pi\circ1,uv} = X_{\pi\circ1,uv} + (\Sc_{\pi\circ1}/\Sc_{\pi})\Psi_{\pi,uv}$ we get
\begin{align}
  \E\left[\Delta_{yz}\Psi_{\pi,uv}\mid\cF_i\right] + \frac{6}{n^2 p}\frac{\Sc_{\pi}}{\Sc_{\pi^-}}\Psi_{\pi^-,uv} &=
  -\frac{ X_{\pi\circ 1,uv} + np^2 \Psi_{\pi,uv}}{ Q}
  +  \frac{6}{n^2 p}\frac{\Sc_{\pi}}{\Sc_{\pi^-}}\Psi_{\pi^-,uv}
  + \frac{O \left( \zeta^2\right)}{ n^2p} \Psi_{\pi,uv} \nonumber\\
  &= -\frac{6}{n^2 p}X_{\pi,uv} -\frac{6-o(1)}{n^3 p^3}X_{\pi\circ1,uv} + \frac{O \left( \zeta^{2}\right)}{ n^2p} \Psi_{\pi,uv}\,.\label{eq-initial-DX-2}
\end{align}
(The final error term above absorbed a $1+O(\zeta^2)$ factor due to the approximation of $Q$.)

\item $y=|\pi|-1$ and $z=|\pi|+1$: Here $\pi_y=\pi^-$ and $\pi_y\circ 1 = \pi$ (we must have $\pi(y+1)=1$ to allow such an edge $yz$). The identity $\Psi_{\pi,uv}=X_{\pi,uv}+(\Sc_{\pi}/\Sc_{\pi^-})\Psi_{\pi^-,uv}$ now yields
\begin{align}
  \E\left[\Delta_{yz}\Psi_{\pi,uv}\mid\cF_i\right] + \frac{6}{n^2 p}\frac{\Sc_{\pi}}{\Sc_{\pi^-}}\Psi_{\pi^-,uv}
  &= -\frac{ X_{\pi,uv}+np^2\Psi_{\pi^-,uv}}{ \Psi_{\pi^-,uv} Q} \Psi_{\pi,uv}
  +  \frac{6}{n^2 p}\frac{\Sc_{\pi}}{\Sc_{\pi^-}}\Psi_{\pi^-,uv}
  + \frac{O \left( \zeta^2\right)}{ n^2p} \Psi_{\pi,uv} \nonumber\\
  &= -\frac{12+o(1)}{n^2 p}X_{\pi,uv} +   \frac{O \left( \zeta^{2}\right)}{ n^2p} \Psi_{\pi,uv}\,,\label{eq-initial-DX-3}
\end{align}
where we applied the facts $Q=(\frac16+O(\zeta^2))n^3p^3$ and
$\Psi_{\pi ,uv}/\Psi_{\pi^-,uv} = (1+o(1))np^2$.
\end{enumerate}

\item \emph{Interior edges:} First consider an interior edge $yz$ where $y < z\leq |\pi|$.
 Here we let $\sL_{\pi_{z-1}}$ be the intermediate graph containing $yz$. As before, for each $ \varphi \in \Psi_{\pi_{z-1},uv} $ the probability of choosing a triangle that contains $ \varphi(y) \varphi(z) $ is $ Y_{ \varphi(y) \varphi(z) } /Q$, and the number
of extensions in $ \Psi_{\pi,uv} $ that are lost with this choice is the number
of forward extensions from the fixed $ \varphi $ to $ F=F_{\pi, yz} $. Therefore, as in the argument for initial edges we get
\[
\E[\Delta_{yz} \Psi_{\pi,uv}\mid\cF_i] = -\sum_{ \varphi \in \Psi_{\pi_{z-1},uv}} \frac{ Y_{\psi(y) \psi(z) } \Psi_{ F, \varphi } }{Q}
=
- \frac{ \Psi_{ \pi_{z-1}\circ b_{yz},uv}}{ \Psi_{ \pi_{z-1},uv} Q}\Psi_{\pi,uv} + \frac{O(\zeta^2)}{n^2 p}\Psi_{\pi,uv}\,,
\]
where $ b_{yz} = 0 $ if $ y < z-1 $ whereas $ b_{yz} = 1 $ if $ y =z-1 $. Observe that $\pi_{z-1}\circ b_{yz} \in \cB_M$ by the characterization of $yz$ as an interior edge, hence we have control over $\Psi_{\pi_{z-1}\circ b_{yz},uv}$ and so
\begin{align}
\E\left[\Delta_{yz}\Psi_{\pi,uv}\mid\cF_i\right] - \frac{\Sc_{\pi}}{\Sc_{\pi^-}}\E\left[\Delta_{yz}\Psi_{\pi^-,uv}\mid\cF_i\right]
&= -\frac{\Psi_{ \pi_{z-1}\circ b_{yz},uv}}{ \Psi_{ \pi_{z-1},uv} Q} X_{\pi,uv} + \frac{O \left( \zeta^2\right)}{ n^2p} \Psi_{\pi,uv} \nonumber\\
&= -\frac{ 6-o(1)}{ n^2 p} X_{\pi,uv} + \frac{O \left( \zeta^2\right)}{ n^2p} \Psi_{\pi,uv} \,.\label{eq-interior-1}
\end{align}

Now consider an interior edge $yz$ with $y < z = |\pi|+1 $. Again let $b_{yz} =0$ if $ y < z-1$ and
$ b_{yz}= 1 $ if $ y = z-1$. Since here $\pi_{z-1}=\pi$ there is no need to enumerate over extensions from an intermediate graph to $\sL_\pi$ and we simply have, due to our $1+O(\zeta^2)$ control over $Q$,
\begin{align*}
\E[\Delta_{yz} \Psi_{\pi,uv}\mid\cF_i] &= -\frac{ \Psi_{\pi\circ b_{yz},uv}}{Q} = -\frac{ X_{\pi\circ b_{yz},uv} + np^2\Psi_{\pi,uv}}{Q} = -(6+O(\zeta^2))\frac{X_{\pi\circ b_{yz},uv}}{n^3 p^3} - \frac{6+O(\zeta^2)}{n^2 p}\Psi_{\pi,uv}\,.
 \end{align*}
In particular,
\begin{align}
  \E\left[\Delta_{yz}\Psi_{\pi,uv}\mid\cF_i\right] + \frac{6}{n^2 p}\frac{\Sc_{\pi}}{\Sc_{\pi^-}}\Psi_{\pi^-,uv} &=   \E\left[\Delta_{yz}\Psi_{\pi,uv}\mid\cF_i\right]
  +  \frac{6}{n^2 p}(\Psi_{\pi,uv}-X_{\pi,uv}) \nonumber\\
  &= -\frac{6}{n^2 p}X_{\pi,uv} -\frac{6-o(1)}{n^3 p^3}X_{\pi\circ b_{yz},uv} +   \frac{O \left( \zeta^{2}\right)}{ n^2p} \Psi_{\pi,uv}\,.\label{eq-interior-2}
\end{align}
Observe that this last equation provides a large self-correcting
term, and yet we see that errors in $ X_{ \pi \circ 0,uv} $ and $ X_{ \pi \circ 1,uv} $ will influence the
error in $ X_{ \pi,uv} $.

\item \emph{Side boundary edges with a standard fan:}
Let $ yz$ be a side boundary edge such that $y<z$ and $\pi(y-2)\neq\edg$ (the case $\pi(y-2)=\edg$, where the fan
captured by the corresponding side boundary edge has an exceptional structure, will be treated separately).

Recall that in this case there are
$\mathfrak{C}_\pi\in\{0,1\}$
copies of $\edg$ in $\pi$ and $|\pi|\leq 3M-1-(M-1)\mathfrak{C}_\pi$.
Moreover, $y>0$ and $z = y+M$ following Eq.~\eqref{eq-side-boundary-yz-relation},
and $z \leq 3M-1-(M-1)\mathfrak{C}_\pi$ as otherwise $yz$ would be classified as an outer boundary edge rather than a side boundary edge.
From these facts we infer that $\pi_{y-1}\circ \edg 1 \in \cB_M$ since this string contains
$\mathfrak{C}_\pi+1$ copies of $\edg$, its last symbol is 1 (as opposed to $\edg$) and it has length $y+1 = z-M +1 \leq 3M-1-(M-1)(\mathfrak{C}_\pi+1)$.

Similar to one of the arguments used for initial edges, we will let $y+1$ in the graph
$\sL_{\pi_{y-1}\circ \edg}$ play the role of $z$. To be precise, for each $ \varphi \in \Psi_{\pi_{y-1}\circ\edg,uv} $ let
 $ \varphi'$ be the composition of $\varphi$ and the map $z\mapsto y+1$.
 By the definition of the forward extension $F = F_{\pi, yz} $ as having $I_F=\{1,\ldots,y\}\cup\{z\}$, the variable $\Psi_{F,\varphi'}$ exactly captures the number of ways to extend $\varphi'$ to a full copy of $\sL_{\pi}$, thus
\[
\E[\Delta_{yz}\Psi_{\pi,uv}\mid\cF_i] = - \sum_{ \varphi \in \Psi_{\pi_{y-1}\circ\edg ,uv}} \frac{ Y_{ \varphi(y) \varphi(y+1)} \Psi_{ F, \varphi'}}{Q}\,.\]
Since $\sum_{\varphi} Y_{ \varphi(y) \varphi(y+1)}=\Psi_{\pi_{y-1}\circ\edg1}$ while $\Psi_{F,\varphi'}=(1+O(\zeta^{\delta}))\Sc_F$ for all $\varphi'$ by Corollary~\ref{cor:Psi-forward}, we can conclude that
\begin{equation}\label{eq-side-boundary}
\E[\Delta_{yz}\Psi_{\pi,uv}\mid\cF_i] = - \frac{ \Psi_{ \pi_{y-1} \circ \edg 1, uv} }{ \Psi_{\pi_{y-1}\circ \edg,uv} Q}\Psi_{\pi,uv}
+ \frac{O (\zeta^{1+\delta})}{ n^2p}\Psi_{\pi,uv}\,.
\end{equation}

We now turn to the effect this has on the one-step change in $\Delta X$. When $y < z \leq |\pi|$ we have
\begin{align}
\E\left[\Delta_{yz}\Psi_{\pi,uv}\mid\cF_i\right] - \frac{\Sc_{\pi}}{\Sc_{\pi^-}}\E\left[\Delta_{yz}\Psi_{\pi^-,uv}\mid\cF_i\right]
&= - \frac{ \Psi_{ \pi_{y-1} \circ \edg 1, uv} }{ \Psi_{\pi_{y-1}\circ \edg,uv} Q}X_{\pi,uv}
+ \frac{O (\zeta^{1+\delta})}{ n^2p}\Psi_{\pi,uv} \nonumber\\
&= -\frac{ 6-o(1)}{ n^2 p} X_{\pi,uv} + \frac{O \left( \zeta^{1+\delta}\right)}{ n^2p} \Psi_{\pi,uv} \,.\label{eq-side-1}
\end{align}
Alternatively, if $y < z = |\pi|+1$, 
since the main term of the right-hand-side of~\eqref{eq-side-boundary} equals
\[ 
\frac{X_{\pi_{y-1}\circ\edg1,uv} + np^2\Psi_{\pi_{y-1}\circ\edg,uv}}{\Psi_{ \pi_{y-1} \circ \edg, uv}Q}\Psi_{\pi,uv} =
\frac{6+o(1)}{n^3p^3}\frac{\Psi_{\pi,uv}}{\Psi_{ \pi_{y-1} \circ \edg, uv}} X_{\pi_{y-1}\circ\edg1,uv} + \frac{6+O(\zeta^2)}{n^2p}\Psi_{\pi,uv}
\]
(the error terms above are due solely to the approximation of $Q$) it follows that
\begin{align}
  \E\left[\Delta_{yz}\Psi_{\pi,uv}\mid\cF_i\right] + \frac{6}{n^2 p}\frac{\Sc_{\pi}}{\Sc_{\pi^-}}\Psi_{\pi^-,uv} = -\frac{6}{n^2 p}X_{\pi,uv} - \frac{6+o(1)}{n^3p^3}\frac{\Psi_{\pi,uv}}{\Psi_{ \pi_{y-1} \circ \edg, uv}} X_{\pi_{y-1}\circ\edg1,uv} +\frac{O \left( \zeta^{1+\delta}\right)}{ n^2p} \Psi_{\pi,uv}\,.\label{eq-side-2}
\end{align}

\item \emph{Side boundary edges with an exceptional fan:}
Let $ yz$ where $ y < z $ be a side boundary edge such that $\pi(y-2)=\edg$.
To tackle the expected change in this scenario we will modify the argument from the previous case to incorporate estimates on
homomorphisms from $\sL_{\pi_{y-2}\circ \edg}$ and $\sL_{\pi_{y-2}\circ \edg1}$.

Recall that now there is precisely 1 occurrence of $\edg$ in $\pi$ (at index $y-2$) and so $|\pi| \leq 2M$. We further have $y \geq 3$
and $z=y+M-1$ following~\eqref{eq-side-boundary-yz-relation}, and $z \leq 2M$ so as not to make $yz$ an outer boundary edge. Thus
$y= z-(M-1) \leq M+1$ and we see that $\pi_{y-2}\circ \edg1\in\cB_M$ as it has 2 copies of $\edg$ and ends with a 1 at index $y\leq M+1$
(of course $\pi_{y-2}\circ\edg\in\cB_M$ as well).

For each $\varphi\in \Psi_{\pi_{y-2}\circ\edg}$ let $\varphi'$ be the composition of $\varphi$ and the map the
sends $y\mapsto y-1$ and $z\mapsto y$. By definition, the forward extension
$F_{\pi,yz}$ has $I_F = \{0,1,\ldots,y-2\}\cup\{y,z\}$ and it is easy to verify that as in the previous case (where $\pi(y-2)\neq\edg$) we have
\[ \E[\Delta_{yz}\Psi_{\pi,uv}\mid\cF_i] =
- \sum_{ \varphi \in \Psi_{\pi_{y-2}\circ\edg ,uv}} \frac{ Y_{ \varphi(y-1) \varphi(y)} \Psi_{ F, \varphi'}}{Q}\,.\]
The rest of the argument proceeds exactly as in the case $\pi(y-2)\neq\edg$ with the single exception that $\pi_{y-2}$
replaces all occurrences of $\pi_{y-1}$. Overall we obtain that~\eqref{eq-side-1} is valid whenever $y<z\leq |\pi|$,
whereas in the situation of $y < z = |\pi|+1$ we have
\begin{align}
  \E\left[\Delta_{yz}\Psi_{\pi,uv}\mid\cF_i\right] + \frac{6}{n^2 p}\frac{\Sc_{\pi}}{\Sc_{\pi^-}}\Psi_{\pi^-,uv} = -\frac{6}{n^2 p}X_{\pi,uv} - \frac{6+o(1)}{n^3p^3}\frac{\Psi_{\pi,uv}}{\Psi_{ \pi_{y-2} \circ \edg, uv}} X_{\pi_{y-2}\circ\edg1,uv} +\frac{O \left( \zeta^{1+\delta}\right)}{ n^2p} \Psi_{\pi,uv}\,.\label{eq-side-2-excep}
\end{align}

\item \emph{Outer boundary edges:} Let $yz$ with $ z = |\pi|+1$ be an outer boundary edge.
By definition, extending $\sL_\pi$ by adding a new vertex incident to $y,z$ leads us out of $\cB_M$, foiling the argument we used for the other edge classes. In lieu of counting mappings of subgraphs that include $yz$ and then seeking estimates on their forward extensions to $\sL_\pi$, we will first directly identify the images of $y,z$ and then estimate the number of backward extensions to $\sL_\pi$. Formally, let $B=B_{\pi,yz}$ be the backward extension from $yz$ to $\sL_\pi$ and for a
given ordered pair $(a,b)$ with $ab \in E(G)$ and $\{a,b\}\cap\{u,v\}=\emptyset$ let $ \varphi^{ab}$ be the injection that maps
$0\mapsto u$, $1\mapsto v$, $y \mapsto a$ and $z\mapsto b$. Then
\[ \E[\Delta_{yz}\Psi_{\pi,uv}\mid\cF_i]=-\sum_{ \substack{(a,b) \,:\, \{a,b\}\cap\{u,v\}=\emptyset \\ ab \in E(G)}}
\frac{ Y_{ab} \Psi_{B, \varphi^{ab}}}{Q} \,. \]
Summing $Y_{ab}$ over the $2(|E(G)|-Y_u)$ ordered pairs $(a,b)$ as above gives $6Q-4\Psi_{\edg1,uv}$, as triangles containing $u$
are only featured in two possible orderings in this sum. The corresponding summation over $\Psi_{B,\varphi}$ as usual gives
$\Psi_{\pi,uv}$, and now thanks to Corollary~\ref{cor-backward-ext-pM} we infer that
\[
\E[\Delta_{yz}\Psi_{\pi,uv}\mid\cF_i]= - \frac{ (3Q-2\Psi_{\edg1,uv}) \Psi_{ \pi,uv}}{(|E(G)| - Y_u) Q} + \frac{O \left( \zeta^{1+\delta}\right)}{ n^2p}\Psi_{\pi,uv} = - \frac{ 6}{n^2 p}\Psi_{ \pi,uv} + \frac{O \left( \zeta^{1+\delta}\right)}{ n^2p}\Psi_{\pi,uv}\,,
\]
where the fact that $\Psi_{\edg1,uv}/Q$ and $Y_u / |E(G)|$ are both $O(1/n)$ allowed us to omit the terms $\Psi_{\edg1,uv}$ and $Y_u$ at an additive cost of $O\big(\frac1n\big)\frac{\Psi_{\pi,uv}}{n^2 p} = o(\zeta^2) \frac{\Psi_{\pi,uv}}{n^2 p}$, and similarly, we replaced the quantity $|E(G)| = \frac12(n^2 p -n)$ simply by $\frac12 n^2 p$ at an additive cost of $O\big(\frac{1}{np}\big)\frac{\Psi_{\pi,uv}}{n^2 p} = o(\zeta^2) \frac{\Psi_{\pi,uv}}{n^2 p}$.

It now follows that for any outer boundary edge $yz$,
\begin{align}
  \E\left[\Delta_{yz}\Psi_{\pi,uv}\mid\cF_i\right] + \frac{6}{n^2 p}\frac{\Sc_{\pi}}{\Sc_{\pi^-}}\Psi_{\pi^-,uv}
 &= -\frac{6}{n^2 p}X_{\pi,uv} + \frac{O \left( \zeta^{1+\delta}\right)}{ n^2p} \Psi_{\pi,uv}\,.\label{eq-outer-boundary}
\end{align}

\end{itemize}

To conclude the proof, observe that every edge type contributes the term $-(6+o(1))\frac1{n^2p} X_{\pi,uv}$ to the summation in~\eqref{eq-E-DX}, and correction terms accumulate  when $z=|\pi|+1$.
Outer boundary edges do not contribute any additional terms.
When $y=z-1$, Eqs.~\eqref{eq-initial-DX-2},\eqref{eq-interior-2} reveal that both initial edges and interior edges contribute an extra term of $-(6+o(1))\frac1{n^3p^3}X_{\pi\circ 1,uv}$, whereas side boundary edges cannot have $y=z-1$ for any $M\geq 3$.
When $y<z-1$, initial edges contribute the extra term $-(6+o(1))\frac1{np^2}X_{\pi,uv}$ as per~\eqref{eq-initial-DX-3}, interior edges contribute the extra term $-(6+o(1))\frac1{n^3p^3}X_{\pi\circ 1,uv}$ as per~\eqref{eq-interior-2} and side boundary edges contribute the term featured in~\eqref{eq-side-2}. These correction terms match the definition of $\Xi_0,\Xi_1$ in the statement of the theorem, thus completing the proof.
\end{proof}

\subsection{Concentration}\label{sec:ladders-concentrate}

In this subsection we prove the following theorem:
\begin{theorem}\label{thm:ladders-concentrate}
Let $M\geq 3$ and let $0<\delta<1$ be the constant from Corollaries~\ref{cor-backward-ext-pM} and~\ref{cor:Psi-forward}.
For any $\pi $
in $\cB_M$
let $\pi^-$ denote its $(|\pi|-1)$-prefix and define
\begin{align*}
X_{\pi,uv} &= \Psi_{ \pi,uv} - \frac{\Sc_{\pi}}{\Sc_{\pi^-}}\Psi_{ \pi^-,uv} \,,
\end{align*}
where $ \Psi_{\emptyset,uv} := 1 $.
Let $\zeta  = n^{-1/2} p^{-1} \log n$ and for some absolute constant $\kappa>0$ let
\[\tau_{\textsc{q}}^* = \tau_{\textsc{q}}^*(\kappa) = \min\left\{ t : \left| Q/(\tfrac16 n^3 p^3)-1\right| \geq \kappa\, \zeta^2 \right\}\,.\]
Then w.h.p., for every $\pi \in \cB_M$, every $u,v\in V_G$ and all $t$ such that $t \leq \tau_{\textsc{q}}^*$ and $p(t)\geq p_M$
\begin{align*}
|X_{\pi,uv}| &\leq \omega_\pi \zeta  \Sc_\pi\,,
\end{align*}
where $\omega_\pi = 3^{3M-|\pi| - (M-1)\mathfrak{C}_\pi}$ with $\mathfrak{C}_\pi$ denoting the number of copies of $ \emph{\edg} $ in $\pi$.
\end{theorem}
\begin{proof}
Let $\bar{\tau}$ be the stopping time where for some $\pi \in \cB_M$ and $u,v\in V_G$
we have $| X_{\pi,uv} | >  \omega_\pi \zeta \Sc_{\pi} $.
To show that w.h.p.\ we do not encounter $\bar{\tau}$ as long as $t\leq \tau_\star$ and $p(t) \geq p_M$, consider a pair $\pi,uv$ where at some step $i_0$
the variable $|X_{\pi,uv}|$ has just entered its critical
interval
\[ I_\pi = \left[\Big(1-\frac{1}{4e_\pi} \Big)\omega_\pi \zeta  \Sc_\pi~,~\omega_\pi
\zeta \Sc_\pi\right]\,.
\]
We define
\[ Z_{\pi,uv} = |X_{\pi,uv}| - \omega_\pi \zeta  \Sc_\pi\,, \]
and will now argue that $Z_{\pi,uv}$ is a supermartingale as long as $|X_{\pi,uv}| \in I_\pi$ and in addition we have not yet encountered $\bar{\tau} \wedge \tau_\star$ and are above the final probability threshold $p_M$. Indeed, since $\Delta(\zeta \Sc_\pi) = -\frac{6+o(1)}{n^2 p} (e_\pi -1) \zeta \Sc_\pi$, in this setting Theorem~\ref{thm:ladders-expected} implies that
\begin{align}
\E[ \Delta Z_{\pi,uv}\mid\cF_i] &\leq \frac{6+o(1)}{n^2 p}\Big(\big( e_\pi -1 \big)\omega_\pi
\zeta\Sc_\pi - e_\pi|X_{\pi,uv}|\Big) + |\Xi_0| + |\Xi_1| + \frac{ O \left( \zeta^{1+\delta}\right)}{ n^2p } \Sc_\pi \nonumber\\
&\leq -\frac{6+o(1)}{n^2 p}  \tfrac{3}4 \omega_\pi  \zeta \Sc_\pi + |\Xi_0| + |\Xi_1| + \frac{ O \left( \zeta^{1+\delta}\right)}{ n^2p } \Sc_\pi \label{eq-Z-pi-diff}\,.
\end{align}
Examining the definitions of $\Xi_0,\Xi_1$ from that theorem, we see that the contribution of $\Xi_0$ when the mentioned edge $(y,|\pi|+1)$ is an initial edge translates in our setting to $-\frac{6+o(1)}{n^2 p} |X_{\pi,uv}|$, strictly negative for large enough $n$. As we look for an upper bound on $\Xi_0+\Xi_1$, it is thus legitimate to ignore this case.
However, in all other cases with these correction terms, we have no guarantee that $X_{\sigma,uv}$ for $\sigma\neq \pi$ shares the same sign as $X_{\pi,uv}$. Worst case estimates for these variables (recalling the bounds on their absolute values, available throughout $t<\bar{\tau}$) give
 \[ |\Xi_1| \leq \frac{6+o(1)}{n^3p^3} \omega_{\pi\circ1} \zeta \Sc_{\pi\circ 1} =
 \frac{6+o(1)}{n^2p} \omega_{\pi\circ1} \zeta \Sc_{\pi}\,,\]
and similarly,
 \[ \frac{|\Xi_0|}{\frac{6}{n^2 p}\zeta \Sc_\pi} \leq \left\{\begin{array}
   {ll}
(1+o(1)) \omega_{\pi_{y-1}\circ\edg 1} &\mbox{if $\sL_\pi$ has a side boundary edge $yz$ with $z=|\pi|+1$, $\pi(y-2)\neq\edg$}\\
\noalign{\medskip}
(1+o(1)) \omega_{\pi_{y-2}\circ\edg 1} &\mbox{if $\sL_\pi$ has a side boundary edge $yz$ with $z=|\pi|+1$, $\pi(y-2)=\edg$}\\
\noalign{\medskip}
(1+o(1)) \omega_{\pi\circ0} &\mbox{if $\sL_\pi$ has an interior edge $yz$ with $ y < z-1 $ and $ z = |\pi|+1$}
 \end{array}\right.
\]
Returning to~\eqref{eq-Z-pi-diff}, when $\sL_\pi$ does not contain a side boundary edge $yz$ with $z=|\pi|+1$,
\begin{align*}
\E[ \Delta Z_{\pi,uv}\mid\cF_i] &\leq -\frac{6+o(1)}{n^2 p} \left( \tfrac{3}4 \omega_\pi - \omega_{\pi\circ 0} - \omega_{\pi\circ1}\right)  \zeta \Sc_\pi + \frac{ O \left( \zeta^{1+\delta}\right)}{ n^2p } \Sc_\pi \\
&= -\frac{1/2+o(1)}{n^2 p} \omega_\pi \zeta \Sc_\pi + \frac{ O \left( \zeta^{1+\delta}\right)}{ n^2p } \Sc_\pi
\,,
\end{align*}
where we set $ \omega_\sigma = 0 $ if $ \sigma \not\in \cB_M $ and the equality stemmed from the fact that for any $b\in\{0,1\}$, since $\pi$ and $\pi\circ b$ share the same number of $\edg$ symbols, our definition of the constants $(\omega_\sigma)_{\sigma\in\cB_M}$ implies that $\omega_{\pi\circ b} = \frac13 \omega_{\pi}$.

Similarly, when $yz$ with $z=|\pi|+1$ is a side boundary edge of $\sL_\pi$ we can let $r=y-1$ if $\pi(y-2)\neq\edg$ and $r=y-2$ otherwise,
and obtain that
\[
\E[ \Delta Z_{\pi,uv}\mid\cF_i] \leq -\frac{6+o(1)}{n^2 p} \left( \tfrac34 \omega_\pi - \omega_{\pi\circ 1} - \omega_{\pi_{r}\circ\edg 1}\right)  \zeta \Sc_\pi + \frac{ O \left( \zeta^{1+\delta}\right)}{ n^2p } \Sc_\pi \,.
\]
 By definition of side boundary edges, $\pi$ has at most a single copy of $\edg$, thus $\omega_{\pi} = 3^{3M-|\pi| - (M-1) \mathfrak{C}_\pi}$
  and $\omega_{\pi_{r\circ\edg 1}}=3^{3M-(r+2)- (M-1)
 (\mathfrak{C}_\pi+1)}=3^{2M-1-r- (M-1) \mathfrak{C}_\pi}$.
Furthermore, revisiting~\eqref{eq-side-boundary-yz-relation} we have $y=z-M=|\pi|+1-M$ if $\pi(y-2)\neq \edg$ (and then $r=|\pi|-M$)
while $y=z-(M-1)=\|pi|-M+2$ if $\pi(y-2)= \edg$ (and then again $r=|\pi|-M$). In both cases we therefore get
that $\omega_{\pi_{r\circ\edg 1}} = \frac13 \omega_\pi = \omega_{\pi\circ1}$ and so again we obtain that
\[ \E[ \Delta Z_{\pi,uv}\mid\cF_i] \leq -\frac{1/2+o(1)}{n^2 p} \omega_\pi \zeta \Sc_\pi + \frac{ O \left( \zeta^{1+\delta}\right)}{ n^2p } \Sc_\pi \,.\]
In particular, $\E[\Delta Z_{\pi,uv}\mid\cF_i] < 0$ for large enough $n$, as claimed.

We are left with the task of  bounding the one-step changes $\Delta Z$ given $\cF_i$ in $L^\infty$ and $L^2$ norm.
Revisiting the events $\cE_{\psi}$ and $\cE_{\psi}^{yz}$ defined in~\eqref{eq-cE-psi-def}, observe that when aiming for an upper bound on $|\Delta\Psi_{\pi,uv}|$ (as opposed to tight asymptotics, where~\eqref{eq-cE-approx} incurred an approximation error) one always has
\[ \left|\Delta \Psi_{\pi,uv}\right| = \sum_{\psi\in\Psi_{\pi,uv}} \one_{\cE_{\psi}} \leq
\sum_{\psi\in\Psi_{\pi,uv}} \sum_{yz\in E(\sL_\pi)} \one_{\cE_{\psi}^{yz}} = \sum_{yz\in E(\sL_\pi)} \left|\Delta_{yz} \Psi_{\pi,uv}\right|\,.\]
In particular, for the sake of an upper bound on $|\Delta Z|$ and $\E[(\Delta Z)^2 \mid \cF_i]$ it will suffice to analyze the effect of each edge $yz\in E(\sL_\pi)$ separately.

Let $yz\in E(\sL_\pi)$ with $y<z$ and let $H=(V(\sL_\pi),E(\sL_\pi))$ be the extension graph featuring $I_H=\{0,1\}$. As argued in the proof of Theorem~\ref{thm-Psi-H-estimate}, let $T=T(yz)\subset H$ be the subextension of $H$ with minimal scaling out of all subextensions that contain the edge $yz$ (and, if multiple such subextensions exist, with minimal cardinality out of those).

Consider some subextension of the quotient $H/T$. This graph can be written as $K/T$ for some subextension $K\subset H$ that contains all edges of $T$, thus $\Sc_K \geq \Sc_T$ by the minimality of $T$, and equivalently $\Sc_{K/T} \geq 1$. Part~\eqref{it-scaling-geq-1} of Corollary~\ref{cor:general-ext} now implies that for some $c_1>0$ and any injection $\varphi : T \to V_G$ we have
\begin{equation}
  \label{eq-Psi-H/T-upper-bound}
  \Psi_{H/T,\varphi} \leq \Sc_{H/T} (\log n)^{c_1}\,.
\end{equation}
Notice that at the same time any subextension $K\subset H$ has $\Sc_K \geq 1$. This follows from the fact that, by Definition~\ref{def:ladders}, every vertex $1<i\leq |\pi|+1$ shares at most $2$ edges with the vertices $\{ j : j < i\}$, thus $e_K \leq 2(v_K-\iota_K)$ and $\Sc_{K} \geq n^{(v_K-\iota_K)/M} \geq 1$ due to the hypothesis $p\geq p_M$. Since every $K\subset T$ has $\Sc_K \geq 1$ (being also a subextension of $H$) we may again appeal to Part~\eqref{it-scaling-geq-1} of Corollary~\ref{cor:general-ext} to obtain that for some $c_2>0$ and any $u,v\in V_G$ we have
\begin{equation}
  \label{eq-Psi-T-upper-bound}
  \Psi_{T,uv} \leq \Sc_{T} (\log n)^{c_2}\,.
\end{equation}
A final ingredient we need is to control the number of extensions to a copy of $T$ given the image of the edge $yz$.
Formally, let $T_0$ be the edgeless graph on the vertex set $\{0,1,y,z\}$ and consider the quotient $T^* = T/T_0$ (that is, the extension graph obtained from $T$ by increasing the distinguished vertex set to $\{0,1,y,z\}$ and removing any edges among these vertices).
The choice of $T$ implies that every $K^* \subsetneq T^*$ has $\Sc_{K^*} > \Sc_{T*}$ (as otherwise the result of moving the vertices $y,z$ from $I_K$ to $V_K\setminus I_K$ and adding all edges of $T$ among $\{0,1,y,z\}$ would yield a graph $K$ with $\Sc_{K}\leq \Sc_{T}$ and a strictly smaller vertex set). Equivalently,
$\Sc_{T^*/K^*} \leq 1$ for all $K^*\subset T^*$ and so, by Part~\eqref{it-scaling-leq-1} of Corollary~\ref{cor:general-ext}, we deduce that for some $c_3>0$ and any injection $\eta : \{0,1,y,z\} \to V_G$,
\begin{equation}
  \label{eq-Psi-T*-upper-bound}
  \Psi_{T/T_0,\eta} \leq (\log n)^{c_3}\,.
\end{equation}
We now translate these facts to $L^\infty$ and $L^2$ bounds on $\Delta_{yz} X_{\pi,uv}$ given $\cF_i$.
The analysis will differ depending on whether or not $ y \in \{0,1\}$.
\begin{enumerate}[\!\!(i)]
\item \label{it-y>1} In case $y>1$:  Since $T$ contains the edge $yz$ we have $\Sc_T \geq n^2 p^4$ already due to the vertices $y,z$ (as argued above, each vertex $i>1$ that belongs to $V_T$ sends at most two edges backwards).

    For the $L^\infty$ bound we look at the maximum number of extensions $\psi\in\Psi_{\pi,uv}$ all of which map $yz$ to a particular edge in $E(G)$. Equivalently, we are examining the maximum possible value of $\Psi_{H/T_0,\eta}$ over all injections $\eta:V(T_0)\to V_G$ that map $0\to u$ and $1 \to v$ with the
    additional property of preserving all edges of the induced graph of $H$ on $\{0,1,y,z\}$.  Each such extension $\psi$ factors into extensions $T_0 \hookrightarrow T$ and $T \hookrightarrow H$, thus an application of~\eqref{eq-Psi-T*-upper-bound} and~\eqref{eq-Psi-H/T-upper-bound} gives
 \begin{equation}\label{eq-y>1-Linf-bound}
|\Delta_{yz} \Psi_{\pi,uv} | \leq \max_{\eta} \Psi_{H/T_0,\eta} \leq \max_{\eta} \sum_{\varphi \in \Psi_{T/T_0,\eta}} \Psi_{H/T,\varphi} \leq \Sc_{H/T}(\log n)^{c_1+c_3}\,.
 \end{equation}
 Since $\Sc_{H/T}=\Sc_H/\Sc_T$ and $\Sc_T \geq n^2 p^4$ we further have
 \begin{equation}\label{eq-y>1-Linf-bound-2}
|\Delta_{yz} \Psi_{\pi,uv} | \leq \max_{\eta} \Psi_{H/T_0,\eta} \leq \frac{(\log n)^{c_1+c_3}}{n^2 p^4} \Sc_{H}\,.
 \end{equation}

    For the $L^2$ bound, observe that the set of edges in $E(G)$ that can play the role of $yz$ for some copy of $H$ rooted at $uv$ is clearly contained in $\{ \psi(y)\psi(z) : \psi \in \Psi_{T,uv}\}$. Using~\eqref{eq-Psi-T-upper-bound}, the probability that such an edge is chosen in the next round is therefore at most
    \[ \sum_{\psi\in \Psi_{T,uv}}\frac{Y_{\psi(y)\psi(z)}}Q \leq (6+o(1))\frac{\Sc_T (\log n)^{c_2}}{n^2 p}\,.\]
    By~\eqref{eq-y>1-Linf-bound}, combined with the usual fact that $\Sc_{H/T} = \Sc_{H}/\Sc_T$, it follows that
    \[ \E [(\Delta_{yz} \Psi_{\pi,uv})^2 \mid \cF_i] \leq (6+o(1))\frac{(\Sc_H)^2 (\log n)^{c_2+2(c_1+c_3)}}{\Sc_T n^2 p} \leq
    \frac{(\Sc_H)^2 (\log n)^{C}}{n^4 p^5}
    \,,\]
    where the last inequality featured some suitably large constant $C>0$ and relied on the fact that $\Sc_T \geq n^2 p^4$.
    Integrating this over a maximum of $n^2 p$ remaining steps gives
    \begin{align}
    \label{eq-y>1-L2-bound}
    \sum_{i>i_0} \E [(\Delta_{yz} \Psi_{\pi,uv})^2 \mid \cF_i] \leq \sum_{i>i_0} \frac{(\Sc_H)^2 (\log n)^{C}}{n^4 p(i)^5} \leq \frac{(\Sc_H)^2}{n^2 p(i_0)^4} (\log n)^{C}\,.
    \end{align}

\item \label{it-pi(1)=1} In case $ y \in \{0,1\} $ and  $ \pi(1) =1 $:

\begin{compactitem}[\noindent$\bullet$]
\item Suppose $z=2$. Here $V(T)=\{0,1,2\}$ and $E(T)=\{02,12\}$ (any additional vertex would increase $\Sc_T$).
The number of the number of extensions $\psi\in\Psi_{\pi,uv}$ which map $yz$ to a prescribed edge in $E(G)$ is
exactly the number of forward extensions from $yz$ to $\pi$. Since $F=F_{\pi,yz}$ has $V(T)$ as its distinguished vertices,
$\Sc_F = \Sc_H / (np^2)$ and Corollary~\ref{cor:Psi-forward} gives
\begin{equation}
\label{eq-y=1,z=2-Linf-bound}
   \Delta_{yz}\Psi_{\pi,uv}  = O( \Sc_H/ (n p^2) )\,.
\end{equation}
We claim that the probability of choosing a triangle incident to $yz$ is $O(p/n)$. Indeed, if $y=0$ then
we are looking at triangles incident to $u$ and to some neighbor of $v$ (to assume the role of the image of $z$), and
there are $N_{u,v}=O(np^2)$ candidates for the image of $z$ each sharing $O(np^2)$ triangles with $u$. The same holds symmetrically
if $y=1$ and we deduce that
\begin{equation}
\label{eq-y=1,z=2-L2-bound}
\sum_{i>i_0} \E [( \Delta_{yz}\Psi_{\pi,uv} )^2 \mid \cF_i ] \leq \sum_{i>i_0} O\bigg( \frac{p}n\, \frac{\Sc_H^2}{n^2 p^4}\bigg)
= O\bigg( \frac{\Sc_H^2} {np(i_0)^2}\bigg)\,.
\end{equation}

\item Suppose $z\geq 3$. Here necessarily $y=1$ and either $\pi_z=110\ldots 0$ or $\pi_z = 11$ (since $\pi(2)\neq 0$).

Note that if $ V(T) = \{ 0,1,2, \ldots, z \} $ then $\Sc_T \geq n^2 p^4$
already due to two backward edges sent from $2$ and two such edges from $z$. Thus, the argument
for $y>1$ (Case~\eqref{it-y>1}) applies in this case as well and implies the bounds in~\eqref{eq-y>1-Linf-bound-2} and~\eqref{eq-y>1-L2-bound} for
the $L^\infty$ and $L^2$ norms, resp.

Otherwise, we must have
$ V(T) = \{0,1,z\} $ and $E(T)=\{ 1z\}$. (Indeed, including any $w>z$ in $V(T)$ would clearly increase $\Sc_T$.
If $V(T)\cap A \subsetneq A$ for $A= \{2,\ldots,z-1\}$ then of course $z>3$ and so $\pi_z=110\ldots 0$. Thus,
the vertex with minimal index in $V(T) \cap A$ sends 2 edges backwards and its removal from $T$ would decrease $\Sc_T$,
contradicting the minimality of $T$.)

\medskip
For the $L^\infty$ bound, we can control the number of forward extensions $F_{\pi,yz}$ from $yz$ to $\pi$ via Corollary~\ref{cor:Psi-forward}.
If $yz$ happens to be a side boundary edge (in our setting this will happen iff $z$ concludes an $(M-1)$-fan at vertex 1) then
$I_F=V(T)=\{0,1,z\}$ thus $\Sc_F = \Sc_H / (np)$ and
\[
   \Delta_{yz}\Psi_{\pi,uv}  = O( \Sc_H/ (n p) )\,.
\]
For all other edge-types $I_F=\{0,1,\ldots,z\}$ and we must examine the number of extensions from $V(T)=\{0,1,z\}$ to the induced graph
of $H$ on $I_F$. Specifically, Let $R$ be the extension graph of $\sL_{\pi_z}$ with the distinguished vertex set $\{0,1,z\}$.
Recalling that $\pi_z = 11\ldots0$, there are $z-2$ vertices and $2(z-1)-1=2z-3$ edges in $R$.
Observe that $R$ is strictly balanced, since we can view any vertex $1<j<z$ is accountable for two exclusive
directed edges going from it to $1$ and $j+1$ (the maximal density is attained when all such $j$'s are present). Finally, $\Sc_R \geq 1$ as otherwise adding the vertices to $\{2,\ldots,z-1\}$ to $V(T)$ would give an extension graph
containing $yz$ with a scaling of $\Sc_T \Sc_R < \Sc_T$, violating the definition of $T$. That is,
\[ n^{z-2} p^{2z-3} \geq 1\,.\]
As before, we count extensions $\psi\in\Psi_{\pi,uv}$ which map $yz$ to a prescribed edge $e = vw$ via the product
of $ \Psi_{R, uvw}$ and $\max_{\eta} \Psi_{F_{\pi, yz}, \eta} $.
The latter is asymptotically $\Sc_F = \Sc_H / (np \Sc_R)$ by Corollary~\ref{cor:Psi-forward}.
For the former we apply Theorem~\ref{thm-Psi-H-estimate}, yielding an upper bound of $(1+o(1))\Sc_R$
as long as $t \leq t_R^-(\sqrt{\delta})$ and the upper bound $\Sc_R + \Sc_R^{1-\sqrt{\delta}}(\log n)^\rho $
for some fixed $\rho=\rho_R>0$ for $t_R^-(\sqrt{\delta}) t \leq t_R$.
We have $\Sc_R \geq 1$, hence indeed $t \leq t_R$ and the combined upper bound of
$\Sc_R + \max\{ \Sc_R, \Sc_R^{1-\sqrt{\delta}}(\log n)^\rho\}$ is a valid one for $\Psi_{R,uvw}$.
Moreover, since $\Sc_R$ is decreasing we see that this bound is $2\Sc_R$ up to the first point where
$\Sc_R \leq (\log n)^{\rho/\sqrt{\delta}}$ (namely, $p=n^{-\frac{z-2}{2z-3}+o(1)} \leq n^{-2/3+o(1)}$ for any $z\geq 3$).
Beyond that point we can simply use the bound
$2 (\log n)^{c}$ for $c = ( 1 + \frac{1-\sqrt{\delta}}{\delta})\rho$. (
Altogether,
\begin{equation}
   \label{eq-y=1-Linf-bound}
   \Delta_{yz}\Psi_{\pi,uv}  \le (2+o(1)) \max\{  \Sc_R\;,\; \log^c n \} \frac{\Sc_H}{np\Sc_R}
   =  O\bigg( \frac{\Sc_H}{np^{5/4}} \bigg) \,,
\end{equation}
where the last inequality used the aforementioned fact that the term $\log^c n$ can dominate $\Sc_R$ only starting at
$p \leq n^{-2/3+o(1)}$, whilst it is then easily absorbed by the extra factor of $p^{-1/4}$.
The $L^\infty$ bound in~\eqref{eq-y=1-Linf-bound} is of larger order then the $O(\Sc_H/(np))$ bound given for the case where
$yz$ is a side boundary edge, thus we can safely use it for any edge-type.

\medskip
For the $L^2$ bound, note that the probability of selecting a triangle incident to $v$ (supporting the removal of $yz$) is
$\sum_{w\in N_v} N_{vw} / Q =  O(1/n) $. Summing the one-step expected variances in the usual way yields
\begin{align}\label{eq-y=1-L2-bound}
    \sum_{i>i_0} \E [( \Delta_{yz}\Psi_{\pi,uv} )^2 \mid \cF_i] \leq
    \sum_{i>i_0} O\bigg( \frac{\Sc_H^2}{n^3 p(i)^{5/2}}\bigg)
= O\bigg( \frac{\Sc_H^2} {n p(i_0)^{3/2}}\bigg)\,.
\end{align}
\end{compactitem}
\item In case $ \pi(1)= \edg$ and $ y =1 $:
By definition $yz$ is an initial edge, and since the corresponding forward extension $F_{\pi,yz}$ has $I_F=\{0,1,z\}$ we
can invoke Corollary~\ref{cor:Psi-forward} and infer that
\begin{equation}
\label{eq-y=1,edg-Linf-bound}
   \Delta_{yz}\Psi_{\pi,uv}  \leq  \max_\eta \Psi_{F_{\pi,yz},\eta} = O(\Sc_F) = O( \Sc_H/ (n p) )
\end{equation}
and (again by the fact that the probability of choosing a triangle incident to $v$ is $O(1/n)$)
\begin{equation}
\label{eq-y=1,edg-L2-bound}
\sum_{i>i_0} \E [( \Delta_{yz}\Psi_{\pi,uv} )^2 \mid \cF_i] \leq
    \sum_{i>i_0} O\bigg( \frac{\Sc_H^2}{n^3 p(i)^{2}}\bigg) \leq
O\bigg( \frac{\Sc_H^2}{np(i_0)}\bigg)\,.
\end{equation}
\end{enumerate}

We now combine the bounds above for analogous estimates on $\Delta Z$.
Comparing~\eqref{eq-y>1-Linf-bound-2}, \eqref{eq-y=1,z=2-Linf-bound}, \eqref{eq-y=1-Linf-bound} and~\eqref{eq-y=1,edg-Linf-bound}
we see that for all $\pi$ and all edges $yz$ in $ \sL_\pi$ we have
\begin{equation}
\label{eq-Linf-maxbound}
   \Delta_{yz}\Psi_{\pi,uv}  = O\bigg( \frac{\Sc_H}{n p^2} \bigg)
\end{equation}
whereas a comparison of~\eqref{eq-y>1-L2-bound}, \eqref{eq-y=1,z=2-L2-bound}, \eqref{eq-y=1-L2-bound} and~\eqref{eq-y=1,edg-L2-bound}
shows that similarly we always have
\begin{equation}
\label{eq-L2-maxbound}
\sum_{i>i_0} \E [( \Delta_{yz}\Psi_{\pi,uv} )^2 \mid \cF_i] \leq O\bigg( \frac{\Sc_H^2}{np^2}\bigg)\,.
\end{equation}
By summing these expressions over all $yz\in E(H)$ we arrive at the corresponding upper bounds on $\Delta \Psi_{\pi,uv}$ (of the same order). Furthermore, if $z\leq |\pi|$ we may repeat this analysis for $\sL_{\pi^-}$, and since the obtained bounds will be in terms of $\Sc_{\pi^-}$ while the variable $X_{\pi,uv}$ features $(\Sc_{\pi}/\Sc_{\pi^-}) \Psi_{\pi^-,uv}$ we deduce that these bounds extend  the aforementioned bounds to $X_{\pi,uv}$ itself. It is furthermore easy to verify that the change in $Z$ due to the term $\zeta \Sc_\pi$, changing deterministic with $p$ by an additive term of relative order $O(1/(n^2 p))$, is negligible in comparison with these bounds.
(Indeed, it is a factor of $n /(p\zeta)$ smaller than our Lipschitz bound.) Altogether, the $L^\infty$ and cumulative $L^2$ bounds hold for the supermartingale $Z$ and we can now exploit these to obtain a large deviation estimate.

Observe that the $L^\infty$ bound in~\eqref{eq-Linf-maxbound} is $o(\zeta \Sc_H)$
since $n^{1/2} p$ tends to infinity (it is at least as large as $n^{1/M}$). Thus, the first step of $Z$ upon which $|X_{\pi,uv}|$ enters the critical interval $I_\pi$ is negligible in terms of the order of this window and we have $Z(i_0) = (-1/(4e_\pi)+o(1))\omega_\pi \zeta \Sc_\pi$.

Set $s = (5e_\pi)^{-1} \omega_\pi \zeta \Sc_\pi$ as the target deviation for our supermartingale $Z$.
The product of the uniform $L^\infty$ bound and $s$ has order $\zeta (\Sc_\pi)^2 / (np^2)$.
This is of lower order compared to the cumulative $L^2$-bound in~\eqref{eq-L2-maxbound}  since $\zeta \to 0$ given the fact that $n p^2 \geq n^{1/M}$ throughout the range $p\geq p_M$.

We may now invoke Freedman's inequality and get that for some fixed $c>0$,
\begin{align*} \P\Big(\exists t>t_0: Z(t\wedge \tau_\star \wedge\bar{\tau}) \geq Z(t_0) + s\Big) &\leq
\exp\left(- \frac{ c (\zeta \Sc_\pi)^2}{ \Sc_\pi^2 / (np^2)} \right) =
e^{-c np^2 \zeta^2 } = e^{- c (\log n)^2} \,.
\end{align*}
Having arrived at an error probability that decays to $0$ at a super-polynomially rate, a union bound over the ensemble of $(\Psi_{\pi,uv})$ variables establishes the required estimates for all $\pi \in \cB_M$.
\end{proof}


\section{Lower bound on the final number of edges}\label{sec:lower}
Our lower bound of $n^{3/2-o(1)}$ on the final number of edges will rely solely on the fact that for any fixed $\epsilon>0$ we have asymptotically tight estimates for all the co-degrees in the process at time $t$ such that $p(t)=n^{-1/2+\epsilon}$.
\begin{theorem}\label{thm-lower-bound}
Suppose that for some fixed $0<\epsilon<\frac16$, all co-degrees satisfy $Y_{u,v} = (1+o(1)) np^2$ throughout $p\geq p_0=n^{-1/2+\epsilon}$.
Then w.h.p.\ the final number of edges is at least $n^{3/2-6\epsilon-o(1)}$.
\end{theorem}
\begin{remark}
  \label{rem-lower-bound}
The above result was formulated in accordance with our upper bound analysis, which achieves the required estimate over the co-degrees as long
as the number of edges is of order $n^{3/2+\epsilon}$ for an arbitrarily small fixed $\epsilon>0$.
The proof argument of Theorem~\ref{thm-lower-bound} in fact yields a stronger result, namely that an asymptotic upper bound of $np^2$ over all co-degrees
throughout the point of $n^{3/2} h$ edges, for some $h(n) < n^{1/6}$ that grows to $\infty$ with $n$, implies that the final number of edges is w.h.p.\ at least
$c n^{3/2} h^{-6}$ for some absolute $c>0$.
\end{remark}

An important ingredient in the proof is the following straightforward lemma which provides an upper bound on $Q$ valid well beyond the point where our co-degree estimates break down. An estimate of this nature, albeit weaker, was provided in~\cite{BFL}. Towards the end of the process, what begins as a second order term in this bound will become the main term, hence the importance of estimating it sharply.
\begin{lemma}\label{lem-Q-upper-bound}
Let $\tau_{c}$ denote the minimal time where $Y_{u,v} > 2(np^2 + n^{1/3})$ for some
$u,v\in V_G$.
 With high probability, for all $t\leq \tau_c$ such that $p(t)\geq n^{-3/5}$, the number of triangles satisfies
\[    Q - \tfrac16\left(n^3 p^3 + n^2 p\right)\leq n^{7/3} p^{2} \,.
\]
\end{lemma}
\begin{proof}
By a slight abuse of notation we let $Q$ denote both the collection of all triangles in $G$ and its cardinality. By definition,
\[ \E[\Delta Q \mid \cF_i] = -\sum_{uvw \in Q} \frac{Y_{u,v}+Y_{u,w}+Y_{v,w}-2}{Q} = 2 - \frac1Q \sum_{uv\in E} Y_{u,v}^2\,.\]
Since $\sum_{uv\in E} Y_{u,v} = 3Q$, applying Cauchy-Schwarz shows that $ \sum_{uv\in E} Y_{u,v}^2 \geq 9Q^2 /|E|$ and thus
\[ \E[\Delta Q \mid \cF_i] \leq 2 - \frac{9Q}{|E|} = 2 - \frac{9Q}{\frac12 (n^2 p - n)} \leq 2 - \frac{18}{n^2p}Q\,.\]
Define
\[ \Upsilon = \frac16\left(n^3 p^3 + n^2 p\right) + n^{7/3} p^2\]
and observe that the one-step change in this variable is deterministically given by
\[ \Delta(\Upsilon) = -\bigg(1+O\Big(\frac1{n^2p}\Big)\bigg)\frac6{n^2}\left(\frac12 n^3 p^2 + \frac16 n^2 + 2 n^{7/3} p \right)
=-\left(3 np^2 + 1 + 12n^{1/3} p\right) + O(p/n)
\,.\]
Therefore, if we set
\begin{align*}
Z = Q - \Upsilon &\quad,\quad I_Q =\big[ \Upsilon - \tfrac16 n^{7/3} p^{2} ~,~ \Upsilon \big]
\end{align*}
it follows that whenever $Q \in I_Q$ we have
\begin{align*}
\E[ \Delta Z \mid \cF_i] &\leq 2 - \left(3np^2 + 3 + 15 n^{1/3} p\right) + \left(3 np^2 + 1 + 12n^{1/3} p\right) + O(p/n)= -3 n^{1/3} p + O(p/n) < 0\,,
 \end{align*}
 where the last inequality holds for any sufficiently large $n$.

Let $i_0$ be the first time where $Q\in I_Q$ and write $t_0=t(i_0)$ and $p_0=p(i_0)$. For any $t_0 < \tau_c$
\[ |\Delta Q| \leq 3 \max_{u,v}Y_{u,v} = O(n p_0^2 + n^{1/3})\,,\]
and observing that this bound is of order strictly lower than $n^{7/3} p_0^{2}$ for all $p_0\geq 1/n$ we obtain that
\[ Z(i_0) = -(\tfrac16-o(1))n^{7/3} p_0^{2}\,.\]
From this point the process has at most $n^2 p_0/2$ steps until its conclusion, therefore we may substitute $s = \frac18 n^{7/3} p_0^2 $ and deduce from Hoeffding's inequality that for some fixed $c>0$
\[ \P( \exists i>i_0:\, Z(i) \geq Z(i_0) + s) \leq \exp\left(-c\frac{s^2}{n^2 p_0 (n p_0^2 + n^{1/3})^2}\right)
= \exp\left(-c\frac{n^{8/3} p_0^3}{( n p_0^2 + n^{1/3})^2}\right)\,.\]
If $p_0 \geq n^{-1/3}$ then the last exponent has order $n^{2/3} / p_0 \geq n^{2/3}$, and otherwise it has order $n^{2} p_0^3 \geq n^{1/5}$ for any $p_0 \geq n^{-3/5}$.
A union bound over $i_0$ now completes the proof.
\end{proof}

\begin{proof}
  [\emph{\textbf{Proof of Theorem~\ref{thm-lower-bound}}}]
 Fix $0<\epsilon<\frac16$ and let $i_0$ be such that \[ p_0=p(i_0) = n^{-1/2+\epsilon}\,.\]
 Condition on $\cF_{i_0}$ and assume that up to this point we had $Y_{u,v}=(1+o(1))n^2 p$ for all $u,v\in V_G$.

For two distinct triangles $uvw,xyz\in Q$ we write $uvw \sim xyz$ to denote that they share an edge.
Further denote the number of distinct triangles incident to some $uvw\in Q$ by
 \[ B_{uvw} = \#\{xyz\in Q : xyz\sim uvw\} = Y_{u,v} + Y_{u,w} + Y_{v,w} - 3 \,.\]
 Observe that our hypothesis on the co-degrees implies that $B_{uvw} = n^{2\epsilon+o(1)}$ for all $Q$ triangles at time $i_0$.
 In addition, we claim that if $uvw \in Q(j-1)$ then at the end of round $j$ either $uvw \notin Q(j)$ or $B_{uvw}(j-1) - B_{uvw}(j) \leq 4$, and in the latter case the co-degree of at most 2 edges in $uvw$ has changed. Indeed, if the triangle selected in round $j$ is either $uvw$ itself or some triangle $xyz \sim uvw$ then at once $uvw\notin Q(j)$, whereas removing a vertex disjoint triangle $xyz$ does not affect $B_{uvw}$. Consider therefore the case of removing a triangle $xyz$ sharing exactly one vertex in common with $uvw$, that is, w.l.o.g.\ $x=u$ and $\{y,z\}\cap \{v,w\}=\emptyset$. In this case, each of the edges $uy,uz$ may participate in at most one triangle involving the edge $uv$ and at most another involving the edge $uw$, thus at most 4 triangles can be deducted from the count of $B_{uvw}$ in this case whereas $Y_{v,w}$ remains unchanged. 
 (This last scenario is tight whenever both $y$ and $z$ belong to $N_{u,v} \cap N_{u,w}$.)

 Next we will consider a random subset $\sX$ of all triangles in $Q(i_0)$. Initialize this set by $\sX(i_0)=\emptyset$ and let it evolve until the culmination of the triangle removal process, denoting its final value by $\sX_\star$. To formulate its evolution rule we further maintain an auxiliary set $\sY$, also initialized by $\sY(i_0)=\emptyset$.
At round $j> i_0$ do as follows:
 \begin{itemize}[\noindent$\bullet$]
   \item  Process the triangles $Q(j)$ in an arbitrary order.
   \item If the current triangle $uvw \in Q(j)$ has $uvw\notin \sY$ and in addition this is the first round $j$ after which
   some edge $uv$ in this triangle has no other triangles incident to it ($Y_{u,v}=1$) then:
       \begin{itemize}
         \item Add $uvw$ to $\sX$.
         \item Add all triangles in $\cN_2(uvw) $ to $\sY$, where $\cN_1(uvw) = \{xyz : xyz \sim uvw\}$ denotes all triangles incident to $uvw$ and $\cN_2(uvw) = \cN_1(uvw) \cup\{x'y'z' : x'y'z'\sim xyz \in \cN_1(uvw)\}$.
       \end{itemize}
 \end{itemize}
(Note that at no point do we ever delete triangles from $\sX$ or from $\sY$.)
As explained above, removing some triangle $xyz$ can affect triangles incident to at most two edges of $uvw$. Therefore, if $j$ is the first round at the end of which $uvw$ has some edge with co-degree $1$, and $uvw$ is intact at the end of this round, then at least one of its other edges has co-degree greater than $1$ and so $B_{uvw}>0$ upon the insertion of $uvw$ to $\sX$.

With the above explanation in mind, consider some time $i_1 > i_0$. At time $i_0$, every triangle $uvw\in Q(i_0)$ had all of its co-degrees equal $n^{2\epsilon+o(1)}>1$. If $uvw \in Q(i_1)$ has $B_{uvw}=0$ then either $uvw \in \sX$ or necessarily there exists some $xyz\in \sX$ such that at some point $i_0 < j \leq i_1$ we had $uvw\in \cN_2(xyz)$.
By the co-degrees hypothesis we know that already at time $i_0$ we had $|\cN_2(xyz)| \leq n^{4\epsilon+o(1)}$ for all $xyz\in Q(i_0)$, and since co-degrees only decrease throughout our process, we can conclude that
\begin{equation}
  \label{eq-X-lower-bound}
  |\sX(i_1)| \geq |\left\{ uvw \in Q(i_1) : B_{uvw} = 0\right\}| \, n^{-4\epsilon-o(1)}\,.
\end{equation}
Let $i_1$ be such that
$ p_1 = p(i_1) = \frac{1}{\sqrt{n} \log n}$.
At this point $|E| = \frac12(n^2p_1 -n) = (\frac12+o(1))n^{3/2}(\log n)^{-1}$ while
$n^3 p_1^3 = O(n^{3/2} (\log n)^{-3})$. Obviously, if the process concludes before time $i_1$ we have at least $|E(i_1)| = n^{3/2-o(1)}$ edges
in the final graph and there is nothing left to prove. Assume therefore that the process is still active at time $i_1$.
By Lemma~\ref{lem-Q-upper-bound} (noting that our assumption on the co-degrees satisfies the hypothesis $t(i_1)<\tau_c$ in that lemma for any sufficiently large $n$ since $\epsilon<\frac16$) we get
\begin{equation}
  \label{eq-Q-nearly-E/3}
  Q(i_1) \leq \tfrac16 n^{3/2} (\log n)^{-3} + \tfrac16 n^{3/2} (\log n)^{-1} + n^{4/3} = (\tfrac13 + o(1))|E(i_1)|\,.
\end{equation}
Clearly, if $\{ uv\in E(i_1) : Y_{u,v}=0\} \geq \delta |E(i_1)|$ for some arbitrarily small fixed $\delta>0$, that is, a positive fraction of the edges have no triangles incident to them, then at least $\delta |E(i_1)| = n^{3/2-o(1)}$ edges will survive the removal process and we can conclude the proof. Assume therefore that there are at most $o(|E(i_1)|)$ such edges. In particular,
\[ Q(i_1) = \tfrac13 \sum_{uv\in E(i_1)} Y_{u,v} \geq \tfrac13 \left|\left\{ uv \in E(i_1) : Y_{u,v}\neq 0\right\}\right| \geq \left(\tfrac13-o(1)\right)\left|E(i_1)\right|\,.\]
Furthermore, together with~\eqref{eq-Q-nearly-E/3}, this implies that almost all triangles are edge disjoint. A simple way to see this is to order the $Q(i_1)$ triangles arbitrarily and then sequentially add their edges to the initial edge set $\{uv\in E(i_1): Y_{u,v}=0\}$, eventually arriving at the edge set of $G(i_1)$. If for some fixed $\delta>0$ at least $\delta Q(i_1)$ triangles share any edges with the triangles preceding them in this procedure then the total number of edges would be at most $(3-\delta)Q(i_1) + o(|E(i_1)|)$, which by~\eqref{eq-Q-nearly-E/3} is at most
$(1-\frac\delta3+o(1)) |E(i_1)|$, contradiction.

Examining the above statement we see that one can exclude $o(|E(i_1)|)$ triangles and be left with a set $\cA$ of $(\frac13-o(1))|E(i_1)|$ edge disjoint triangles, i.e., $B_{uvw}=0$ for every $uvw\in\cA$. Re-adding the excluded triangles may now increase the value of $B_{uvw}$ for at most $o(|E(i_1)|)$ distinct triangles $uvw\in \cA$ (an excluded triangle impacts at most $3$ triangles in $\cA$). Altogether,
\begin{align*}
  \#\left\{ uvw \in Q(i_1) : B_{uvw} = 0\right\} = (\tfrac13 -o(1))|E(i_1)| = n^{3/2-o(1)}\,,
\end{align*}
which when added to~\eqref{eq-X-lower-bound} now implies that
\begin{align}
  \label{eq-X-lower-bound-2}
 |\sX_\star| \geq |\sX(i_1)| \geq n^{3/2-4\epsilon-o(1)}\,.
\end{align}

We now focus our attention on which edges among the triangles of $\sX_\star$ belongs to the final outcome of the process.
To this end, it will be useful to consider an equivalent formulation of the process starting from time $i_0$. Instead of drawing a uniform triangle in every step, consider a uniform permutation $\sigma$ over the triangles $Q(i_0)$, and the deterministic process that examines these triangles sequentially and each time removes the 3 edges of the triangle under consideration iff it is still intact (i.e., iff all of its edges still belong to the current graph). Clearly this process is merely a time-rescaled version of the usual triangle removal process, hence it suffices to analyze this one for establishing some properties of the final output.

For each $uvw\in \sX$ define an event $\cE_{uvw}$ which will depend only on the internal ordering of the triangles $\cN_2(uvw)$ in $\sigma$. The formulation of the event will depend on the constellation of triangles incident to $uvw$ upon its addition to $\sX_\star$. Let $j$ denote the round at the end of which $uvw$ was added to $\sX_\star$. By definition, at the end of this round there was some edge of $uvw$ that had no other triangles incident to it, that is, w.l.o.g.\ we had $Y_{u,v}=1$. Choose some arbitrary $xyz \in Q(j)$ which at the end of round $j$ satisfied $xyz \in \cN_1(uvw)$, while recalling that by construction $B_{uvw}(j) > 0$ thus such a triangle necessarily exists. Define $\cE_{uvw}$ to be the event that $\sigma$ arranges $xyz$ prior to all triangles $\cN_1(xyz)$ (including $uvw$).

First, recall that for any $xyz$ as above $\cN_1(xyz)\leq n^{2\epsilon+o(1)}$ by the co-degree hypothesis, thus
$ \P(\cE_{uvw}) \geq n^{-2\epsilon-o(1)}$.
Second, observe that on the event $\cE_{uvw}$ the triangle $xyz$ under consideration will be examined for possible removal prior to examining any other $x'y'z'\sim xyz$ from $Q(j)$, thus the process will delete the edges $\{xy,xz,yz\}$ at that time. Only later will $uvw \in \cN_1(xyz)$ be processed, and at that time no action will be taken due to its missing edge (the edge in common with $xyz$). In particular, the edge $uv$ with $Y_{u,v}(j)=1$, whose only incident triangle at time $j$ was $uvw$, will belong to the final output.

Since the triangles in $\sX_\star$ are edge disjoint, the final graph contains at least $\sum \one_{\cE_{uvw}}$ edges.
Moreover, the occurrence of $\cE_{uvw}$ is purely a function of the internal ordering according to $\sigma$ between $xyz\in \cN_1(uvw)$ and $\cN_1(xyz) \subseteq \cN_2(uvw)$. By construction, we exclude all triangles in $\cN_2(uvw)$ from possible future inclusion in $\sX_\star$ (by adding them to $\sY$), thus $\sum \one_{\cE_{uvw}}$ stochastically dominates a binomial random variable with parameters $\Bin(|\sX_\star|,n^{-2\epsilon-o(1)})$. By~\eqref{eq-X-lower-bound-2} and standard estimates for the binomial distribution this is at least $n^{3/2-6\epsilon-o(1)}$ w.h.p., as required.
\end{proof}

\section*{Acknowledgments}
We thank Noga Alon, Jeff Kahn, Peter Keevash and Joel Spencer for fruitful discussions.


\begin{bibdiv}
\begin{biblist}


\bib{AKS}{article}{
   author={Alon, Noga},
   author={Kim, Jeong-Han},
   author={Spencer, Joel},
   title={Nearly perfect matchings in regular simple hypergraphs},
   journal={Israel J. Math.},
   volume={100},
   date={1997},
   pages={171--187},
}


\bib{BB}{article}{
  author={Bennett, Patrick},
  author={Bohman, Tom},
  title={A natural barrier in random greedy hypergraph matching},
  note={Manuscript},
}

\bib{Bohman}{article}{
   author={Bohman, Tom},
   title={The triangle-free process},
   journal={Adv. Math.},
   volume={221},
   date={2009},
   number={5},
   pages={1653--1677},
}

\bib{BFL}{article}{
   author={Bohman, Tom},
   author={Frieze, Alan},
   author={Lubetzky, Eyal},
   title={A note on the random greedy triangle-packing algorithm},
   journal={J. Comb.},
   volume={1},
   date={2010},
   number={3-4},
   pages={477--488},
}

\bib{BK}{article}{
   author={Bohman, Tom},
   author={Keevash, Peter},
   title={The early evolution of the $H$-free process},
   journal={Invent. Math.},
   volume={181},
   date={2010},
   number={2},
   pages={291--336},
}


\bib{Bol1}{article}{
   title={The Life and Work of Paul Erd\H{o}s},
   author={Bollob{\'a}s, B{\'e}la},
   book={
      title={Wolf Prize in mathematics. Vol. 1},
      editor={Chern, S. S.},
      editor={Hirzebruch, F.},
      publisher={World Scientific Publishing Co. Inc.},
      place={River Edge, NJ},
      date={2000},
   },
   pages={292--315},
}

\bib{Bol2}{article}{
   author={Bollob{\'a}s, B{\'e}la},
   title={To prove and conjecture: Paul Erd\H os and his mathematics},
   journal={Amer. Math. Monthly},
   volume={105},
   date={1998},
   number={3},
   pages={209--237},
}


\bib{ESW}{article}{
   author={Erd{\H{o}}s, Paul},
   author={Suen, Stephen},
   author={Winkler, Peter},
   title={On the size of a random maximal graph},
   journal={Random Structures Algorithms},
   volume={6},
   date={1995},
   number={2-3},
   pages={309--318},
}

\bib{Freedman}{article}{
   author={Freedman, David A.},
   title={On tail probabilities for martingales},
   journal={Ann. Probability},
   volume={3},
   date={1975},
   pages={100--118},
}

\bib{GKPS}{article}{
   author={Gordon, Daniel M.},
   author={Kuperberg, Greg},
   author={Patashnik, Oren},
   author={Spencer, Joel H.},
   title={Asymptotically optimal covering designs},
   journal={J. Combin. Theory Ser. A},
   volume={75},
   date={1996},
   number={2},
   pages={270--280},
}

\bib{Grable}{article}{
   author={Grable, David A.},
   title={On random greedy triangle packing},
   journal={Electron. J. Combin.},
   volume={4},
   date={1997},
   pages={Research Paper 11, 19 pp.},
}



\bib{Rodl}{article}{
   author={R{\"o}dl, Vojt{\v{e}}ch},
   title={On a packing and covering problem},
   journal={European J. Combin.},
   volume={6},
   date={1985},
   number={1},
   pages={69--78},
}

\bib{RT}{article}{
   author={R{\"o}dl, Vojt{\v{e}}ch},
   author={Thoma, Lubo{\v{s}}},
   title={Asymptotic packing and the random greedy algorithm},
   journal={Random Structures Algorithms},
   volume={8},
   date={1996},
   number={3},
   pages={161--177},
}

\bib{Spencer}{article}{
   author={Spencer, Joel},
   title={Asymptotic packing via a branching process},
   journal={Random Structures Algorithms},
   volume={7},
   date={1995},
   number={2},
   pages={167--172},
}



\bib{Wormald}{article}{
  author={Wormald, N.C.},
  title={The differential equation method for random graph processes and greedy algorithms},
  pages={73--155},
  date={1999},
   book={
      title={Lectures on approximation and randomized algorithms},
      series={Advanced Topics in Mathematics},
      editor={M.\ Karonski},
      editor={H.J.\ Pr\"{o}mel},
      publisher={Polish Scientific Publishers PWN},
      place={Warsaw},
   },
}

\end{biblist}
\end{bibdiv}
\end{document}